\documentclass[11pt,reqno]{amsart}

\usepackage[T1]{fontenc}
\usepackage[utf8]{inputenc}
\usepackage{lmodern}
\usepackage{microtype}
\usepackage{amsmath,amssymb,amsthm,mathtools}
\usepackage{mathrsfs,bm}
\usepackage{enumitem,booktabs,array,longtable}
\usepackage{xcolor,graphicx}
\graphicspath{{figures/}}
\usepackage[a4paper,margin=1.08in]{geometry}
\usepackage{xurl}
\usepackage{hyperref}
\usepackage[nameinlink,capitalize,noabbrev]{cleveref}

\hypersetup{
 colorlinks=true,
 linkcolor=blue!55!black,
 citecolor=green!40!black,
 urlcolor=blue!60!black,
 breaklinks=true,
 pdftitle={Biharmonic Conformal Immersions of Hypersurfaces: Dirac Rigidity and Mixed Clifford Geometry},
 pdfauthor={Dipesh Bhandari}
}

\newtheorem{theorem}{Theorem}[section]
\newtheorem{proposition}[theorem]{Proposition}
\newtheorem{lemma}[theorem]{Lemma}
\newtheorem{corollary}[theorem]{Corollary}
\newtheorem{definition}[theorem]{Definition}
\newtheorem{remark}[theorem]{Remark}

\newcommand{\R}{\mathbb R}
\newcommand{\C}{\mathbb C}
\newcommand{\Sph}{\mathbb S}
\newcommand{\Hyp}{\mathbb H}
\newcommand{\grad}{\operatorname{grad}}
\newcommand{\diver}{\operatorname{div}}
\newcommand{\Ric}{\operatorname{Ric}}
\newcommand{\tr}{\operatorname{tr}}
\newcommand{\Id}{\operatorname{Id}}
\newcommand{\dd}{\,\mathrm d}

\newcommand{\Bop}{\mathscr B_c}

\newcommand{\vol}{\omega}
\newcommand{\cl}{\mathbin{\cdot}}
\newcommand{\Disc}{\mathfrak D}
\newcommand{\Rot}{\operatorname{Rot}}
\newcommand{\doi}[1]{\href{https://doi.org/\detokenize{#1}}{\nolinkurl{https://doi.org/#1}}}
\newcommand{\arxiv}[1]{\href{https://arxiv.org/abs/\detokenize{#1}}{\nolinkurl{arXiv:#1}}}

\newcommand{\Qop}{\mathscr{Q}_{\eta}}

\newcommand{\doilink}[1]{\doi{#1}}
\newcommand{\arxivlink}[1]{\href{https://arxiv.org/abs/#1}{\nolinkurl{arXiv:#1}}}
\newcommand{\introheading}[1]{%
  \par\addvspace{1.1\baselineskip}%
  \noindent\textbf{#1.}\par
  \nobreak\smallskip
}

\newcommand{\safeincludegraphics}[2][]{\includegraphics[#1]{#2}}

\title[Biharmonic conformal immersions]{Biharmonic Conformal Immersions of Hypersurfaces:\ Dirac Rigidity and Mixed Clifford Geometry}

\author{Dipesh Bhandari}
\address{Department of Physics, Southern Methodist University, Dallas, Texas, USA}
\email{dbhandari@smu.edu}
\date{September 7, 2026}


\begin{document}

\begin{abstract}
A biharmonic conformal immersion asks whether a hypersurface can become a critical point of the bienergy after a conformal change of the domain metric.  We study when this flexibility can produce compact nonconstant-mean-curvature geometry in space forms.  For surfaces in three-dimensional space forms, with $U=\lambda^2H$, Ou's system is encoded exactly by a Laplace-type operator acting on a restricted ambient Killing spinor.  We classify all monic scalar--chiral factorizations of this operator.  In nonzero curvature such a factorization exists locally exactly when the principal curvatures are locally constant; the normalized Euclidean branch is equally rigid, and the remaining real spherical non-CMC branch is excluded by a finite-type Frobenius obstruction.

For $m>2$, this scalar rigidity does not imply geometric nonexistence.  The conformal tension has independent normal and tangential components.  We prove that $\mathbb R\xi\oplus TM$ is the minimal natural $O(m)$-equivariant carrier, derive the corresponding strongly elliptic Jacobi block and its two-grade Clifford realization, and identify off-diagonal hybridization as the escape mechanism.  Exact cohomogeneity-one reductions, combined with validated interval and Chebyshev/radii-polynomial arguments, yield embedded proper non-CMC, nowhere-umbilical examples in round spheres, including
\[
S^1\times\mathbb S^n\times\mathbb S^n\hookrightarrow\mathbb S^{2n+2}
\qquad(n\ge1).
\]
Thus the scalar obstruction is sharp with respect to coefficient bundle: the smallest enlargement forced by conformal tension already supports compact non-CMC solutions.
\end{abstract}

\maketitle

\section{Introduction}\label{sec:introduction}

\introheading{Biharmonic maps and the conformal immersion problem}
Harmonic maps are critical points of the Dirichlet energy.  Biharmonic maps, introduced variationally by Jiang \cite{Jiang}, are critical points of the bienergy and satisfy a fourth-order Euler--Lagrange equation.  For an isometric hypersurface immersion, the tension field is the mean-curvature vector; minimal hypersurfaces are therefore harmonic, while a nonminimal biharmonic immersion is called \emph{proper biharmonic}.  In round spheres, small hyperspheres and generalized Clifford products provide the basic proper examples, and much of the subsequent theory has been shaped by classification and rigidity questions, especially under constant-mean-curvature or principal-curvature hypotheses \cite{CaddeoS3,CaddeoSpheres,BalmusMontaldoOniciuc,FetcuOniciucSurvey,GuanLiVrancken,AndronicFuOniciuc}.

There is a different way to introduce flexibility.  Instead of changing the image hypersurface, keep the immersion and change the metric on its domain conformally.  We consider
\begin{equation}\label{eq:intro-conformal-immersion}
\phi:(M^m,\bar g)\longrightarrow N^{m+1}(c),
\qquad
\phi^*h=g=\lambda^2\bar g,
\qquad \lambda>0,
\end{equation}
where $N^{m+1}(c)$ is a space form, and ask when $\phi:(M,\bar g)\to N$ is biharmonic.  We refer to such a map as a \emph{biharmonic conformal immersion} (BCI).  This problem was developed for surfaces by Ou and collaborators and recently in general hypersurface dimension by Cherif--Ou \cite{OuConformal,OuThree,OuF,OuFII,CherifOu}.  It asks which of the familiar rigidity phenomena for biharmonic hypersurfaces survive when the domain metric is allowed to vary within its conformal class.  The conformal factor is a new degree of freedom, but it is not free: it is coupled to the shape operator and the mean curvature through both the normal and tangential parts of the bitension field.

This viewpoint also distinguishes the present problem from the recently studied \emph{conformal-biharmonic} or $c$-biharmonic theory, whose Euler--Lagrange equation comes from a modified conformally invariant bienergy functional \cite{BrandingMontaldoNistorOniciucRatto}.  Here the functional is the ordinary bienergy; it is the \emph{domain metric of the immersion} that is changed conformally.

Recent results already show that conformal reparametrization can alter the usual picture.  Wang--Chen construct biharmonic conformal immersions in three-dimensional conformally flat spaces, while Cherif--Ou derive the general hypersurface system and obtain, among other results, a proper totally umbilical example in $\Sph^5$ \cite{WangChenConformalFlat,CherifOu}.  The question pursued here is deliberately more stringent:
\begin{quote}
Can a compact hypersurface in a round sphere admit a positive nonconstant conformal factor that makes the immersion proper biharmonic with nonconstant mean curvature?  If so, what geometric mechanism allows such a solution to escape the scalar rigidity visible in the surface equations?
\end{quote}
A compact non-CMC, nowhere-umbilical example is a useful test because it cannot be explained by the standard constant-mean-curvature or totally umbilical mechanisms.  The purpose of this paper is both to identify the obstruction encountered by the most economical scalar search and to show that the first enlargement forced by the conformal geometry already overcomes it.

\introheading{Why a Dirac formulation enters}
For a surface, put $U=\lambda^2H$.  Ou's equations become
\begin{equation}\label{eq:intro-Ou-unified}
\Delta U-U(|A|^2-2c)=0,
\qquad
A(\grad U)+U\grad H=0.
\end{equation}
The first equation is scalar and Schr\"odinger-like; the second ties the direction of variation of $U$ to the principal geometry.  Spinorial representations of immersed surfaces provide the natural background for this step \cite{Friedrich,Morel,BayardLawnRoth}.  A restricted ambient Killing spinor records the shape operator through the spinorial Gauss formula, so these two equations can be combined into one Laplace-type Dirac operator.  We use this as a search device: first determine exactly how far the smallest scalar--chiral factorization class can go, and only then enlarge the coefficient bundle if the geometry requires it.

That first search closes completely.  The factorization equations produce a shifted chiral field $s$ satisfying
\begin{equation}\label{eq:intro-discriminant-unified}
\grad s=J\grad H,
\qquad
4s^2=9c-4|A|^2.
\end{equation}
The resulting discriminant excludes the real non-CMC Euclidean and hyperbolic branches, and a finite-type Frobenius obstruction eliminates the remaining spherical branch.  Thus the scalar factorization fails for a structural reason rather than for lack of examples.

The higher-dimensional conformal tension then shows what the scalar description omitted.  If $\rho=\log\lambda$, then
\begin{equation}\label{eq:intro-tension-split}
\tau_{\bar g}(\phi)
=
\lambda^2\bigl(mH\xi-(m-2)\grad\rho\bigr).
\end{equation}
For $m>2$ the unknown therefore has independent normal and tangential components.  The natural bundle is $\mathbb R\xi\oplus TM$, and the conformal biharmonic equation becomes a coupled normal--tangential Jacobi system.  Its off-diagonal terms allow a zero mode of the full system even when neither diagonal channel vanishes separately.  We call this \emph{normal--tangential hybridization}.  The Clifford realization is a faithful representation of this block, not a modification of the biharmonic equation.

\introheading{Main results}
The obstruction and the constructions are summarized in the following theorem.

\begin{theorem}[Main theorem]\label{thm:combined-main}
Let $\phi$ be an oriented conformal hypersurface immersion into a Riemannian space form.
\begin{enumerate}[label=\textup{(\roman*)},leftmargin=2.4em]
\item For $m=2$, after restriction of an ambient Killing spinor there is a natural Laplace-type operator $\mathscr B_c$ such that $\mathscr B_c(U\psi)=0$ is exactly equivalent to \eqref{eq:intro-Ou-unified}.
\item Every monic scalar--chiral factorization
\[
\mathscr B_c=(D+a+b\omega)(D+p+q\omega)
\]
is explicitly classifiable.  In nonzero curvature such a factorization exists locally if and only if the principal curvatures are locally constant.  The normalized Euclidean branch has the same rigidity; the remaining Euclidean factors form the exceptional holomorphic--antiholomorphic family described below.
\item In dimension $m>2$, the smallest natural $O(m)$-equivariant coefficient bundle capable of carrying arbitrary conformal tensions is $\mathbb R\xi\oplus TM$.  On this bundle the conformal biharmonic equation is equivalent to a strongly elliptic normal--tangential Jacobi block with an exact Clifford realization.
\item The mixed block possesses compact zero modes with both normal and tangential components nonzero.  In particular, for every integer $n\ge1$ there is a smooth positive nonconstant conformal factor for which an embedded, fixed-sign, non-CMC, nowhere-umbilical hypersurface
\[
S^1\times\mathbb S^n\times\mathbb S^n\hookrightarrow\mathbb S^{2n+2}
\]
is a proper biharmonic conformal immersion.  Additional certified examples occur with topologies $S^1\times\mathbb S^2$, $T^2\times\mathbb S^2$, $S^1\times\mathbb S^3$, and $S^1\times\mathbb S^2\times\mathbb S^3$ in the corresponding round spheres.
\end{enumerate}
\end{theorem}

The theorem answers the motivating question in two steps.  Parts (i)--(ii) show that the smallest Dirac factorization class is genuinely rigid: its non-CMC branch cannot be rescued by a better choice of scalar coefficient.  Parts (iii)--(iv) show that this is not a nonexistence theorem for biharmonic conformal immersions.  The conformal tension itself supplies the missing tangential component, and the coupling with the normal component produces compact non-CMC solutions.  The all-$n$ family is important in this respect: the construction is not tied to one exceptional dimension or to one isolated shooting orbit.

The operator formulation also has a practical role in the search.  It separates three questions that are easy to conflate: whether the geometric operator factors, whether an admissible zero mode exists, and whether that zero mode reconstructs a positive conformal factor.  The surface examples in \Cref{sec:surface-diagnostics} diagnose these distinctions.  The higher-dimensional block then identifies the coupled modes to target in the cohomogeneity-one reductions.

\introheading{Proof strategy and guide to the paper}
The analytic and computer-assisted parts have separate roles.  All geometric reductions, reversibility identities, and substitutions into the Ou--Cherif equations are exact.  Floating-point shooting is used only to locate candidate closed profiles.  Existence and local uniqueness of the required returns are then proved by outward-rounded interval Taylor integration and Krawczyk inclusions; positivity, embeddedness, and curvature separation are verified by interval sweeps.  For the infinite equal-multiplicity tail, a Chebyshev representation and a radii-polynomial estimate replace infinitely many separate shooting calculations.  The resulting reconstructed immersions are finally checked directly against the original conformal biharmonic equation.

Readers primarily interested in the geometry can read \Cref{sec:surface-bridge,sec:factorization-wall,sec:higher-dimension,sec:compact-main-theorem,sec:synthesis} first.  The intervening sections explain the scalar no-go mechanisms and develop the compact constructions in increasing geometric complexity.  The appendices record the finite Bernstein certificates, moving-frame and reduced-system derivations, and the reproducibility information for Online Resource~1.

\section{The surface biharmonic conformal system and its Dirac bridge}\label{sec:surface-bridge}\label{sec:Ou}

This section supplies the map-theoretic background needed for the rest of the paper.  It also explains why the single field $U=\lambda^2H$ is the correct unknown for a conformal surface immersion.

Let $\phi:(M,g)\to(N,h)$ be a smooth map.  Its Dirichlet energy is
\begin{equation}\label{eq:energy}
E(\phi)=\frac12\int_M |\dd\phi|^2\,\dd v_g,
\end{equation}
and its Euler--Lagrange field is the tension field
\begin{equation}\label{eq:tension}
\tau(\phi)=\operatorname{tr}_g\nabla\dd\phi.
\end{equation}
The map is harmonic when $\tau(\phi)=0$.  If $\phi:M^2\to N^3$ is an isometric immersion with unit normal $\xi$, then
\begin{equation}\label{eq:tension-immersion}
\tau(\phi)=2H\xi.
\end{equation}
Thus harmonic surface immersions are precisely minimal surfaces.

The bienergy is
\begin{equation}\label{eq:bienergy}
E_2(\phi)=\frac12\int_M |\tau(\phi)|^2\,\dd v_g.
\end{equation}
Its critical points are biharmonic maps.  Every harmonic map is biharmonic, while a biharmonic map with nonzero tension is called proper biharmonic.  For immersions, the bitension field splits into normal and tangential parts.  This splitting is the source of both the analytic difficulty and the geometric content of the theory: the normal equation controls a curvature potential, whereas the tangential equation couples the mean-curvature gradient to the shape operator.

Now suppose
\begin{equation}\label{eq:conformal-immersion}
\phi:(M^2,\bar g)\longrightarrow N^3(c),
\qquad
\phi^*h=g=\lambda^2\bar g,
\end{equation}
with $\lambda>0$.  We regard the same map as an isometric immersion from $(M,g)$ and set
\begin{equation}\label{eq:U}
U=\lambda^2H.
\end{equation}
In dimension two, biharmonicity of the conformal immersion is equivalent to the $f$-biharmonic equation for the induced isometric immersion with weight $f=\lambda^2$; see \cite{OuConformal,OuThree,OuF,OuFII}.  This is why the conformal factor and the mean curvature occur through their product $U$.

\begin{theorem}[Ou]\label{thm:Ou}
The conformal immersion \eqref{eq:conformal-immersion} is biharmonic if and only if
\begin{align}
\Delta U-U(|A|^2-2c)&=0,\label{eq:Ou-normal}\\
A(\grad U)+U\grad H&=0.\label{eq:Ou-tangent}
\end{align}
All differential operators are computed with the induced metric $g$.
\end{theorem}

For completeness, we recall the geometric origin of the two equations.  Let $(e_1,e_2,\xi)$ be an orthonormal frame along the immersion.  The Gauss--Weingarten equations are
\begin{equation}\label{eq:GW}
\nabla^N_XY=\nabla_XY+g(A(X),Y)\xi,
\qquad
\nabla^N_X\xi=-A(X).
\end{equation}
For a weighted isometric immersion, splitting the weighted bitension field into its normal and tangential components gives
\begin{align}
\Delta(fH)-(fH)(|A|^2-\Ric^N(\xi,\xi))&=0,\label{eq:f-normal-general}\\
A(\grad(fH))+(fH)\bigl(\grad H-(\Ric^N(\xi))^\top\bigr)&=0.\label{eq:f-tangent-general}
\end{align}
In a three-dimensional space form,
\begin{equation}\label{eq:Ric-space-form}
\Ric^N(\xi,\xi)=2c,
\qquad
(\Ric^N(\xi))^\top=0.
\end{equation}
Taking $f=\lambda^2$ yields \eqref{eq:Ou-normal}--\eqref{eq:Ou-tangent}.

The system has two immediate geometric readings.  If $H$ is constant, then the tangential equation becomes $A(\grad U)=0$: nonconstant weighted modes can propagate only along a zero-principal-curvature direction.  If $A$ is invertible, the same equation strongly constrains $U$, and in the isoparametric case forces it to be constant.  These two alternatives reappear later as the cylinder and small-sphere mechanisms.  Our aim is to place both equations inside one spinorial operator without erasing this distinction.

\subsection{Geometric and spinorial conventions}\label{sec:conventions}

The proofs use a fixed collection of sign and Clifford conventions.  They are recorded here so that the later operator identities can be checked without ambiguity.  Readers interested first in the geometric mechanism may proceed to \Cref{sec:restricted} and return to this section as needed.

Let $N^3(c)$ be an oriented Riemannian three-dimensional space form of constant sectional curvature $c$.  We write
\[
\phi:M^2\longrightarrow N^3(c)
\]
for an oriented isometric immersion with induced metric $g$.

Unless a global hypothesis is stated explicitly, every spinorial identity below is local.  More precisely, on a simply connected chart $\Omega\subset M$ the immersion lifts to the simply connected model $\widetilde N^3(c)$, a Killing spinor on $\widetilde N^3(c)$ is restricted to $\Omega$, and all operator identities are formed on the induced local spinor bundle.  If $N^3(c)$ itself is the simply connected model, these spinors are global.  For a quotient space form, global descent depends on the chosen spin structure and the holonomy action on the Killing-spinor space; no such descent is assumed.  The local classification statements therefore apply to arbitrary space forms, while statements using a globally defined restricted spinor require this additional global spinorial hypothesis.  Since every oriented surface is spin, the intrinsic factorization identities themselves may also be considered globally after choosing a spin structure on $M$ and a global factorization field.  In particular, operator-factorization statements are intrinsic once these data are fixed, whereas identities involving a specific global section $U\psi$ additionally require descent of the restricted ambient Killing spinor.

Let $\xi$ be the chosen unit normal, and define the shape operator by
\begin{equation}\label{eq:Weingarten}
\nabla^N_X\xi=-A(X).
\end{equation}
The normalized mean curvature is
\begin{equation}\label{eq:H}
H=\frac12\tr A.
\end{equation}
Our Laplace--Beltrami operator is
\[
\Delta f=\diver(\grad f).
\]
Thus at a point where an orthonormal frame $(e_1,e_2)$ is geodesic,
\[
\Delta f=e_1e_1(f)+e_2e_2(f).
\]

Every oriented surface is spin.  We use the spin structure induced from the ambient spin structure and write $\Sigma M$ for the complex spinor bundle.  Clifford multiplication is denoted by a centered dot and satisfies
\begin{equation}\label{eq:Clifford}
X\cl Y+Y\cl X=-2g(X,Y).
\end{equation}
Let
\begin{equation}\label{eq:omega}
\vol=e_1\cl e_2
\end{equation}
be the real volume element for a positively oriented local orthonormal frame.  Then
\begin{equation}\label{eq:omega-identities}
\vol^2=-1,
\qquad
X\cl\vol=-\vol\cl X,
\qquad
D(\vol\cl\Psi)=-\vol\cl D\Psi.
\end{equation}
We choose the complex spinor inner product so that Clifford multiplication by a tangent vector is skew-Hermitian.

The intrinsic Dirac operator is
\begin{equation}\label{eq:Dirac}
D\Psi=e_1\cl\nabla_{e_1}\Psi+e_2\cl\nabla_{e_2}\Psi.
\end{equation}
For a scalar function $f$,
\begin{align}
D(f\Psi)&=\grad f\cl\Psi+fD\Psi,\label{eq:D-product}\\
D^2(f\Psi)&=-(\Delta f)\Psi-2\nabla_{\grad f}\Psi+fD^2\Psi.\label{eq:D2-product}
\end{align}

Let $J$ denote rotation by $+\pi/2$ on $TM$, so $Je_1=e_2$.  Our conventions imply
\begin{equation}\label{eq:Jomega}
X\cl\vol=-JX\cl,
\qquad
\vol\cl X=JX\cl.
\end{equation}

The simply connected model $\widetilde N^3(c)$ carries Killing spinors with Killing number $\eta\in\C$ satisfying
\begin{equation}\label{eq:eta-c}
4\eta^2=c.
\end{equation}
For a sphere and hyperbolic space of curvature radius $L$ we may choose
\begin{equation}\label{eq:eta-models}
\eta=\frac{1}{2L}
\quad\text{on }\Sph^3(L),
\qquad
\eta=\frac{i}{2L}
\quad\text{on }\Hyp^3(L).
\end{equation}
Changing the sign of the ambient Killing spinor number changes the corresponding chiral signs throughout, without changing the geometric content.

\subsection{Restricted Killing spinors}\label{sec:restricted}

The geometric input behind the Dirac bridge is a spinor already supplied by the immersion.  On a simply connected chart, the surface lifts to the simply connected ambient model, where a parallel, real Killing, or imaginary Killing spinor can be restricted to the surface.  Its covariant derivative records the shape operator, so one spinorial identity contains both intrinsic differentiation and extrinsic curvature.

Let $\Psi$ be a Killing spinor on the lifted ambient model,
\begin{equation}\label{eq:ambient-Killing}
\nabla^N_Y\Psi=\eta\,Y\cl_N\Psi.
\end{equation}
Under the standard hypersurface identification $\Sigma N|_M\cong\Sigma M$, let $\psi$ denote its restriction.  Morel's spinorial Gauss formula gives the following identity \cite{Morel}.

\begin{proposition}[Restricted Killing-spinor equation]\label{prop:restricted}
For every $X\in TM$,
\begin{equation}\label{eq:restricted-Killing}
\nabla_X\psi
=-\frac12A(X)\cl\psi-\eta X\cl\vol\cl\psi.
\end{equation}
Consequently,
\begin{equation}\label{eq:Dpsi}
D\psi=H\psi+2\eta\vol\cl\psi.
\end{equation}
The spinor $\psi$ is nowhere zero.
\end{proposition}

\begin{proof}
The spinorial Gauss formula, together with the identification of ambient and intrinsic Clifford multiplication, gives
\[
(\nabla^N_X\Psi)|_M
=\nabla_X\psi+\frac12A(X)\cl\psi.
\]
For the orientation convention adopted here, tangential ambient Clifford multiplication restricts as
\[
(X\cl_N\Psi)|_M=-X\cl\vol\cl\psi.
\]
Substitution into \eqref{eq:ambient-Killing} gives \eqref{eq:restricted-Killing}.  Tracing with $e_1\cl$ and $e_2\cl$ gives \eqref{eq:Dpsi}, since $\tr A=2H$.  A nonzero Killing spinor is parallel for a modified connection and hence has no zeros; therefore its restriction is nowhere zero.
\end{proof}

The next identity is the basic curved analogue of the Euclidean spinorial Weierstrass calculation.

\begin{lemma}\label{lem:D2psi}
The restricted Killing spinor satisfies
\begin{equation}\label{eq:D2psi}
D^2\psi=(H^2+c)\psi+\grad H\cl\psi.
\end{equation}
\end{lemma}

\begin{proof}
Using \eqref{eq:Dpsi}, the product rule, and $D(\vol\cl\psi)=-\vol\cl D\psi$, we compute
\begin{align*}
D^2\psi
&=D(H\psi)+2\eta D(\vol\cl\psi)\\
&=\grad H\cl\psi+H(H\psi+2\eta\vol\cl\psi)
   -2\eta\vol\cl(H\psi+2\eta\vol\cl\psi).
\end{align*}
The mixed terms cancel, while $\vol^2=-1$ and $4\eta^2=c$.  This gives \eqref{eq:D2psi}.
\end{proof}

\begin{lemma}\label{lem:Dgrad}
For every smooth real-valued function $u$,
\begin{equation}\label{eq:Dgrad}
D(\grad u\cl\psi)
=-(\Delta u)\psi+(A-H\Id)(\grad u)\cl\psi.
\end{equation}
In particular, the ambient Killing number $\eta$ cancels from this identity.
\end{lemma}

\begin{proof}
At a point $p$, choose a geodesic orthonormal frame diagonalizing $A$, with $A(e_j)=k_je_j$.  Expanding $D(\grad u\cl\psi)$ gives
\[
\sum_j e_j\cl(\nabla_{e_j}\grad u)\cl\psi
+\sum_j e_j\cl\grad u\cl\nabla_{e_j}\psi.
\]
The symmetric Hessian contributes $-(\Delta u)\psi$.  Inserting \eqref{eq:restricted-Killing}, the shape-operator contribution is
\[
-\frac12\sum_j e_j\cl\grad u\cl A(e_j)\cl\psi
=(A-H\Id)(\grad u)\cl\psi.
\]
The Killing-spinor contribution vanishes because, in dimension two,
\[
\sum_{j=1}^2e_j\cl X\cl e_j=0
\]
for every tangent vector $X$.
\end{proof}

\subsection{The exact spinorial Ou operator}\label{sec:main}

We now answer the first question from the introduction.  Applying $D^2$ to the weighted spinor $U\psi$ nearly reproduces Ou's equations, but ambient curvature creates an additional first-order chiral term.  The calculation below identifies that term and then removes it by a uniquely natural lower-order correction.

We begin by computing the square of the Dirac operator on $U\psi$.

\begin{theorem}[Curved spinorial square identity]\label{thm:Sc}
For every smooth real-valued function $U$,
\begin{equation}\label{eq:Sc}
\begin{split}
D^2(U\psi)
={}&\bigl[-\Delta U+(H^2+c)U\bigr]\psi\\
&+\bigl[A(\grad U)+U\grad H\bigr]\cl\psi\\
&+2\eta\,\grad U\cl\vol\cl\psi.
\end{split}
\end{equation}
\end{theorem}

\begin{proof}
Apply \eqref{eq:D2-product} and use \eqref{eq:restricted-Killing}:
\begin{align*}
D^2(U\psi)
&=-(\Delta U)\psi-2\nabla_{\grad U}\psi+UD^2\psi\\
&=-(\Delta U)\psi
  +A(\grad U)\cl\psi
  +2\eta\grad U\cl\vol\cl\psi\\
&\qquad+U\bigl[(H^2+c)\psi+\grad H\cl\psi\bigr].
\end{align*}
Collecting scalar, tangent-vector, and chiral terms proves the formula.
\end{proof}

\begin{remark}[The chiral drift]\label{rem:chiral}
By \eqref{eq:Jomega}, the final term in \eqref{eq:Sc} is
\[
2\eta\grad U\cl\vol\cl\psi
=-2\eta J\grad U\cl\psi.
\]
Thus the ambient curvature contributes not only the scalar shift $cU$ but also a first-order rotation of $\grad U$.  This is the obstruction to encoding Ou's system by $D^2$ plus a scalar potential alone.
\end{remark}

Define the second-order operator on $\Sigma M$ by
\begin{equation}\label{eq:B-def}
\boxed{
\Bop
:=D^2+2\eta\vol D-2\eta H\vol+|A|^2-H^2-2c.
}
\end{equation}
Here $H$, $|A|^2$, and $c$ act by scalar multiplication and $\vol$ acts by Clifford multiplication.

\begin{theorem}[Exact spinorial encoding]\label{thm:exact}
For every real-valued function $U$,
\begin{equation}\label{eq:B-action}
\boxed{
\Bop(U\psi)
=\bigl[-\Delta U+(|A|^2-2c)U\bigr]\psi
 +\bigl[A(\grad U)+U\grad H\bigr]\cl\psi.
}
\end{equation}
Consequently, for $U=\lambda^2H$, the conformal immersion \eqref{eq:conformal-immersion} is biharmonic if and only if
\begin{equation}\label{eq:B-zero}
\Bop(U\psi)=0.
\end{equation}
\end{theorem}

\begin{proof}
Using \eqref{eq:D-product} and \eqref{eq:Dpsi},
\begin{align*}
2\eta\vol D(U\psi)
={}&2\eta\vol\cl\grad U\cl\psi
 +2\eta HU\vol\cl\psi
 +4\eta^2U\vol^2\cl\psi\\
={}&-2\eta\grad U\cl\vol\cl\psi
 +2\eta HU\vol\cl\psi-cU\psi.
\end{align*}
The first term cancels the chiral drift in \eqref{eq:Sc}; the term $-2\eta H\vol$ cancels the remaining $HU\vol\cl\psi$ term.  The scalar contributions combine as
\[
(H^2+c)U-cU+(|A|^2-H^2-2c)U=(|A|^2-2c)U.
\]
This proves \eqref{eq:B-action}.

It remains to show that the scalar and vector terms vanish separately.  Here
both coefficients are real because $U,H,A$, and $c$ are real.  Suppose
\[
a\psi+X\cl\psi=0
\]
with $a\in\R$ and $X\in TM$.  Taking the real part of the Hermitian product with $\psi$ gives $a|\psi|^2=0$, because tangent Clifford multiplication is skew-Hermitian.  Since $\psi$ is nowhere zero, $a=0$.  Then $X\cl\psi=0$, and applying $X\cl$ once more gives $-|X|^2\psi=0$, hence $X=0$.  Applying this observation to \eqref{eq:B-action} and invoking \Cref{thm:Ou} proves the equivalence.
\end{proof}

\begin{corollary}[Euclidean reduction]\label{cor:Euclidean}
For $c=0$ and $\eta=0$,
\begin{equation}\label{eq:B-Euclidean}
\mathscr{B}_0=D^2+|A|^2-H^2,
\end{equation}
and
\[
\mathscr{B}_0(U\psi)
=\bigl[-\Delta U+|A|^2U\bigr]\psi
 +\bigl[A(\grad U)+U\grad H\bigr]\cl\psi.
\]
Thus the Euclidean biharmonic conformal surface equations are exactly a Dirac-square equation with scalar geometric potential.
\end{corollary}

\begin{remark}[Ellipticity]
The principal symbol of $\Bop$ is the principal symbol of $D^2$ and is therefore $|\zeta|^2\Id_{\Sigma M}$.  Hence $\Bop$ is a Laplace-type elliptic operator, although its lower-order chiral terms need not be self-adjoint for all choices of $\eta$ and conventions.
\end{remark}

\subsection{A first-order Dirac-type prolongation}\label{sec:first-order}

The second-order encoding is universal, whereas an undoubled first-order factorization will turn out to be exceptional.  Before imposing any factorization hypothesis, it is useful to record the first-order structure that is always available: a doubled elliptic system obtained by introducing one auxiliary spinor.

The exact operator has the following near-factorization.  Define
\begin{equation}\label{eq:Q-def}
\Qop:=D-H-2\eta\vol.
\end{equation}
By \eqref{eq:Dpsi},
\begin{equation}\label{eq:Q-action}
\Qop(U\psi)=\grad U\cl\psi.
\end{equation}

\begin{proposition}[Near-factorization]\label{prop:near-factorization}
As an operator on $\Sigma M$,
\begin{equation}\label{eq:near-factorization}
\boxed{
\Bop
=(D+H)(D-H-2\eta\vol)+\grad H\cl+|A|^2-2c.
}
\end{equation}
\end{proposition}

\begin{proof}
Using $D(H\Phi)=\grad H\cl\Phi+HD\Phi$ and $D(\vol\cl\Phi)=-\vol\cl D\Phi$, we obtain
\begin{align*}
(D+H)(D-H-2\eta\vol)
={}&D^2+2\eta\vol D-2\eta H\vol-H^2-\grad H\cl.
\end{align*}
Adding $\grad H\cl+|A|^2-2c$ gives \eqref{eq:B-def}.
\end{proof}

The residual term in \eqref{eq:near-factorization} is geometrically transparent: $\grad H\cl$ measures failure of the CMC condition, while $|A|^2-2c$ is precisely the normal potential in Ou's equation.

Set
\begin{equation}\label{eq:Phi-chi}
\Phi:=U\psi,
\qquad
\chi:=\Qop\Phi.
\end{equation}
Then \eqref{eq:B-zero} is equivalent, on the real line subbundle generated by $\psi$, to a first-order system.

\begin{theorem}[Canonical first-order prolongation]\label{thm:first-order}
Let $\Phi=U\psi$ with $U$ real.  The equation $\Bop\Phi=0$ is equivalent to the existence of a spinor $\chi$ such that
\begin{equation}\label{eq:first-order-system}
\boxed{
\begin{pmatrix}
D-H-2\eta\vol & -\Id\\[1mm]
\grad H\cl+|A|^2-2c & D+H
\end{pmatrix}
\begin{pmatrix}
\Phi\\ \chi
\end{pmatrix}
=0.
}
\end{equation}
The matrix operator in \eqref{eq:first-order-system} is elliptic of first order.
\end{theorem}

\begin{proof}
The first row defines $\chi=\Qop\Phi$.  The second row then becomes
\[
(D+H)\Qop\Phi+(\grad H\cl+|A|^2-2c)\Phi=\Bop\Phi
\]
by \eqref{eq:near-factorization}.  Conversely, eliminating $\chi$ from the system yields $\Bop\Phi=0$.

For a nonzero covector $\zeta$, the principal symbol is block diagonal:
\[
\sigma_1(\zeta)
=\begin{pmatrix}
\zeta^\sharp\cl&0\\0&\zeta^\sharp\cl
\end{pmatrix}.
\]
Since $(\zeta^\sharp\cl)^2=-|\zeta|^2\Id$, this symbol is invertible for $\zeta\neq0$.
\end{proof}

\begin{remark}
The line constraint $\Phi\in\R\psi$ is essential for the direct equivalence with a scalar conformal factor.  The unconstrained matrix system has a larger spinorial solution space.  Thus \eqref{eq:first-order-system} is best viewed as a Dirac-type prolongation of Ou's equation rather than an unconstrained replacement of the immersion problem.
\end{remark}

\section{The scalar--chiral factorization wall}\label{sec:factorization-wall}\label{sec:factorization}

We now turn from the universal encoding to the genuinely restrictive question.  We ask for an \emph{undoubled} factorization with the same Dirac principal symbol in both factors and with zeroth-order coefficients in the smallest natural Clifford algebra, the scalar--chiral span of $\Id$ and $\vol$.  The coefficient comparison is elementary, but its consequences are geometric: curvature fixes the scalar normalization, while the remaining coefficients satisfy a Cauchy--Riemann equation coupled to an algebraic Riccati relation.

\begin{definition}[Monic scalar--chiral factorization]\label{def:monic-factorization}
On an open set $\Omega\subset M$, a \emph{monic scalar--chiral factorization} of $\Bop$ is an identity
\begin{equation}\label{eq:general-factorization}
\Bop=(D+a+b\vol)(D+p+q\vol),
\end{equation}
where $a,b,p,q\in C^\infty(\Omega,\C)$.  The word monic means that both first-order factors have principal symbol equal to that of $D$.  This class does not include matrix-valued superpotentials or factors with a different Clifford-valued principal symbol.
\end{definition}

\begin{theorem}[Complete coefficient normal form]\label{thm:general-factorization}
The identity \eqref{eq:general-factorization} holds if and only if
\begin{align}
a&=-p,\label{eq:general-a}\\
b&=q+2\eta,\label{eq:general-b}\\
\grad p&=J\grad q,\label{eq:general-CR}\\
2\eta(p+H)&=0,\label{eq:general-H}\\
p^2+q(q+2\eta)&=H^2+2c-|A|^2.\label{eq:general-algebraic}
\end{align}
In particular, if $c\neq0$, every monic scalar--chiral factorization is forced to have $p=-H$.
\end{theorem}

\begin{proof}
For arbitrary scalar functions $a,b,p,q$, the product rule and
$D(\vol\Phi)=-\vol D\Phi$ give
\begin{align*}
(D+a+b\vol)(D+p+q\vol)
={}&D^2+(a+p)D+(b-q)\vol D\\
&+\grad p\cl+\grad q\cl\vol
 +(ap-bq)+(aq+bp)\vol.
\end{align*}
The coefficient comparison is legitimate on the full complex spinor bundle.  Indeed,
for every nonzero covector $\zeta$, Clifford multiplication by $\zeta^\sharp$
is invertible, so the principal symbols of $D$ and $\vol D$ are independent
because $\Id$ and $\vol$ are linearly independent.  At zeroth order,
$\{\Id,e_1\cl,e_2\cl,\vol\}$ is a complex basis of
$\operatorname{End}(\Sigma M)$ in an oriented orthonormal frame.
Matching the first-order terms with \eqref{eq:B-def} therefore gives
$a=-p$ and $b=q+2\eta$.  The vector term then vanishes precisely when
\[
\grad p\cl+\grad q\cl\vol
=(\grad p-J\grad q)\cl=0,
\]
which is \eqref{eq:general-CR}.  The remaining chiral coefficient is
$2\eta p$ and must equal $-2\eta H$, giving \eqref{eq:general-H}.
Finally, matching the scalar coefficient gives
\eqref{eq:general-algebraic}.  Conversely, these identities reproduce every
coefficient of $\Bop$.
\end{proof}

\begin{corollary}[Exhaustiveness in nonzero curvature]\label{cor:nonflat-exhaustive}
Assume $c\neq0$.  Then every monic scalar--chiral factorization is uniquely of the form
\begin{equation}\label{eq:factorization-ansatz}
\boxed{
\Bop
=\bigl(D+H+(\beta+2\eta)\vol\bigr)
 \bigl(D-H+\beta\vol\bigr)
}
\end{equation}
for a smooth $\beta:M\to\C$.  It exists if and only if
\begin{align}
\grad\beta&=J\grad H,\label{eq:beta-differential}\\
\beta(\beta+2\eta)+|A|^2-2c&=0.\label{eq:beta-algebraic}
\end{align}
\end{corollary}

\begin{proof}
Since $4\eta^2=c\neq0$, Equation \eqref{eq:general-H} gives $p=-H$.
Set $q=\beta$ in \Cref{thm:general-factorization}; then
\eqref{eq:general-a}--\eqref{eq:general-algebraic} become exactly
\eqref{eq:factorization-ansatz}--\eqref{eq:beta-algebraic}.
\end{proof}

\begin{proposition}[Exceptional Euclidean factors]\label{prop:flat-exceptional}
For $c=0$ and $\eta=0$, every monic scalar--chiral factorization has the form
\begin{equation}\label{eq:flat-general-factorization}
\boxed{
\mathscr B_0=(D-p+q\vol)(D+p+q\vol)
}
\end{equation}
with
\begin{equation}\label{eq:flat-factor-system}
\grad p=J\grad q,
\qquad
p^2+q^2=H^2-|A|^2.
\end{equation}
In a conformal coordinate $z=x+iy$ satisfying $J\partial_x=\partial_y$, the functions
\begin{equation}\label{eq:rt-flat}
r:=p-iq,
\qquad
t:=p+iq
\end{equation}
are respectively holomorphic and antiholomorphic and satisfy
\begin{equation}\label{eq:rt-product}
rt=H^2-|A|^2.
\end{equation}
On a simply connected region where $A\neq0$, such a factorization exists if and only if
\begin{equation}\label{eq:flat-log-harmonic}
\Delta\log(|A|^2-H^2)=0.
\end{equation}
On a connected planar region $A=0$, one of $r,t$ vanishes identically while the other is an arbitrary holomorphic or antiholomorphic function.
\end{proposition}

\begin{proof}
The first statement is \Cref{thm:general-factorization} with $\eta=c=0$.
Equation \eqref{eq:general-CR} is the Cauchy--Riemann system
$r_{\bar z}=0$ and $t_z=0$, while the algebraic equation gives
\eqref{eq:rt-product}.

Put $W=|A|^2-H^2$.  The inequality $H^2\leq |A|^2/2$ shows that
$W>0$ wherever $A\neq0$.  On such a simply connected region, $rt=-W$ is real and nonzero.  Since $r/\overline t$ is holomorphic and equals its complex conjugate, it is a real constant $\kappa$.  Hence
$r=\kappa\overline t$ with $\kappa<0$, and
$W=(-1/\kappa)|r|^2$; therefore \eqref{eq:flat-log-harmonic} holds.
Conversely, if \eqref{eq:flat-log-harmonic} holds, then locally
$W=|r|^2$ for a nowhere-zero holomorphic function $r$.  Taking
$t=-\overline r$ and recovering $p,q$ from \eqref{eq:rt-flat} gives
\eqref{eq:flat-factor-system}.  If $A=0$ on an open set, then
$rt=0$ identically.  On each connected component, either
$r\equiv0$ and $t$ is an arbitrary antiholomorphic function, or
$t\equiv0$ and $r$ is an arbitrary holomorphic function.
\end{proof}

\begin{definition}[Mean-curvature-normalized factorization]\label{def:normalized-factorization}
A monic scalar--chiral factorization is \emph{mean-curvature-normalized} if its right scalar coefficient is $p=-H$, equivalently if it has the form \eqref{eq:factorization-ansatz}.  By \Cref{cor:nonflat-exhaustive}, this normalization is automatic when $c\neq0$; in the Euclidean case it selects the geometric branch governed by the Dirac discriminant from the larger family in \Cref{prop:flat-exceptional}.
\end{definition}

\begin{theorem}[Normalized factorization criterion]\label{thm:factorization}
For arbitrary $c$, the mean-curvature-normalized identity \eqref{eq:factorization-ansatz} holds if and only if \eqref{eq:beta-differential}--\eqref{eq:beta-algebraic} hold.
\end{theorem}

\begin{proof}
This is the specialization $p=-H$ and $q=\beta$ of
\Cref{thm:general-factorization}.
\end{proof}

\begin{corollary}[Harmonicity obstruction]\label{cor:harmonic-H}
If the normalized factorization \eqref{eq:factorization-ansatz} holds on an open set, then
\begin{equation}\label{eq:H-harmonic}
\Delta H=0
\end{equation}
on that open set.  Conversely, on a simply connected open set where $H$ is harmonic, the differential equation \eqref{eq:beta-differential} has a local solution $\beta$, unique up to an additive constant; factorization then reduces to the algebraic constraint \eqref{eq:beta-algebraic}.
\end{corollary}

\begin{proof}
The one-form metrically dual to $J\grad H$ is, up to the fixed orientation sign, the Hodge dual of $\dd H$.  It is closed if and only if $\Delta H=0$.  Since \eqref{eq:beta-differential} asserts that this one-form is $\dd\beta$, harmonicity is necessary.  The local converse follows from the Poincar\'e lemma.
\end{proof}

Differentiating the algebraic condition yields an additional compatibility equation.

\begin{corollary}[Differential Riccati compatibility]\label{cor:Riccati-diff}
Under the hypotheses of \Cref{thm:factorization},
\begin{equation}\label{eq:dA2}
\dd |A|^2=-2(\beta+\eta)\dd\beta.
\end{equation}
Equivalently, after using \eqref{eq:beta-differential}, the gradients of $|A|^2$ and $H$ are related by a complex quarter-turn law.
\end{corollary}

\begin{remark}
The pair \eqref{eq:beta-differential}--\eqref{eq:beta-algebraic} is a geometric Riccati system.  It is substantially more restrictive than harmonicity of $H$ alone: a harmonic conjugate of $H$ must also lie pointwise on the quadratic spectral curve determined by $|A|^2$, $c$, and $\eta$.
\end{remark}

\begin{corollary}[CMC criterion]\label{cor:CMC}
Assume $M$ is connected and $H$ is constant.  A factorization of the form \eqref{eq:factorization-ansatz} exists if and only if $|A|^2$ is constant.  In that case $\beta$ is constant and satisfies
\begin{equation}\label{eq:beta-quadratic}
\beta^2+2\eta\beta+|A|^2-2c=0,
\end{equation}
so
\begin{equation}\label{eq:beta-roots}
\boxed{
\beta=-\eta\pm\frac12\sqrt{9c-4|A|^2}.
}
\end{equation}
\end{corollary}

\begin{proof}
When $H$ is constant, \eqref{eq:beta-differential} gives $\grad\beta=0$, so $\beta$ is constant.  Then \eqref{eq:beta-algebraic} forces $|A|^2$ to be constant.  The converse is immediate, and solving the quadratic gives \eqref{eq:beta-roots} because $4\eta^2=c$.
\end{proof}

\subsection{The Dirac discriminant and rigidity}\label{sec:discriminant}

The coefficient equations become geometrically transparent after one shift of the chiral parameter.  The resulting field is a harmonic conjugate of the mean curvature, while its square is determined by $|A|^2$ and the ambient curvature.  This converts factorization into a holomorphic-data problem and immediately rules out most curvature signs.

We begin with the shifted chiral mass.

\begin{definition}[Dirac discriminant]\label{def:discriminant}
Suppose that the normalized factorization \eqref{eq:factorization-ansatz} holds on a connected open set.  Define
\begin{equation}\label{eq:s-def}
s:=\beta+\eta
\end{equation}
and
\begin{equation}\label{eq:disc-def}
\boxed{\Disc:=9c-4|A|^2.}
\end{equation}
\end{definition}

\begin{theorem}[Dirac-discriminant structure]\label{thm:disc-structure}
On every factorizing chart,
\begin{equation}\label{eq:s-system}
\boxed{
\grad s=J\grad H,
\qquad
s^2=\frac{\Disc}{4}=\frac{9c}{4}-|A|^2.
}
\end{equation}
Consequently, $H$ and $s$ are harmonic.  If $H$ is nonconstant, then $s$ is real-valued, $c>0$, and
\begin{equation}\label{eq:real-chamber}
|A|^2\leq\frac{9c}{4}.
\end{equation}
In particular, there is no non-CMC mean-curvature-normalized factorization in $\R^3$ or $\Hyp^3$; by \Cref{cor:nonflat-exhaustive}, the hyperbolic statement applies to every monic scalar--chiral factorization.
\end{theorem}

\begin{proof}
The differential equation follows from \eqref{eq:beta-differential}, since $\eta$ is constant.  The algebraic equation follows from \eqref{eq:beta-algebraic} and $\eta^2=c/4$:
\[
s^2=\beta(\beta+2\eta)+\eta^2=2c-|A|^2+\frac c4.
\]
The one-form dual to $J\grad H$ is, up to the orientation convention, $*\dd H$.  Since $\dd s=*\dd H$, applying $\dd$ gives $\Delta H=0$; applying the same argument after rotating once more gives $\Delta s=0$.

Write $s=a+ib$.  Since $\grad s=J\grad H$ is real, $b$ is constant.  The second equation in \eqref{eq:s-system} is real, so $2ab=0$.  If $b\neq0$, then $a=0$, hence $s$ is constant and $H$ is constant.  Thus a non-CMC branch has $b=0$.  The inequality and the exclusion of $c\leq0$ follow immediately from $s^2=9c/4-|A|^2$.
\end{proof}

\begin{corollary}[Holomorphic discriminant datum]\label{cor:holomorphic-datum}
On a non-CMC factorizing chart, the complex-valued function
\begin{equation}\label{eq:F-holo}
\boxed{\mathcal F:=H+is}
\end{equation}
is holomorphic, after fixing the orientation so that $J\partial_x=\partial_y$ in a conformal coordinate.  Moreover,
\begin{align}
|A^\circ|^2&=\frac{9c}{4}-s^2-2H^2,\label{eq:A0-disc}\\
K&=2H^2+\frac12s^2-\frac c8.\label{eq:K-disc}
\end{align}
\end{corollary}

\begin{proof}
The equation $\grad s=J\grad H$ is the Cauchy--Riemann system for $H+is$.  Since $|A^\circ|^2=|A|^2-2H^2$, \eqref{eq:A0-disc} follows from \eqref{eq:s-system}.  Finally,
\[
\det A=2H^2-\frac12|A|^2
\]
and the Gauss equation $K=c+\det A$ gives \eqref{eq:K-disc}.
\end{proof}

\begin{proposition}[Differential discriminant identities]\label{prop:disc-identities}
On a non-CMC factorizing chart,
\begin{align}
\grad|A|^2&=-2sJ\grad H,\label{eq:gradA2-disc}\\
\langle\grad|A|^2,\grad H\rangle&=0,\label{eq:orth-disc}\\
\Delta|A|^2&=-2|\grad H|^2,\label{eq:lapA2-disc}\\
\Delta\Disc&=8|\grad H|^2,\label{eq:lapDisc}\\
|\grad\Disc|^2&=2\Disc\,\Delta\Disc.\label{eq:disc-transnormal}
\end{align}
\end{proposition}

\begin{proof}
Differentiate $|A|^2=9c/4-s^2$ and use $\grad s=J\grad H$.  Since $s$ is harmonic and $|\grad s|=|\grad H|$,
\[
\Delta|A|^2=-\Delta(s^2)=-2|\grad s|^2.
\]
The last two identities follow from $\Disc=4s^2$.
\end{proof}

\begin{corollary}[Umbilic exclusion]\label{cor:no-umbilic}
A non-CMC mean-curvature-normalized factorizing surface has no umbilic points.
\end{corollary}

\begin{proof}
By \eqref{eq:A0-disc}, the holomorphic map $\mathcal F=H+is$ takes values in the closed ellipse
\[
2(\Re\mathcal F)^2+(\Im\mathcal F)^2\leq\frac{9c}{4}.
\]
An umbilic point is a point where equality holds.  A nonconstant holomorphic map is open, and therefore cannot attain a boundary point of this ellipse while its image remains inside the ellipse.  Hence $|A^\circ|>0$ everywhere.
\end{proof}

\begin{proposition}[The factorization wall]\label{prop:wall}
On the real factorization branch, the repeated-root wall is
\[
\mathcal W=\{\Disc=0\}=\{s=0\}.
\]
At a regular wall point,
\begin{equation}\label{eq:hess-wall}
\operatorname{Hess}\Disc=8\,\dd s\otimes\dd s.
\end{equation}
Thus the discriminant has quadratic contact with zero and does not cross to a negative real-factorization chamber.
\end{proposition}

\begin{proof}
Differentiate $\Disc=4s^2$ twice and restrict to $s=0$.
\end{proof}

\begin{theorem}[Compact factorization rigidity]\label{thm:compact-rigidity}
Let $M$ be closed and connected.  If $\mathscr B_c$ admits a global mean-curvature-normalized factorization, then $H$, $s$, and $|A|^2$ are constant, and the principal curvatures are constant.  Conversely, a surface with constant principal curvatures admits a mean-curvature-normalized factorization over the complex spinor bundle, with
\begin{equation}\label{eq:beta-isoparametric}
\beta=-\eta\pm\frac12\sqrt{9c-4|A|^2}.
\end{equation}
\end{theorem}

\begin{proof}
The harmonic function $H$ is constant on a closed surface.  Then \eqref{eq:s-system} makes $s$ constant, hence $|A|^2$ is constant.  The principal curvatures are the roots of a quadratic polynomial whose trace and squared norm are constant, so their unordered pair is constant.  The converse is \Cref{cor:CMC}.
\end{proof}

\begin{corollary}[A noncompact Liouville extension]\label{cor:liouville-rigidity}
Suppose the factorization field is globally defined and $M$ has the
bounded-holomorphic Liouville property, meaning that every bounded holomorphic
function on $M$ is constant.  Then the principal curvatures are constant.
\end{corollary}

\begin{proof}
A non-CMC branch would have $c>0$, and the nonconstant holomorphic function $\mathcal F=H+is$ would be bounded by the ellipse in the proof of \Cref{cor:no-umbilic}.  This contradicts the assumed Liouville property.
\end{proof}

\subsection{Local spherical rigidity by a finite-type Frobenius obstruction}\label{sec:local-reduction}

The preceding section has reduced the entire non-CMC question to one sharply defined case: a nonumbilic surface in $\Sph^3(c)$ with $c>0$.  The purpose of this section is to close that final branch.  The argument proceeds in four steps:
\begin{enumerate}[label=\textup{(\arabic*)},leftmargin=2.2em]
\item use the holomorphic coordinate $H+is$ to rewrite Gauss--Codazzi;
\item eliminate the phase of the Hopf differential and obtain a scalar eikonal--Liouville system;
\item prolong the phase form to a rank-two distribution on a finite-dimensional state space;
\item prove that its Frobenius torsion is strictly negative by exact algebraic certificates.
\end{enumerate}
An integral surface would require the torsion to vanish, so the strict sign gives the desired local contradiction.

\subsubsection{Phase-free Gauss--Codazzi reduction}

Assume for contradiction that the holomorphic function
\[
\mathcal F=H+is
\]
is nonconstant.  Its regular set is nonempty and open.  On any connected regular chart we use
\begin{equation}\label{eq:z-F}
z=x+iy:=\mathcal F
\end{equation}
as a conformal coordinate and write
\[
g=e^{2u}|\dd z|^2,
\qquad
H=x,
\qquad
s=y.
\]
Set
\begin{equation}\label{eq:R-K-local}
R(x,y):=\frac{9c}{4}-2x^2-y^2,
\qquad
K(x,y):=2x^2+\frac12y^2-\frac c8.
\end{equation}
By \Cref{cor:no-umbilic}, $R>0$.

Let $Q\,\dd z^2$ be the Hopf differential, normalized by
\[
II=Q\,\dd z^2+He^{2u}|\dd z|^2+\overline Q\,\dd\overline z^2.
\]
The Gauss--Codazzi equations are
\begin{align}
\Delta_0u+e^{2u}K&=0,\label{eq:GC-gauss}\\
Q_{\bar z}&=\frac14e^{2u},\label{eq:GC-codazzi}\\
|Q|^2&=\frac18e^{4u}R,\label{eq:GC-modulus}
\end{align}
where $\Delta_0=\partial_x^2+\partial_y^2$.

\begin{proposition}[Phase-free elimination of the Hopf differential]\label{prop:Phi-reduction}
Locally there is a real function $\Phi$ such that $Q=\Phi_z$.  With
\[
\rho:=\Delta_0\Phi=e^{2u}>0,
\]
Equations \eqref{eq:GC-gauss}--\eqref{eq:GC-modulus} are equivalent to
\begin{align}
|\grad_0\Phi|^2&=\frac R2(\Delta_0\Phi)^2,\label{eq:eikonal-Phi}\\
\Delta_0\log(\Delta_0\Phi)+2K\Delta_0\Phi&=0,\label{eq:Liouville-Phi}\\
\Delta_0\Phi&>0.\label{eq:positive-Phi}
\end{align}
\end{proposition}

\begin{proof}
Writing $Q=A+iB$, the reality of $Q_{\bar z}$ gives $B_x+A_y=0$, which is the local integrability condition for $Q=\Phi_z$ with $\Phi$ real.  Then
\[
\Delta_0\Phi=4\Phi_{z\bar z}=4Q_{\bar z}=e^{2u}
\]
and $|\grad_0\Phi|^2=4|Q|^2$, yielding \eqref{eq:eikonal-Phi}.  Since $u=\frac12\log \rho$, the Gauss equation becomes \eqref{eq:Liouville-Phi}.
\end{proof}

\subsubsection{Normalization and the phase system}

\begin{lemma}[Curvature normalization]\label{lem:spherical-normalization}
For $c>0$, replace the ambient metric $h$ by $\widehat h=ch$.  Then the ambient curvature becomes $1$, while
\[
\widehat g=cg,
\qquad
\widehat A=c^{-1/2}A,
\qquad
\widehat H=c^{-1/2}H,
\qquad
\widehat s=c^{-1/2}s.
\]
Consequently the dimensionless Gauss--Codazzi and factorization equations have the same form with $c=1$.
\end{lemma}

\begin{proof}
A constant rescaling does not change the Levi--Civita connection.  The unit normal rescales by $c^{-1/2}$, so the shape operator and its trace rescale by $c^{-1/2}$.  The shifted chiral field has the same inverse-length weight because $s^2=9c/4-|A|^2$.  Substitution into the Gauss--Codazzi equations proves the claim.
\end{proof}

We henceforth set $c=1$ and write
\begin{equation}\label{eq:R-K-unit}
R=\frac94-2x^2-y^2,
\qquad
K=2x^2+\frac12y^2-\frac18,
\qquad
\mu=\sqrt{\frac2R},
\qquad
\rho=e^{2u}>0.
\end{equation}
Because $R>0$, Equation \eqref{eq:GC-modulus} gives $Q\neq0$.  After shrinking the chart, a real phase $\theta$ therefore exists such that
\begin{equation}\label{eq:Q-phase}
Q=e^{2u}\sqrt{\frac R8}\,e^{i\theta}.
\end{equation}
Put $C=\cos\theta$, $S=\sin\theta$ and introduce
\begin{equation}\label{eq:XY-frame}
X=C\partial_x-S\partial_y,
\qquad
Y=-S\partial_x-C\partial_y,
\end{equation}
together with
\[
d=2xC-yS,
\qquad
e=2xS+yC.
\]
The Gauss-phase equation determines
\begin{equation}\label{eq:a-def-local}
a:=S\theta_x+C\theta_y
=\frac dR+\frac3{\mu R}
 +\frac{8x^2+2y^2}{\mu R^2}+\frac{2K\rho}{\mu}.
\end{equation}
Combining the two Codazzi equations gives
\begin{equation}\label{eq:first-order-local}
Xu=P,
\qquad
Yu=v,
\qquad
X\theta=B,
\qquad
Y\theta=-a,
\end{equation}
where
\begin{equation}\label{eq:PB-local}
P=\frac12\left(a+\mu+\frac dR\right),
\qquad
B=2v+\frac eR,
\end{equation}
and $v:=Yu$ is the sole undetermined first derivative.  Since
\begin{equation}\label{eq:XY-commutator}
[X,Y]=-BX+aY,
\end{equation}
compatibility for $u$ and $\theta$ determines both derivatives of $v$:
\begin{align}
Xv&=f:=YP-BP+av,\label{eq:f-local}\\
Yv&=g:=\frac12\left(B^2+a^2-Xa-Y\!\left(\frac eR\right)\right).\label{eq:g-local}
\end{align}
Thus the first prolongation defines two vector fields on the five-dimensional state space $(x,y,u,\theta,v)$,
\begin{align}
\mathcal X&=C\partial_x-S\partial_y
 +P\partial_u+B\partial_\theta+f\partial_v,\label{eq:Xcal-local}\\
\mathcal Y&=-S\partial_x-C\partial_y
 +v\partial_u-a\partial_\theta+g\partial_v.\label{eq:Ycal-local}
\end{align}
Here $v$ records a first derivative of $u$, while $f$ and $g$ record its
derivatives in the $X$- and $Y$-directions.  In
\eqref{eq:f-local}--\eqref{eq:g-local}, expressions such as $YP$ and $Xa$
are total derivatives on the prolonged state space: the prescribed relations
$Xu=P$, $Yu=v$, $X\theta=B$, and $Y\theta=-a$ are substituted before
differentiation is completed.  A direct commutator calculation gives
\begin{equation}\label{eq:frobenius-bracket}
[\mathcal X,\mathcal Y]+B\mathcal X-a\mathcal Y
=\mathcal C\,\partial_v,
\end{equation}
where
\begin{equation}\label{eq:torsion-def}
\boxed{
\mathcal C:=\mathcal X(g)-\mathcal Y(f)+Bf-ag.
}
\end{equation}
Therefore an integral surface of the prolonged rank-two distribution can
exist only if $\mathcal C=0$.

\subsubsection{Strict sign of the Frobenius curvature}

\begin{theorem}[Finite-type spherical obstruction]\label{thm:finite-type-spherical}
On the region $R>0$ and $\rho>0$, the curvature $\mathcal C$ in \eqref{eq:torsion-def} is strictly negative.  Hence the spherical non-CMC compatibility system has no local solution.
\end{theorem}

\begin{proof}
Set
\begin{equation}\label{eq:XYDLM-local}
X_0=x^2,
\qquad
Y_0=y^2,
\qquad
D=9-8X_0-4Y_0=4R,
\end{equation}
\[
L=63-24X_0-20Y_0,
\qquad
M=24X_0+4Y_0+9.
\]
On the closed discriminant triangle $X_0,Y_0\ge0$, $D\ge0$, one has $L\ge18$ and $M>0$.  The differential system itself is used only on the interior $D>0$; the closure is used solely for polynomial sign certification.

Exact algebraic reduction of \eqref{eq:torsion-def}, using $C^2+S^2=1$ and $\mu^2R=2$, gives
\begin{equation}\label{eq:C-quadratic-local}
\mathcal C
=\gamma_2v^2+\gamma_1v
 +\alpha_2\rho^2+\alpha_1\rho+\alpha_0,
\end{equation}
where
\begin{align}
\gamma_2&=-\frac{2\sqrt2\,L}{D^{3/2}}<0,\label{eq:gamma2-local}\\
\gamma_1&=\frac{16\sqrt2}{D^{5/2}}
 \Bigl[xS(48X_0+56Y_0-198)
       +yC(8X_0+20Y_0-81)\Bigr],\label{eq:gamma1-local}\\
\alpha_2&=-\frac{\sqrt2}{256\sqrt D}
 (16X_0+4Y_0-1)^2M\le0.\label{eq:alpha2-local}
\end{align}
For the coefficient of $\rho$, define
\begin{align}
U={}&192X_0^2+256X_0Y_0-720X_0
   +64Y_0^2-292Y_0+261,\label{eq:U-local}\\
V={}&64X_0^2-128X_0Y_0+704X_0
   -48Y_0^2+288Y_0-261,\label{eq:V-local}\\
N={}&4096X_0^3-2304X_0^2Y_0+23616X_0^2
   -2816X_0Y_0^2+19968X_0Y_0-18576X_0\notag\\
 &\quad -576Y_0^3+4592Y_0^2-9468Y_0+7209.\label{eq:N-local}
\end{align}
Then
\begin{equation}\label{eq:alpha1-local}
\alpha_1
=-\frac{\sqrt2\,N}{8D^{5/2}}
 -\frac{4xUC+yVS}{D^2}.
\end{equation}
The exact polynomial certificates in \Cref{app:bernstein-certificate} give
\begin{equation}\label{eq:N-amplitude-local}
N>0,
\qquad
N^2-32D\bigl(16X_0U^2+Y_0V^2\bigr)>0.
\end{equation}
Moreover,
\[
|4xUC+yVS|
\le \sqrt{16X_0U^2+Y_0V^2}.
\]
After squaring, the second inequality in \eqref{eq:N-amplitude-local} is precisely
\[
\frac{\sqrt2\,N}{8D^{5/2}}
>
\frac{\sqrt{16X_0U^2+Y_0V^2}}{D^2}.
\]
Thus the negative constant term in \eqref{eq:alpha1-local} strictly dominates its trigonometric part and
\begin{equation}\label{eq:alpha1-negative-local}
\alpha_1<0.
\end{equation}

Regard \eqref{eq:C-quadratic-local} as a quadratic in $v$.  Its discriminant is
\begin{equation}\label{eq:Delta-v-local}
\Delta_v
=\gamma_1^2-4\gamma_2
 (\alpha_2\rho^2+\alpha_1\rho+\alpha_0)
=-J(\rho),
\end{equation}
where
\[
J(\rho)=J_2\rho^2+J_1\rho+J_0.
\]
The first two coefficients satisfy
\begin{equation}\label{eq:J21-local}
J_2=\frac{L(16X_0+4Y_0-1)^2M}{16D^2}\ge0,
\qquad
J_1=4\gamma_2\alpha_1>0.
\end{equation}
It remains to prove $J_0>0$.

Let $z=(C,S)^T$.  A further exact reduction gives
\begin{equation}\label{eq:J0-matrix-local}
D^6J_0=z^TQz+\ell^Tz+c_0,
\end{equation}
where
\[
Q=\begin{pmatrix}q_{11}&q_{12}\\q_{12}&q_{22}\end{pmatrix},
\]
\begin{align}
q_{11}&=32DP_C,
& q_{22}&=32DP_S,
& q_{12}&=8192Dxy(8X_0+9)(2Y_0-9),\label{eq:q-local}
\end{align}
with
\begin{align}
P_C={}&10752X_0^3+12544X_0^2Y_0-93888X_0^2
 +2944X_0Y_0^2-72000X_0Y_0+139320X_0\notag\\
 &-2880Y_0^3+19152Y_0^2-64476Y_0+86751,\label{eq:PC-local}\\
P_S={}&10752X_0^3+16640X_0^2Y_0-93888X_0^2
 -1152X_0Y_0^2-25920X_0Y_0+56376X_0\notag\\
 &-2880Y_0^3+19152Y_0^2-59292Y_0+86751.\label{eq:PS-local}
\end{align}
Moreover,
\begin{align}
\ell_1&=1024\sqrt2\,x(2Y_0-9)\sqrt D\,ML,
&\ell_2&=256\sqrt2\,y(8X_0+9)\sqrt D\,ML,\label{eq:ell-local}\\
c_0&=16(8X_0-4Y_0+27)^2ML.\label{eq:c0-local}
\end{align}
Define
\begin{align}
F_2={}&448X_0^2+320X_0Y_0-2736X_0
       -144Y_0^2+504Y_0-1377,\label{eq:F2-local}\\
F_3={}&10752X_0^3+12544X_0^2Y_0-93888X_0^2
       -1152X_0Y_0^2-35136X_0Y_0+56376X_0\notag\\
 &\quad -2880Y_0^3+19152Y_0^2-64476Y_0+86751,\label{eq:F3-local}\\
F_4={}&28672X_0^4+90112X_0^3Y_0+18432X_0^3
 +92160X_0^2Y_0^2-377856X_0^2Y_0+1047168X_0^2\notag\\
 &\quad +30720X_0Y_0^3-281088X_0Y_0^2
 +1337472X_0Y_0-1842912X_0\notag\\
 &\quad -2304Y_0^4+39168Y_0^3-215136Y_0^2
 +711504Y_0-1003833.\label{eq:F4-local}
\end{align}
The exact certificates in \Cref{app:bernstein-certificate} establish
\begin{equation}\label{eq:factor-signs-local}
P_C>0,
\qquad
F_2<0,
\qquad
F_3>0,
\qquad
F_4<0.
\end{equation}
The determinant factors as
\begin{equation}\label{eq:detQ-local}
\det Q=(32D)^2(-L)F_2F_3>0.
\end{equation}
Since $q_{11}=32DP_C>0$, the matrix $Q$ is positive definite.  The Schur-complement numerator has the exact factorization
\begin{align}
&4c_0\det Q-
 \bigl(q_{22}\ell_1^2-2q_{12}\ell_1\ell_2+q_{11}\ell_2^2\bigr)\notag\\
&\hspace{35mm}=-65536D^2ML^2F_3F_4>0.\label{eq:Schur-local}
\end{align}
Therefore
\[
c_0-\frac14\ell^TQ^{-1}\ell>0.
\]
Completing the square explicitly,
\begin{align*}
z^TQz+\ell^Tz+c_0
={}&\left(z+\frac12Q^{-1}\ell\right)^T
Q\left(z+\frac12Q^{-1}\ell\right)\\
&+c_0-\frac14\ell^TQ^{-1}\ell>0.
\end{align*}
Thus $J_0>0$.  Notice that this final estimate is valid for every $z\in\R^2$ and does not require the constraint $C^2+S^2=1$.

It follows from \eqref{eq:J21-local} that $J(\rho)>0$ for every $\rho>0$, hence $\Delta_v<0$.  Because $\gamma_2<0$, the quadratic \eqref{eq:C-quadratic-local} is strictly negative for every real $v$, contradicting the necessary condition \eqref{eq:torsion-def}.  No local solution exists.
\end{proof}

\begin{theorem}[Local Dirac-discriminant rigidity]\label{thm:local-rigidity}
Let $\phi:M^2\to N^3(c)$ be a connected oriented immersed surface.  Suppose its spinorial Ou operator admits a mean-curvature-normalized factorization of the form \eqref{eq:factorization-ansatz}.  Then
\[
\grad H=0,
\qquad
\grad|A|^2=0.
\]
Consequently the principal curvatures are locally constant.
\end{theorem}

\begin{proof}
By \Cref{thm:disc-structure}, a non-CMC factorizing branch can occur only for $c>0$, where $s$ is real and $\mathcal F=H+is$ is holomorphic.  If $\mathcal F$ were nonconstant, its regular set would be nonempty.  On every connected component of that set the preceding reduction applies, while \Cref{thm:finite-type-spherical} forbids an integral surface.  Hence $\mathcal F$ is constant and $H$ is constant.  The factorization equations then make $s$ and $|A|^2$ constant.  Since the trace and squared norm of $A$ are constant, the unordered pair of principal curvatures is locally constant.
\end{proof}

\begin{corollary}[Scope of the local classification]\label{cor:factorization-isoparametric}
Let $M$ be connected.
\begin{enumerate}[label=\textup{(\alph*)},leftmargin=2.2em]
\item If $c\neq0$, then $\mathscr B_c$ admits a monic scalar--chiral factorization if and only if the surface is locally isoparametric.
\item If $c=0$, then $\mathscr B_0$ admits a mean-curvature-normalized factorization if and only if the surface is locally isoparametric.
\end{enumerate}
In both cases the normalized chiral masses are
\[
\beta=-\eta\pm\frac12\sqrt{9c-4|A|^2}.
\]
General Euclidean monic factors not satisfying $p=-H$ are precisely the exceptional family described in \Cref{prop:flat-exceptional}; no isoparametric classification of that larger family is asserted here.
Here and below, \emph{locally isoparametric} means that the two principal
curvatures are locally constant.
\end{corollary}

\begin{proof}
The normalized forward implication is \Cref{thm:local-rigidity}, and for $c\neq0$ every monic factor is normalized by \Cref{cor:nonflat-exhaustive}.  Conversely, constant principal curvatures imply constant $H$ and $|A|^2$, so \Cref{cor:CMC} supplies a constant normalized solution of the scalar Riccati system.
\end{proof}

\begin{remark}[Scope of the classification]
\Cref{cor:factorization-isoparametric} classifies immersions for which the geometric operator $\mathscr B_c$ admits the specified first-order factorization.  It does not claim that every BCI has a factorizing operator.  On a nonminimal region where $H\neq0$, an actual BCI conformal factor is selected instead by a positive constrained zero mode $U=\lambda^2H$, as discussed in \Cref{sec:examples}.  Minimal conformal immersions require separate interpretation because then $U\equiv0$ does not determine $\lambda$.
\end{remark}

\begin{remark}[Exact computer-assisted verification]\label{rem:computer-assisted}
Two exact computer-assisted steps enter \Cref{thm:finite-type-spherical}.
First, symbolic differential algebra expands
\eqref{eq:f-local}--\eqref{eq:torsion-def} and verifies the displayed formulas
for $\gamma_j$, $\alpha_j$, $J_0$, $Q$, its determinant, and its Schur
complement.  Second, rational Bernstein arithmetic certifies positivity of the
six explicitly displayed polynomials in \eqref{eq:N-amplitude-local} and
\eqref{eq:factor-signs-local}.  Neither step uses floating-point sampling.  The symbolic identities are independently reproduced by \path{derive_frobenius_torsion.py}, whose output is supplied as \path{frobenius_torsion_verification.txt}; the Bernstein sign certificates are reproduced by \path{verify_bernstein_certificates.py}, with complete output in \path{bernstein_certificate_full.txt}.
\end{remark}

\section{Diagnostic surface geometries and positive reconstruction}\label{sec:surface-diagnostics}\label{sec:examples}

The rigidity theorem classifies when the geometric operator factors, but the original BCI problem asks for more: a real zero mode satisfying the scalar-line condition $U=\lambda^2H$ and the positivity condition $\lambda^2>0$.  This section returns from operator algebra to that geometric reconstruction problem.  The cylinder, the proper small sphere, and the Euclidean round sphere show that factorization, kernel nontriviality, and positivity are different properties.

On a region where $H\neq0$, a conformal factor is recovered from a real solution $U$ by $\lambda^2=U/H>0$.  When $H\equiv0$, one has $U\equiv0$ for every conformal factor; the weighted field therefore does not recover $\lambda$, although the conformal immersion is harmonic and hence biharmonic.  The following nonminimal examples separate operator factorization from positive conformal-factor reconstruction.

\begin{definition}[Positive constrained kernel]\label{def:positive-kernel-surface}
On a region where $H$ is nowhere zero, define the geometric positive constrained
kernel by
\begin{equation}\label{eq:positive-kernel}
\mathcal K_+(\phi)
:=\left\{U\in C^\infty(M,\R):
\begin{array}{l}
-\Delta U+(|A|^2-2c)U=0,\\
A(\grad U)+U\grad H=0,\\
U/H>0
\end{array}\right\}.
\end{equation}
On any chart carrying a restricted ambient Killing spinor, the first two
conditions are equivalent to $\mathscr B_c(U\psi)=0$ by
\Cref{thm:exact}.  This definition is therefore independent of whether the
ambient Killing spinor descends globally.  Elements of
$\mathcal K_+(\phi)$ are precisely the weighted mean-curvature fields of
biharmonic conformal reparametrizations of the fixed immersion, with
$\lambda^2=U/H$.
\end{definition}

\subsection{Ou's Euclidean cylinder family}\label{subsec:cylinder}

Let
\begin{equation}\label{eq:cylinder-param}
\phi(\theta,z)=(R\cos\theta,R\sin\theta,z),
\qquad R>0,
\end{equation}
with induced metric
\[
g=R^2\dd\theta^2+\dd z^2.
\]
Choose the unit normal so that, for
\[
e_1=R^{-1}\partial_\theta,
\qquad
e_2=\partial_z,
\]
one has
\begin{equation}\label{eq:cylinder-A}
Ae_1=R^{-1}e_1,
\qquad
Ae_2=0,
\qquad
H=\frac1{2R},
\qquad
|A|^2=\frac1{R^2}.
\end{equation}
For $c=0$, Ou's equations reduce to
\[
\Delta U-\frac1{R^2}U=0,
\qquad
A(\grad U)=0.
\]
The tangential equation gives $U_\theta=0$, and hence
\begin{equation}\label{eq:cylinder-U-ode}
U''(z)-R^{-2}U(z)=0.
\end{equation}
Since $H$ is constant and nonzero, $U=H\lambda^2$, so the full conformal family is
\begin{equation}\label{eq:cylinder-lambda}
\boxed{
\lambda^2(z)=a e^{z/R}+b e^{-z/R},
}
\end{equation}
on every connected domain where the right-hand side is positive.  On the
complete cylinder, where $z\in\R$, global positivity holds exactly when
$a,b\geq0$ and $(a,b)\neq(0,0)$.  The two one-sided choices reproduce the
opposite exponential conventions appearing in Ou's cylinder examples.

The Dirac discriminant is
\[
\Disc=-\frac4{R^2},
\]
so the two constant factorization fields are
\[
\beta=\pm\frac{i}{R}.
\]
Thus the cylinder factors only after complexification, even though its constrained scalar modes $\lambda^2$ are real and positive.  The nonconstant freedom occurs because
\[
\ker A=\operatorname{span}\{\partial_z\};
\]
the weighted field propagates along the zero-principal-curvature direction.

The induced metric $g$ is complete.  A globally positive nonconstant solution of \eqref{eq:cylinder-lambda} grows exponentially at one or both ends, so the source metric $\bar g=\lambda^{-2}g$ is incomplete at that end and $U\notin L^2(M,g)$.  This explains why the example evades the standard complete finite-energy rigidity hypotheses.

\subsection{The proper small sphere in \texorpdfstring{$\Sph^3(c)$}{S3(c)}}\label{subsec:small-sphere}

Let $c>0$ and consider
\[
\Sph^2\!\left(\frac1{\sqrt{2c}}\right)
\subset
\Sph^3\!\left(\frac1{\sqrt c}\right).
\]
With a suitable normal,
\begin{equation}\label{eq:small-sphere-general}
A=\sqrt c\,\Id,
\qquad
H=\sqrt c,
\qquad
|A|^2=2c.
\end{equation}
Ou's equations give $\Delta U=0$ and $A(\grad U)=0$, hence $U$ and $\lambda$ are constant.  The discriminant is $\Disc=c$, and with $\eta=\sqrt c/2$ the roots are
\[
\beta=0,
\qquad
\beta=-\sqrt c.
\]
Therefore
\begin{align}
\mathscr B_c
&=(D+H+\sqrt c\,\vol)(D-H),\label{eq:small-factor-general-1}\\
\mathscr B_c
&=(D+H)(D-H-\sqrt c\,\vol).\label{eq:small-factor-general-2}
\end{align}
The small sphere is therefore a real exact factorization point, but its positive constrained kernel contains only constant multiples of $H$.

\subsection{A negative control: the Euclidean round sphere}\label{subsec:euclidean-sphere}

For a Euclidean sphere of radius $R$,
\[
A=R^{-1}\Id,
\qquad
H=R^{-1},
\qquad
|A|^2=2R^{-2}.
\]
Its operator factors over $\C$, with $\beta=\pm i\sqrt2/R$.  Nevertheless, $A(\grad U)=0$ forces $U$ to be constant, while the normal equation gives $-2R^{-2}U=0$.  Hence $U=0$ and
\begin{equation}\label{eq:euclidean-sphere-empty}
\mathcal K_+(\phi)=\varnothing.
\end{equation}
This example shows that operator factorization is necessary for the mechanism studied here but is not sufficient for the existence of a positive conformal factor.

\subsection{Hyperbolic factors}\label{subsec:hyperbolic}

For $\Hyp^3(L)$,
\[
c=-L^{-2},
\qquad
\eta=\frac{i}{2L}.
\]
A CMC surface with constant $|A|^2$ has
\begin{equation}\label{eq:beta-hyperbolic-new}
\beta=-\frac{i}{2L}
\pm\frac{i}{2}\sqrt{4|A|^2+\frac9{L^2}}.
\end{equation}
The natural factorization is intrinsically complex.  Moreover, on an isoparametric nonminimal surface, a positive constrained solution cannot occur: if $A$ is invertible, Ou's tangential equation makes $U$ constant and the normal equation would require $|A|^2=2c<0$; if $A$ is singular, the argument in \Cref{prop:isoparametric-kernel} below forces $c=0$.

\begin{proposition}[Positive kernels on the isoparametric branch]\label{prop:isoparametric-kernel}
Let a connected surface in $N^3(c)$ have constant principal curvatures and $H\neq0$.  Suppose $U=\lambda^2H$ satisfies Ou's equations.
\begin{enumerate}[label=\textup{(\roman*)}]
\item If $A$ is invertible, then $U$ is constant and $|A|^2=2c$.  The nonminimal model is locally the proper small sphere in $\Sph^3(c)$.
\item If $U$ is nonconstant, then $A$ is singular, $c=0$, and the surface is locally a circular cylinder.  Its solutions are exactly \eqref{eq:cylinder-lambda}.
\end{enumerate}
\end{proposition}

\begin{proof}
Since $H$ is constant, the tangential equation is $A(\grad U)=0$.  If $A$ is invertible, $U$ is constant, and the normal equation gives $|A|^2=2c$.  If the two constant principal curvatures were distinct, Codazzi would force the principal connection form to vanish, so $K=0$ and the Gauss equation would give $k_1k_2=-c$.  Together with $k_1^2+k_2^2=2c$, this implies $(k_1+k_2)^2=0$, contradicting $H\neq0$.  Thus the surface is umbilical, with $k_1=k_2=\pm\sqrt c$, which is the small sphere.

If $U$ is nonconstant, $A$ has a zero eigenvalue.  The principal curvatures are then distinct and constant, so again $K=0$.  The Gauss equation gives $c+k_1k_2=c=0$.  The nonminimal constant-principal-curvature model in $\R^3$ is the circular cylinder, and its scalar equation is \eqref{eq:cylinder-U-ode}.
\end{proof}

\begin{remark}[Reconstruction principle]\label{rem:reconstruction-principle}
The cylinder and small sphere exhibit two different mechanisms:
\[
\begin{array}{ll}
\text{kernel propagation:}&\ker A\neq0,\quad U\text{ varies along }\ker A,\\[1mm]
\text{curvature balance:}&A\text{ invertible},\quad |A|^2=2c,\quad U\text{ is constant}.
\end{array}
\]
The Euclidean sphere supplies a third possibility: the operator factors, but the positive constrained kernel is empty.
\end{remark}

\section{From the scalar wall to the minimal normal--tangential block}\label{sec:higher-dimension}

This section supplies the analytic bridge between the scalar and mixed regimes.  We first ask how far the surface scalar spinorial construction survives in arbitrary hypersurface dimension.  The answer is useful precisely because it fails: the second-order encoding extends, but the non-CMC scalar factorization does not.  We then return to the full conformal tension field and derive the normal--tangential block forced by its representation type.

There are therefore two logically different higher-dimensional extensions.  The weighted, or $f$-biharmonic, hypersurface equation retains a scalar field $U=fH$ and provides a clean Clifford-grade rigidity test.  The full conformal hypersurface equation contains the independent tangential field $\grad\rho$ and forces the bundle $\mathbb R\xi\oplus TM$.  Keeping these two problems separate prevents the mixed block from appearing as an arbitrary generalization: it emerges exactly where the scalar theory ceases to be universal.

\subsection{The weighted hypersurface square in dimension \texorpdfstring{$m$}{m}}

Let
\[
\phi:M^m\longrightarrow N^{m+1}(c)
\]
be an oriented spin hypersurface, with normalized mean curvature $H=m^{-1}\tr A$, and let $U=fH$ for a positive weight $f$.  In a space form, the weighted hypersurface equations are
\begin{align}
\Delta U-U(|A|^2-mc)&=0,\label{eq:weighted-m-normal}\\
A(\grad U)+\frac m2U\grad H&=0.\label{eq:weighted-m-tangent}
\end{align}
For $m=2$ and $f=\lambda^2$, these are exactly \eqref{eq:Ou-normal}--\eqref{eq:Ou-tangent}.

Work on the restricted ambient spinor bundle and let $\Gamma$ denote normal Clifford multiplication.  We use the extrinsic-spinor conventions
\begin{equation}\label{eq:Gamma-relations}
\Gamma^2=-\Id,
\qquad
\nabla\Gamma=0,
\qquad
\Gamma X\cl=-X\cl\Gamma.
\end{equation}
If $\psi$ is the restriction of an ambient Killing spinor with $4\eta^2=c$, the hypersurface spinorial Gauss formula becomes
\begin{equation}\label{eq:restricted-m}
\nabla_X\psi=-\frac12A(X)\cl\psi-\eta X\cl\Gamma\cl\psi.
\end{equation}
Writing $a=m/2$, tracing and squaring give
\begin{align}
D\psi&=aH\psi+m\eta\Gamma\cl\psi,\label{eq:Dpsi-m}\\
D^2\psi&=a\grad H\cl\psi+a^2(H^2+c)\psi.\label{eq:D2psi-m}
\end{align}
Consequently,
\begin{equation}\label{eq:D2-U-m}
\begin{split}
D^2(U\psi)
={}&\left[-\Delta U+\frac{m^2}{4}(H^2+c)U\right]\psi\\
&+\left[A(\grad U)+\frac m2U\grad H\right]\cl\psi
+2\eta\grad U\cl\Gamma\cl\psi.
\end{split}
\end{equation}

\begin{theorem}[Dimension-$m$ weighted spinorial encoding]\label{thm:weighted-m-spinor}
Define
\begin{equation}\label{eq:B-m-def}
\boxed{
\mathscr B_{m,c}
=D^2+2\eta\Gamma D-m\eta H\Gamma
+|A|^2-\frac{m^2}{4}H^2-\frac{m(m+2)}4c.
}
\end{equation}
Then
\begin{equation}\label{eq:B-m-action}
\begin{split}
\mathscr B_{m,c}(U\psi)
={}&\left[-\Delta U+(|A|^2-mc)U\right]\psi\\
&+\left[A(\grad U)+\frac m2U\grad H\right]\cl\psi.
\end{split}
\end{equation}
Thus \eqref{eq:weighted-m-normal}--\eqref{eq:weighted-m-tangent} are the scalar and tangent-vector components of one Laplace-type equation.
\end{theorem}

\begin{proof}
Equation \eqref{eq:D2-U-m} follows from the scalar product rule and \eqref{eq:restricted-m}--\eqref{eq:D2psi-m}.  Moreover,
\[
2\eta\Gamma D(U\psi)
=-2\eta\grad U\cl\Gamma\cl\psi
+m\eta HU\Gamma\cl\psi-\frac{mc}{2}U\psi.
\]
The first term cancels the Clifford drift in \eqref{eq:D2-U-m}, and the term $-m\eta H\Gamma$ cancels the remaining normal-Clifford mass.  The scalar constants then combine to $|A|^2-mc$.
\end{proof}

The corresponding scalar--normal factorization ansatz is
\begin{equation}\label{eq:factorization-m}
\mathscr B_{m,c}
=\left(D+\frac m2H+(\beta+2\eta)\Gamma\right)
 \left(D-\frac m2H+\beta\Gamma\right).
\end{equation}

\begin{theorem}[Dimension-$m$ scalar Riccati system]\label{thm:factorization-m}
The identity \eqref{eq:factorization-m} holds if and only if
\begin{align}
\grad\beta\cl\Gamma&=\frac m2\grad H\cl,\label{eq:beta-m-diff}\\
\beta(\beta+2\eta)+|A|^2-\frac{m(m+2)}4c&=0.\label{eq:beta-m-alg}
\end{align}
After the shift $s=\beta+\eta$, the algebraic equation becomes
\begin{equation}\label{eq:disc-m}
\boxed{
\Disc_m:=(m+1)^2c-4|A|^2=4s^2.
}
\end{equation}
For $m=2$, $\Gamma$ identifies with the intrinsic area element and \eqref{eq:beta-m-diff} becomes $\grad\beta=J\grad H$; hence \eqref{eq:disc-m} reduces to $\Disc=9c-4|A|^2$.
\end{theorem}

\begin{proof}
Expanding the product in \eqref{eq:factorization-m} gives
\[
D^2+2\eta\Gamma D-m\eta H\Gamma-\frac{m^2}{4}H^2
-\frac m2\grad H\cl+\grad\beta\cl\Gamma
-\beta(\beta+2\eta).
\]
Comparison with \eqref{eq:B-m-def} yields \eqref{eq:beta-m-diff}--\eqref{eq:beta-m-alg}.  Equation \eqref{eq:disc-m} follows from $4\eta^2=c$.
\end{proof}

\begin{corollary}[Dimensional rigidity of the scalar ansatz]\label{cor:scalar-dim-rigidity}
For $m\geq3$, interpreted in the full ambient Clifford algebra or a faithful doubled module, Equation \eqref{eq:beta-m-diff} forces
\[
\grad H=0,
\qquad
\grad\beta=0.
\]
Thus the non-CMC scalar--chiral factorization mechanism is exceptional to surfaces.  Any genuinely non-CMC higher-dimensional factorization must use additional Clifford grades or a nontrivial matrix block.
\end{corollary}

\begin{proof}
For ambient dimension $m+1\geq4$, the tangent one-vector component $\grad H\cl$ and the tangent-normal bivector component $\grad\beta\cl\Gamma$ are linearly independent.  Their equality therefore forces both to vanish.  The exceptional ambient dimension three is precisely the case in which Hodge duality identifies these components through $X\cl\vol=-JX\cl$.
\end{proof}

\begin{corollary}[Factorization case of the Balmu\c{s}--Montaldo--Oniciuc CMC conjecture]\label{cor:BMO-factorization}
Let
\[
\phi:M^m\longrightarrow\Sph^{m+1},
\qquad m\geq3,
\]
be a connected proper biharmonic hypersurface of the unit sphere.  Then the following are equivalent:
\begin{enumerate}[label=\textup{(\roman*)},leftmargin=2.2em]
\item the weighted spinorial operator $\mathscr B_{m,1}$ admits the scalar--normal factorization \eqref{eq:factorization-m};
\item the mean curvature $H$ is constant;
\item $H$ is a nonzero constant and $|A|^2=m$.
\end{enumerate}
In this case the two factorization fields are
\begin{equation}\label{eq:BMO-beta-roots}
\boxed{
\beta_+=\frac{m-2}{2},
\qquad
\beta_-=-\frac m2,
}
\end{equation}
and the corresponding operator identities are
\begin{align}
\mathscr B_{m,1}
&=
\left(D+\frac m2H+\frac m2\Gamma\right)
\left(D-\frac m2H+\frac{m-2}{2}\Gamma\right),
\label{eq:BMO-factor-plus}\\
\mathscr B_{m,1}
&=
\left(D+\frac m2H-\frac{m-2}{2}\Gamma\right)
\left(D-\frac m2H-\frac m2\Gamma\right).
\label{eq:BMO-factor-minus}
\end{align}
Consequently, every proper biharmonic hypersurface of $\Sph^{m+1}$ whose operator admits the scalar--normal factorization \eqref{eq:factorization-m} is CMC.
\end{corollary}

\begin{proof}
The implication \textup{(i)}$\Rightarrow$\textup{(ii)} is \Cref{cor:scalar-dim-rigidity}.  If $H$ is constant, the normal biharmonic equation \eqref{eq:weighted-m-normal}, with $f=1$ and $c=1$, reduces to
\[
H\bigl(|A|^2-m\bigr)=0.
\]
Properness excludes $H=0$, so \textup{(ii)} implies \textup{(iii)}.  Conversely, under \textup{(iii)} the differential condition \eqref{eq:beta-m-diff} is automatic, while \eqref{eq:beta-m-alg} becomes
\[
\beta(\beta+1)-\frac{m(m-2)}4=0.
\]
Its roots are exactly \eqref{eq:BMO-beta-roots}, and substituting them into \eqref{eq:factorization-m} gives \eqref{eq:BMO-factor-plus}--\eqref{eq:BMO-factor-minus}.
\end{proof}

\begin{remark}[Scope of the BMO corollary]\label{rem:BMO-scope}
\Cref{cor:BMO-factorization} is a sharp characterization inside the scalar--normal factorization class, not a proof of the unrestricted Balmu\c{s}--Montaldo--Oniciuc conjecture.  No implication from biharmonicity to \eqref{eq:factorization-m} is asserted.  A hypothetical non-CMC proper biharmonic hypersurface would therefore have to escape the scalar--normal ansatz, through additional Clifford grades or a genuinely matrix-valued normal--tangential superpotential.
\end{remark}

\subsection{The full conformal equation and its normal--tangential block}

Now let
\begin{equation}\label{eq:conformal-m}
\phi:(M^m,\bar g)\longrightarrow N^{m+1}(c),
\qquad
g=\phi^*h=\lambda^2\bar g,
\end{equation}
and set
\[
\rho=\log\lambda,
\qquad
V=\grad\rho.
\]
The conformal tension field is
\begin{equation}\label{eq:mixed-tension}
\boxed{
\mathcal T_m:=\tau_{\bar g}(\phi)
=\lambda^2\bigl(mH\xi-(m-2)V\bigr).
}
\end{equation}

\begin{proposition}[Minimal universal coefficient bundle]\label{prop:block-minimality}
Fix $m>2$.  Among subbundles of $\phi^{-1}TN=\mathbb R\xi\oplus TM$ that are natural under changes of oriented orthonormal tangent frame and that are required to contain the conformal tension field for arbitrary pointwise values of $H$ and $V$, the bundle $\mathbb R\xi\oplus TM$ is minimal.  For $m=2$, the tangential coefficient vanishes identically and the normal line $\mathbb R\xi$ suffices.
\end{proposition}

\begin{proof}
Allowing arbitrary $H$ forces the coefficient bundle to contain the normal line.  When $m>2$, allowing arbitrary $V\in T_xM$ forces it to contain the standard $O(m)$-module $T_xM$ at every point.  Hence it contains their direct sum.  For $m=2$, the coefficient of $V$ in \eqref{eq:mixed-tension} is zero.  This is a universal equivariant minimality statement; it does not claim that a special fixed immersion cannot preserve a smaller accidental subbundle.
\end{proof}
The transformation law for the Jacobi operator under $\bar g=\lambda^{-2}g$ is \cite{OuConformal}
\begin{equation}\label{eq:Jacobi-conformal}
J_{\bar g}^{\phi}Z
=\lambda^2\left[J_g^{\phi}Z+(m-2)\nabla_V^{\phi}Z\right].
\end{equation}
With the convention $J_g^{\phi}(\tau_g)=-\tau_2(\phi,g)$, biharmonicity is therefore equivalent to
\begin{equation}\label{eq:drifted-Jacobi}
\boxed{
\left(J_g^{\phi}+(m-2)\nabla_V^{\phi}\right)\mathcal T_m=0.
}
\end{equation}
This compact equation is equivalent to the general conformal hypersurface system of Cherif--Ou \cite{CherifOu}.

Write a section of $\phi^{-1}TN$ as $Z=f\xi+Y$, where $Y\in\Gamma(TM)$, and define
\begin{align}
\mathcal C_A(Y)&:=\sum_{i=1}^m\langle Ae_i,\nabla_{e_i}Y\rangle,\label{eq:C-A}\\
\Delta_1Y&:=\sum_{i=1}^m\left(\nabla_{e_i}\nabla_{e_i}Y-\nabla_{\nabla_{e_i}e_i}Y\right).\label{eq:Delta-one}
\end{align}
Here $V$ in a scalar diagonal entry denotes directional differentiation, $V(f)=\langle V,\grad f\rangle$.

\begin{theorem}[Normal--tangential Jacobi block]\label{thm:block-m}
Relative to $\phi^{-1}TN=\mathbb R\xi\oplus TM$, Equation \eqref{eq:drifted-Jacobi} is
\begin{equation}\label{eq:block-equation-m}
\boxed{
\mathbb B_{m,\rho}
\begin{pmatrix}
m\lambda^2H\\[1mm]
-(m-2)\lambda^2V
\end{pmatrix}=0,
}
\end{equation}
where
\begin{equation}\label{eq:block-m}
\begingroup
\setlength{\arraycolsep}{3pt}
\boxed{
\mathbb B_{m,\rho}
=
\begin{pmatrix}
-\Delta+(m-2)V+|A|^2-mc
&-2\mathcal C_A-m\,dH+(m-2)(AV)^\flat\\[1mm]
2A\grad+m(\,\cdot\,)\grad H-(m-2)(\,\cdot\,)AV
&-\Delta_1+(m-2)\nabla_V+A^2-(m-1)c
\end{pmatrix}.
}
\endgroup
\end{equation}
Its principal symbol is $|\zeta|^2\Id_{\mathbb R\oplus TM}$, so the system is strongly elliptic.
\end{theorem}

\begin{proof}
The Gauss--Weingarten equations give
\begin{equation}\label{eq:nabla-V-block}
\nabla_V^{\phi}
=
\begin{pmatrix}
V&(AV)^\flat\\
-(\,\cdot\,)AV&\nabla_V
\end{pmatrix}.
\end{equation}
A direct normal--tangential splitting of the ambient Jacobi operator gives
\begin{equation}\label{eq:Jacobi-block}
J_g^{\phi}
=
\begin{pmatrix}
-\Delta+|A|^2-mc&-2\mathcal C_A-m\,dH\\
2A\grad+m(\,\cdot\,)\grad H&-\Delta_1+A^2-(m-1)c
\end{pmatrix}.
\end{equation}
Combining \eqref{eq:nabla-V-block} and \eqref{eq:Jacobi-block} proves \eqref{eq:block-m}.
\end{proof}

\begin{remark}[Faithfulness of the bridge]\label{rem:block-faithfulness}
The map
\[
\mathfrak c:\mathbb R\xi\oplus TM\longrightarrow
\mathbb C\Gamma\oplus c_N(TM)
\]
is a coefficient-bundle isomorphism, and \eqref{eq:block-equation-m} is obtained by an exact normal--tangential splitting of \eqref{eq:drifted-Jacobi}.  Consequently the block and Clifford formulations neither add nor discard solutions of the Ou--Cherif equations on the mixed conformal-tension field.  Their advantage is structural: they convert the component equations into an operator problem in which the diagonal spectra, off-diagonal couplings, admissible zero modes, and obstructions can be studied separately and then recombined.
\end{remark}

The block has a coefficientwise Clifford realization.  Let $c_N$ denote ambient Clifford multiplication, set $\Gamma=c_N(\xi)$, and define
\[
\mathfrak c(f,Y)=f\Gamma+c_N(Y).
\]
Then
\begin{equation}\label{eq:Theta-m}
\Theta_m
=\mathfrak c\!\left(m\lambda^2H,-(m-2)\lambda^2V\right)
=m\lambda^2H\Gamma-(m-2)\lambda^2c_N(V)
\end{equation}
belongs to the two-grade coefficient bundle
\[
\mathcal K_m=\mathbb C\Gamma\oplus c_N(TM).
\]
The operator
\begin{equation}\label{eq:Cl-block}
\mathscr B^{\mathrm{Cl}}_{m,\rho}
:=\mathfrak c\circ\mathbb B_{m,\rho}\circ\mathfrak c^{-1}
\end{equation}
is the canonical two-grade Clifford realization of the normal--tangential block.  The minimality proved in \Cref{prop:block-minimality} concerns the universal coefficient bundle, not the rank of every possible Clifford-module realization.  Additional Clifford grades may become useful when one seeks a first-order factorization of \eqref{eq:Cl-block}.

A useful equivalent form exposes the critical dimensions.  Set
\[
Q_m=mH\xi-(m-2)V,
\qquad
\mathcal T_m=\lambda^2Q_m.
\]
The product rule for the Jacobi operator transforms \eqref{eq:drifted-Jacobi} into
\begin{equation}\label{eq:unweighted-Qm}
\boxed{
\left[
J_g^{\phi}+(m-6)\nabla_V^{\phi}
-2\bigl(\Delta\rho-(m-4)|V|^2\bigr)
\right]Q_m=0.
}
\end{equation}
Thus $m=2$ removes the tangential tension, $m=4$ removes the explicit quadratic dilation term, and $m=6$ removes the unweighted drift.

\begin{proposition}[Exact recovery of the surface system]\label{prop:block-m2}
At $m=2$, the mixed field collapses to $(2U,0)^T$, with $U=\lambda^2H$, and
\begin{equation}\label{eq:block-m2-action}
\mathbb B_{2,\rho}
\begin{pmatrix}2U\\0\end{pmatrix}
=
\begin{pmatrix}
2[-\Delta U+(|A|^2-2c)U]\\[1mm]
4[A(\grad U)+U\grad H]
\end{pmatrix}.
\end{equation}
Hence the dimension-$m$ block reproduces Ou's two surface equations exactly.  Under the hypersurface-spinor identification $\Gamma\simeq\vol$, the surviving normal block is the component equation encoded by $\mathscr B_c(U\psi)=0$.
\end{proposition}

\subsection{Critical dimension four}

Set $m=4$.  The mixed field and block are
\begin{equation}\label{eq:T4}
\mathcal T_4
=
\begin{pmatrix}
4\lambda^2H\\[1mm]-2\lambda^2V
\end{pmatrix},
\end{equation}
and
\begin{equation}\label{eq:block-four}
\boxed{
\mathbb B_{4,\rho}
=
\begin{pmatrix}
-\Delta+2V+|A|^2-4c
&-2\mathcal C_A-4\,dH+2(AV)^\flat\\[1mm]
2A\grad+4(\,\cdot\,)\grad H-2(\,\cdot\,)AV
&-\Delta_1+2\nabla_V+A^2-3c
\end{pmatrix}.
}
\end{equation}
Let $q=\Delta\rho$.  Expanding the block yields the following compact form of the critical equations.

\begin{proposition}[Expanded $m=4$ equations]\label{prop:m4-expanded}
The block equation $\mathbb B_{4,\rho}\mathcal T_4=0$ is equivalent to
\begin{align}
\mathcal N_4:={}&
\Delta H-H|A|^2+4cH
-\tr(A\circ\nabla^2\rho)
+2Hq-\langle AV,V\rangle=0,\label{eq:N4}\\
\mathcal T_4^{\top}:={}&
4A(\grad H)+8H\grad H
+2\Ric^M(V)+\grad q+2qV+\grad|V|^2=0.\label{eq:T4-tangent}
\end{align}
More precisely,
\begin{equation}\label{eq:block-expanded-factor}
\mathbb B_{4,\rho}
\begin{pmatrix}4\lambda^2H\\-2\lambda^2V\end{pmatrix}
=
\lambda^2
\begin{pmatrix}-4\mathcal N_4\\2\mathcal T_4^{\top}\end{pmatrix}.
\end{equation}
For a space-form hypersurface,
\[
\Ric^M(V)=3cV+4HA(V)-A^2(V),
\]
so \eqref{eq:T4-tangent} is an entirely extrinsic equation in $A,H,\rho$, and $c$.
\end{proposition}

The formulas \eqref{eq:N4}--\eqref{eq:T4-tangent} agree with the $m=4$ specialization of the general equations in \cite{CherifOu}.  The block formulation adds the structural statement that the two equations are the normal and tangential components of the single drifted Jacobi equation
\begin{equation}\label{eq:m4-drifted-J}
\left(J_g^{\phi}+2\nabla_{\grad\rho}^{\phi}\right)
\left[\lambda^2(4H\xi-2\grad\rho)\right]=0.
\end{equation}

\subsection{Role of the critical block}

The formulas above provide the operator used throughout the constructive part of the paper.  At $m=2$, \Cref{prop:block-m2} recovers Ou's scalar and compatibility equations exactly.  At $m=4$, \Cref{prop:m4-expanded} gives the coupled normal and tangential equations used in the rotational and product-orbit reductions below.  Direct substitution also recovers the known totally geodesic and totally umbilical examples of Ou and Cherif--Ou \cite{OuConformal,CherifOu}; these checks fix the signs and dimensional coefficients but are not used as existence input for the new compact constructions.

The important boundary is therefore precise.  The scalar operator $\mathscr B_{m,c}$ extends the factorization analysis of \Cref{sec:factorization-wall}, and \Cref{cor:scalar-dim-rigidity} shows why its non-CMC mechanism is exceptional to surfaces.  The full block $\mathbb B_{m,\rho}$ is a second-order Clifford-block encoding rather than a completed first-order matrix factorization.  The constructions below use the minimal mixed coefficient bundle and its off-diagonal couplings; a first-order matrix Riccati theory for the block remains a separate problem.

\section{Why the natural scalar constructions fail}\label{sec:scalar-obstructions}

Before constructing a mixed solution, it is useful to understand why the most natural surface strategies do not work.  The three obstructions below fail for different reasons: the scalar operator may be coercive, a formal zero mode may violate the geometric reconstruction constraint, or the first-order compatibility equation may overdetermine the conformal factor.  Together they explain why the later examples must use genuinely nonzero normal and tangential components rather than a disguised scalar mode.

For a surface in a three-dimensional space form, set
\begin{equation}\label{eq:L-T-def}
L_\phi=-\Delta+|A|^2-2c,
\qquad
T_\phi U=A(\grad U)+U\grad H.
\end{equation}
On a closed surface, or with boundary conditions for which the formal adjoints are valid, define
\begin{equation}\label{eq:N-def}
\mathcal N_\phi=L_\phi^2+T_\phi^*T_\phi.
\end{equation}
Then
\[
\langle U,\mathcal N_\phi U\rangle
=\|L_\phi U\|_{L^2}^2+\|T_\phi U\|_{L^2}^2,
\qquad
\ker\mathcal N_\phi=\ker L_\phi\cap\ker T_\phi.
\]
The geometric reconstruction imposes one further condition.

\begin{definition}[Positive constrained kernel]\label{def:positive-kernel}
On a connected region on which $H\neq0$, define
\[
\mathcal K_+(\phi)
=
\left\{
0\neq U\in\ker\mathcal N_\phi:
\frac{U}{H}>0
\right\}.
\]
Every $U\in\mathcal K_+(\phi)$ determines the positive conformal factor $\lambda^2=U/H$.
\end{definition}

The following ground-state identity isolates the scalar mechanism.

\begin{proposition}[Geometric ground-state transform]\label{prop:ground-transform}
Suppose $U>0$ and $L_\phi U=0$.  Then for every smooth compactly supported $\Phi$, or every $\Phi$ satisfying compatible boundary conditions,
\begin{equation}\label{eq:ground-transform}
\int_M\left(|\grad\Phi|^2+(|A|^2-2c)\Phi^2\right)\dd\mu
=
\int_M U^2\left|\grad\left(\frac{\Phi}{U}\right)\right|^2\dd\mu.
\end{equation}
Consequently $L_\phi\geq0$ and $U$ is a ground state.
\end{proposition}

\begin{proof}
Write $\Phi=Uv$, expand $|\grad(Uv)|^2$, and integrate the mixed term by parts using $-\Delta U+(|A|^2-2c)U=0$.
\end{proof}

The next three results locate distinct failure modes of the scalar surface search: coercivity, compatibility rigidity, and nodal rigidity.
\subsection{Compact Hopf tori in \texorpdfstring{$\Sph^3$}{S3}: scalar coercivity}
\label{subsec:hopf-torus-no-go}

Let
\[
\pi:\Sph^3(1)\longrightarrow \Sph^2\!\left(\frac12\right)
\]
be the standard Hopf fibration, let $\gamma$ be a closed unit-speed immersed curve in the base, and let
\[
\Sigma_\gamma=\pi^{-1}(\gamma)
\]
be the associated compact Hopf torus.  Choose the orthonormal tangent frame $(X,V)$, where $X$ is the horizontal lift of $\gamma'$ and $V$ is the unit vertical Hopf field.  With a suitable orientation of the normal, the shape operator is
\begin{equation}\label{eq:hopf-shape}
A(X)=\kappa_gX-V,
\qquad
A(V)=-X,
\end{equation}
where $\kappa_g$ is the geodesic curvature of $\gamma$ in $\Sph^2(1/2)$; see \cite{WangChenConformalFlat}.  Hence
\begin{equation}\label{eq:hopf-invariants}
A=
\begin{pmatrix}
\kappa_g&-1\\
-1&0
\end{pmatrix},
\qquad
H=\frac{\kappa_g}{2},
\qquad
|A|^2=\kappa_g^2+2.
\end{equation}
The scalar Ou operator therefore has the exact form
\begin{equation}\label{eq:hopf-Ou-operator}
\boxed{L_{\Sigma_\gamma}=-\Delta+\kappa_g^2.}
\end{equation}

\begin{theorem}[Hopf-torus no-go]\label{thm:hopf-torus-no-go}
Every biharmonic conformal immersion whose image is a compact Hopf torus in $\Sph^3$ is harmonic.  Equivalently,
\[
\mathcal K_+(\Sigma_\gamma)=\varnothing.
\]
More precisely,
\[
\ker L_{\Sigma_\gamma}
=
\begin{cases}
\{0\},&\kappa_g\not\equiv0,\\[1mm]
\R,&\kappa_g\equiv0,
\end{cases}
\]
but the second line is the minimal Clifford torus and contains no nonzero geometrically admissible field $U=\lambda^2H$.
\end{theorem}

\begin{proof}
Let $U$ solve $L_{\Sigma_\gamma}U=0$.  Compactness and \eqref{eq:hopf-Ou-operator} give
\[
0=\int_{\Sigma_\gamma}U L_{\Sigma_\gamma}U\,\dd\mu
 =\int_{\Sigma_\gamma}
 \left(|\grad U|^2+\kappa_g^2U^2\right)\dd\mu.
\]
Thus $U$ is constant and $\kappa_gU=0$ everywhere.  If $\kappa_g$ is not identically zero, then $U=0$.  If $\kappa_g\equiv0$, the base curve is geodesic and \eqref{eq:hopf-invariants} gives $H=0$.  The torus is minimal, so $U=\lambda^2H$ vanishes for every positive conformal factor.  Since harmonicity of a surface immersion is conformally invariant on the domain, every conformal parametrization in this case is harmonic rather than proper biharmonic.
\end{proof}

\begin{remark}[A sterile scalar crossing]\label{rem:hopf-sterile-crossing}
At the Clifford torus the scalar ground state reaches zero,
\[
\ker L_{\Sigma_\gamma}=\R,
\]
but reconstruction fails because $H=0$.  Thus a formal scalar zero mode need not lie in the positive constrained kernel: the relation $U=\lambda^2H$ can make the crossing geometrically sterile.  The conclusion is consistent with the broader Hopf-cylinder classification in \cite{WangChenConformalFlat}; the proof above isolates the compact spectral mechanism.
\end{remark}

\subsection{No nodal locking for surfaces in space forms}

The positivity condition $U/H=\lambda^2>0$ might suggest a sign-changing excited-state branch in which $U$ and $H$ vanish together.  For surfaces in a three-dimensional space form this is impossible.

\begin{theorem}[Nodal rigidity]\label{thm:nodal-rigidity}
Let
\[
\phi:(M^2,\bar g)\longrightarrow N^3(c)
\]
be a biharmonic conformal immersion with $\lambda>0$.  On every connected nonminimal component, $H$ and $U=\lambda^2H$ are nowhere zero and have a fixed sign.
\end{theorem}

\begin{proof}
Suppose $p$ is a regular zero of $H$.  Since $U=\lambda^2H$,
\[
U(p)=0,
\qquad
\grad U(p)=\lambda^2(p)\grad H(p).
\]
The tangential equation gives
\[
A(\grad H)=0.
\]
At $p$, $\tr A=2H=0$.  Because $A$ is symmetric on a two-dimensional space, one zero eigenvalue and zero trace imply $A(p)=0$.

Choose
\[
e_1=\frac{\grad H}{|\grad H|},
\qquad
e_2\ \text{tangent to the regular nodal curve},
\qquad
h=e_1(H)\neq0.
\]
Along the nodal curve, $A=0$, so $(\nabla_{e_2}A)(p)=0$.  Codazzi in a space form yields
\[
(\nabla_{e_1}A)e_2=(\nabla_{e_2}A)e_1=0.
\]
Differentiate $A\grad U+U\grad H=0$ in the $e_1$ direction.  At $p$ this gives
\[
(\nabla_{e_1}A)e_1=-h e_1.
\]
Therefore
\[
\tr(\nabla_{e_1}A)=-h.
\]
On the other hand, differentiating $\tr A=2H$ gives
\[
\tr(\nabla_{e_1}A)=2h,
\]
a contradiction.

Thus $H$ has no regular zero.  Since $U$ solves a second-order elliptic equation, the local nodal-set structure in dimension two implies that any zero of a nontrivial solution is accumulated by regular nodal arcs.  Hence $U$, and therefore $H$, has no zero at all on a connected nonminimal component.
\end{proof}

\begin{remark}
The argument is special to dimension two.  In higher dimension, $A(\grad H)=0$ and $\tr A=0$ do not force $A=0$, so a higher-dimensional nodal mechanism remains possible.
\end{remark}

\subsection{Compatibility rigidity for biconservative surfaces}
\label{subsec:biconservative-no-go}

For a surface in a three-dimensional space form, biconservativity is the tangential part of the ordinary biharmonic equation,
\begin{equation}\label{eq:biconservative-equation}
A(\grad H)+H\grad H=0.
\end{equation}
The condition arises from the divergence-free stress-bienergy tensor \cite{LoubeauMontaldoOniciuc} and was classified locally in three-dimensional space forms in \cite{CaddeoMontaldoOniciucPiu}; see also \cite{MontaldoOniciucRatto,FetcuNistorOniciuc,NistorOniciuc} for intrinsic, global, and uniqueness developments.  On the non-CMC region, $\grad H$ is a principal direction with principal curvature $-H$, while the second principal curvature is $3H$.  This is the linear Weingarten relation $3k_1+k_2=0$ used in the construction of closed non-CMC examples in \cite{MontaldoPampanoClosed}.

\begin{theorem}[Biconservative compatibility rigidity]\label{thm:biconservative-no-go}
Let
\[
\phi:(M^2,\bar g)\longrightarrow N^3(c),
\qquad
\phi^*h=g=\lambda^2\bar g,
\]
be a biharmonic conformal immersion.  If the induced isometric immersion $(M^2,g)\to N^3(c)$ is biconservative, then $H$ is constant.  Consequently, no non-CMC biconservative surface in a three-dimensional space form supports a proper biharmonic conformal immersion.
\end{theorem}

\begin{proof}
Suppose $H$ is nonconstant.  After restricting to a connected open set, write
\[
h:=H\neq0,
\qquad
\grad h\neq0,
\qquad
 e_1=\frac{\grad h}{|\grad h|},
\]
and complete $e_1$ to an oriented orthonormal principal frame $(e_1,e_2)$.  Equation \eqref{eq:biconservative-equation} and $\tr A=2h$ give
\begin{equation}\label{eq:biconservative-principal}
Ae_1=-h e_1,
\qquad
Ae_2=3h e_2,
\qquad
e_2(h)=0.
\end{equation}

The BCI compatibility equation $A(\grad U)+U\grad h=0$ has components
\[
-h e_1(U)+Ue_1(h)=0,
\qquad
3h e_2(U)=0.
\]
Together with $e_2(h)=0$, these imply
\[
\grad\!\left(\frac Uh\right)=0.
\]
Thus $U=ah$ for a positive constant $a=\lambda^2$ on the chosen open set.  The scalar equation reduces to
\begin{equation}\label{eq:biconservative-normal}
\Delta h-10h^3+2ch=0,
\end{equation}
because $|A|^2=h^2+9h^2=10h^2$.

Put $h'=e_1(h)$.  Codazzi applied to \eqref{eq:biconservative-principal} gives
\begin{equation}\label{eq:biconservative-connection}
\nabla_{e_1}e_1=\nabla_{e_1}e_2=0,
\qquad
\nabla_{e_2}e_1=-\frac{3h'}{4h}e_2,
\qquad
\nabla_{e_2}e_2=\frac{3h'}{4h}e_1.
\end{equation}
Consequently,
\[
\Delta h=h''-\frac{3}{4h}(h')^2,
\]
and \eqref{eq:biconservative-normal} becomes
\begin{equation}\label{eq:biconservative-ode-normal}
h''-\frac{3}{4h}(h')^2-10h^3+2ch=0.
\end{equation}
The intrinsic curvature computed from \eqref{eq:biconservative-connection} is
\[
K=\frac{3}{4h}h''-\frac{21}{16h^2}(h')^2,
\]
whereas Gauss gives $K=c-3h^2$.  Hence
\begin{equation}\label{eq:biconservative-ode-gauss}
12hh''-21(h')^2-16ch^2+48h^4=0.
\end{equation}
Multiplying \eqref{eq:biconservative-ode-normal} by $12h$ and subtracting \eqref{eq:biconservative-ode-gauss} yields
\begin{equation}\label{eq:biconservative-first-integral}
(h')^2=14h^4-\frac{10}{3}ch^2.
\end{equation}
Since $h'\neq0$ on the chosen open set, differentiation gives
\[
h''=28h^3-\frac{10}{3}ch.
\]
Substitution into \eqref{eq:biconservative-ode-normal} reduces the system to
\[
45h^2+7c=0.
\]
Because $c$ is constant, this forces $h^2$ and hence $h$ to be constant, contradicting $\grad h\neq0$.  Therefore the non-CMC region is empty.
\end{proof}

\begin{corollary}[Closed biconservative examples are excluded]\label{cor:closed-biconservative-excluded}
None of the closed non-CMC biconservative surfaces in the round three-sphere constructed in \cite{MontaldoPampanoClosed} admits a proper biharmonic conformal parametrization.
\end{corollary}

\begin{remark}[Complementary obstructions]\label{rem:hopf-biconservative-complement}
\Cref{thm:hopf-torus-no-go} is a scalar obstruction: $L_\phi$ is coercive except at a geometrically sterile minimal crossing.  By contrast, \Cref{thm:biconservative-no-go} begins with the first-order equation: $T_\phi U=0$ freezes the conformal factor, after which the scalar equation and Gauss--Codazzi become inconsistent on the non-CMC branch.
\end{remark}

\subsection{Consequences for the constructive search}

The preceding results do not assert that compact non-CMC BCIs are impossible.  They assert that three natural scalar strategies fail for structural reasons.  A positive scalar zero mode can be removed by coercivity, can fail the reconstruction condition $U=\lambda^2H$, or can be incompatible with the first-order equation.  The higher-dimensional conformal tension avoids this bottleneck because its normal and tangential coefficients are independent and the off-diagonal entries of $\mathbb B_{m,\rho}$ can balance the diagonal channels.  The remaining sections construct hybridized zero modes of this type.

\section{A linear model for normal--tangential hybridization}

Before adding lower symmetry, we isolate the spectral mechanism in a setting where it can be seen explicitly.  On an umbilical four-sphere, scalar and exact-vector spherical harmonics form finite-dimensional invariant channels.  Their coupled determinant exhibits a nonminimal zero crossing that is absent from either diagonal channel alone.  This model does not supply the new compact examples, but it explains what the nonlinear shooting constructions are realizing: an off-diagonal coupling moves a constrained mode through zero.

\subsection{Hybridization on umbilical four-spheres}

The nonminimal Cherif--Ou example
\[
\Sph^4\left(\frac{\sqrt3}{2}\right)\subset\Sph^5
\]
already exhibits a coupled kernel.  After the natural conjugation associated with the height-function conformal factor, the scalar and exact-vector spherical harmonics form invariant two-channel spaces.  For a scalar harmonic of degree $\ell$, put
\[
x=a^{-2},
\qquad
m_\ell=\ell(\ell+3).
\]
The determinant of the coupled channel matrix is
\begin{equation}\label{eq:umbilical-det}
\det M_\ell
=
m_\ell(m_\ell+2)x^2-8(m_\ell+3)x+32.
\end{equation}
For $\ell=1$,
\[
\det M_1=8(x-1)(3x-4).
\]
Thus the two crossings are the minimal equator $a=1$ and the nonminimal radius
\[
a^2=\frac34.
\]
For $\ell\ge2$, the quadratic discriminant is negative.  Hence the Cherif--Ou radius is the unique nonminimal umbilical hybridization point in this height-function family.  At that radius the kernel consists of one constant normal mode and five degree-one normal--exact-vector modes, for a total dimension six.

\section{Certified compact hybridized geometries}\label{sec:compact-main-theorem}

The constructions below establish the existence statement used in the main theorem.  We record the five certified geometric families separately so that their hypotheses and geometric conclusions can be referenced throughout the proof.

\begin{theorem}[Certified compact hybridized geometries and equal-multiplicity tower]\label{thm:compact-main}
There exist smooth positive nonconstant conformal factors for which each of the following embedded cohomogeneity-one hypersurfaces is a proper biharmonic conformal immersion:
\begin{enumerate}[label=\textup{(\roman*)},leftmargin=2.2em]
\item an $O(3)$-invariant, mean-convex, non-CMC, nowhere-umbilical hypersurface
\[
S^1\times\Sph^2\hookrightarrow\Sph^4;
\]
\item an $O(2)\times O(3)$-invariant, fixed-sign non-CMC, nowhere-umbilical hypersurface
\[
T^2\times\Sph^2\hookrightarrow\Sph^5,
\]
whose meridional curvature is everywhere separated from the two orbit curvatures and whose three principal curvatures are pairwise distinct on an open dense set;
\item an $O(4)$-invariant, mean-convex, non-CMC, nowhere-umbilical hypersurface
\[
S^1\times\Sph^3\hookrightarrow\Sph^5;
\]
\item for every integer $n\geq1$, an $O(n+1)\times O(n+1)$-invariant, fixed-sign non-CMC, nowhere-umbilical hypersurface
\[
S^1\times\Sph^n\times\Sph^n\hookrightarrow\Sph^{2n+2};
\]
\item an $O(3)\times O(4)$-invariant, fixed-sign non-CMC, nowhere-umbilical hypersurface
\[
S^1\times\Sph^2\times\Sph^3\hookrightarrow\Sph^7.
\]
\end{enumerate}
In every case both coefficients of the mixed conformal tension are nontrivial.  In the equal-multiplicity tower, the meridional curvature is everywhere separated from both orbit curvatures, while the orbit curvatures coincide only on the two factor-exchange symmetry leaves.  The $O(4)$ example has rotation number $1/3$ and lies on a real-analytic local branch containing infinitely many exact rationally closed immersions.
\end{theorem}

\section{\texorpdfstring{The first compact hybridized example: $S^1\times\Sph^2$ in $\Sph^4$}{The first compact hybridized example: S1 x S2 in S4}}
\label{sec:m3-certified-example}

We now turn from obstruction to construction.  This first example is the cleanest nonlinear realization of the mixed-block mechanism.  Rotational symmetry reduces the full conformal biharmonic system to a reversible ODE, but both coefficients of the conformal tension remain active.  A half-orbit satisfying the appropriate symmetry conditions can therefore be reflected to a closed profile, and validated shooting proves that such a profile exists.

The formulas below are not an arbitrary dynamical system.  Each variable records a geometric quantity, each ODE is obtained by exact substitution into the mixed block, and the terminal conditions encode smooth closure in the orbit space.

Cohomogeneity-one and rotational reductions for hypersurfaces in space forms have a classical geometric foundation in \cite{HsiangLawson,doCarmoDajczer}.  Closed rotational families obtained by rational angular closure already occur in the minimal theory of spheres, notably in Otsuki's construction \cite{Otsuki}; the equations and conformal-biharmonic closure conditions below are different, but the comparison clarifies the role of the rotation number.

For the $O(3)$-invariant hypersurface
\[
\Phi(s,\omega)=
\bigl(
\cos\chi(s)\cos\vartheta(s),
\cos\chi(s)\sin\vartheta(s),
\sin\chi(s)\,\omega
\bigr),
\qquad \omega\in\Sph^2,
\]
write
\[
\chi'=\cos\alpha,
\qquad
\vartheta'=\frac{\sin\alpha}{\cos\chi},
\qquad
q=\sin\alpha\cot\chi,
\qquad
C=\cos\alpha\cot\chi.
\]
We use the mixed-field variables
\[
L=\lambda^2,
\qquad
F=3LH,
\qquad
y=-L\rho',
\qquad
p=F',
\qquad
r=y'.
\]
Then
\[
H=\frac{F}{3L},
\qquad
k=3H-2q,
\qquad
w=-\frac{y}{L},
\qquad
H'=\frac{pL+2Fy}{3L^2},
\]
and the $m=3$ normal--tangential block is equivalent to
\begin{equation}
\label{eq:m3-mixed-system}
\begin{aligned}
\chi'&=\cos\alpha,\\
\alpha'&=k+\sin\alpha\tan\chi,\\
L'&=-2y,\\
F'&=p,\\
y'&=r,\\
p'&=-2Cp+wp+(k^2+2q^2-3)F
      -2(kr+2qCy)-3yH'+kwy,\\
r'&=2kp+3FH'-Fkw-2Cr+2C^2y+wr+(k^2-2)y,\\
\vartheta'&=\frac{\sin\alpha}{\cos\chi}.
\end{aligned}
\end{equation}
This formulation is preferable for validation because it retains the normal and
meridional components of the conformal tension as primary variables and avoids
quotient differentiation in the shooting equations.

\begin{proposition}[Explicit reversibility and closure mechanism]
\label{prop:m3-reversibility}
For every $\beta\in\R$, define an involution on the full state of
\eqref{eq:m3-mixed-system} by
\begin{equation}\label{eq:m3-reverser}
\mathcal R_\beta
(\chi,\alpha,L,F,y,p,r,\vartheta)
=
(\chi,\pi-\alpha,L,F,-y,-p,r,2\beta-\vartheta).
\end{equation}
If $\mathcal X$ denotes the vector field in
\eqref{eq:m3-mixed-system}, then
\begin{equation}\label{eq:m3-reversibility-identity}
\mathcal X(\mathcal R_\beta x)
=-D\mathcal R_\beta\,\mathcal X(x).
\end{equation}
Its fixed set is
\begin{equation}\label{eq:m3-fixed-set}
\operatorname{Fix}(\mathcal R_\beta)
=
\left\{
\alpha=\frac\pi2,\quad y=0,\quad p=0,\quad
\vartheta=\beta
\right\}.
\end{equation}
Consequently, if a solution starts in
$\operatorname{Fix}(\mathcal R_0)$ and reaches
$\operatorname{Fix}(\mathcal R_\beta)$ at time $T$, then reflection through
its terminal point extends it smoothly to a reduced periodic orbit of period
$2T$.  The reduced geometric and conformal variables return to their initial
values, while the angular variable advances by $2\beta$:
\begin{equation}\label{eq:m3-reflected-return}
(\chi,\alpha,L,F,y,p,r)(2T)
=(\chi,\alpha,L,F,y,p,r)(0),
\qquad
\vartheta(2T)=\vartheta(0)+2\beta.
\end{equation}
\end{proposition}

\begin{proof}
Under \eqref{eq:m3-reverser}, the quantities $q$, $k$, $H$, and $r$ are
even, whereas $C$, $w$, $H'$, $y$, and $p$ are odd.  Substitution into each
line of \eqref{eq:m3-mixed-system} gives
\eqref{eq:m3-reversibility-identity}; for example, the right-hand side of the
$p'$ equation is even and that of the $r'$ equation is odd.  The angular
velocity $\sin\alpha/\cos\chi$ is even, which is exactly compensated by the
reflection $\vartheta\mapsto2\beta-\vartheta$.  This proves reversibility and
\eqref{eq:m3-fixed-set}.

Let $x(s)$ be a solution on $[0,T]$ with
$x(0)\in\operatorname{Fix}(\mathcal R_0)$ and
$x(T)\in\operatorname{Fix}(\mathcal R_\beta)$.  Define, for
$T\le s\le2T$,
\[
\widetilde x(s)=\mathcal R_\beta x(2T-s).
\]
Equation \eqref{eq:m3-reversibility-identity} shows that $\widetilde x$ solves
the same ODE, and the fixed-point condition at $T$ makes the two pieces agree
smoothly.  At $2T$ one has
$\widetilde x(2T)=\mathcal R_\beta x(0)$, which preserves all reduced
variables and changes $\vartheta(0)=0$ to $2\beta$.  This is
\eqref{eq:m3-reflected-return}.
\end{proof}

For
\[
a=(\chi_0,F_0,r_0,T),
\]
rescale $s=T\tau$, $0\leq\tau\leq1$, and impose
\[
x(0)=
(\chi_0,\pi/2,1,F_0,0,0,r_0,0).
\]
The shooting map is
\[
G(a)=
\begin{pmatrix}
\alpha(1)-\pi/2\\
y(1)\\p(1)\\\vartheta(1)-\pi/3
\end{pmatrix}.
\]

\begin{figure}[t]
\centering
\safeincludegraphics[width=.60\textwidth]{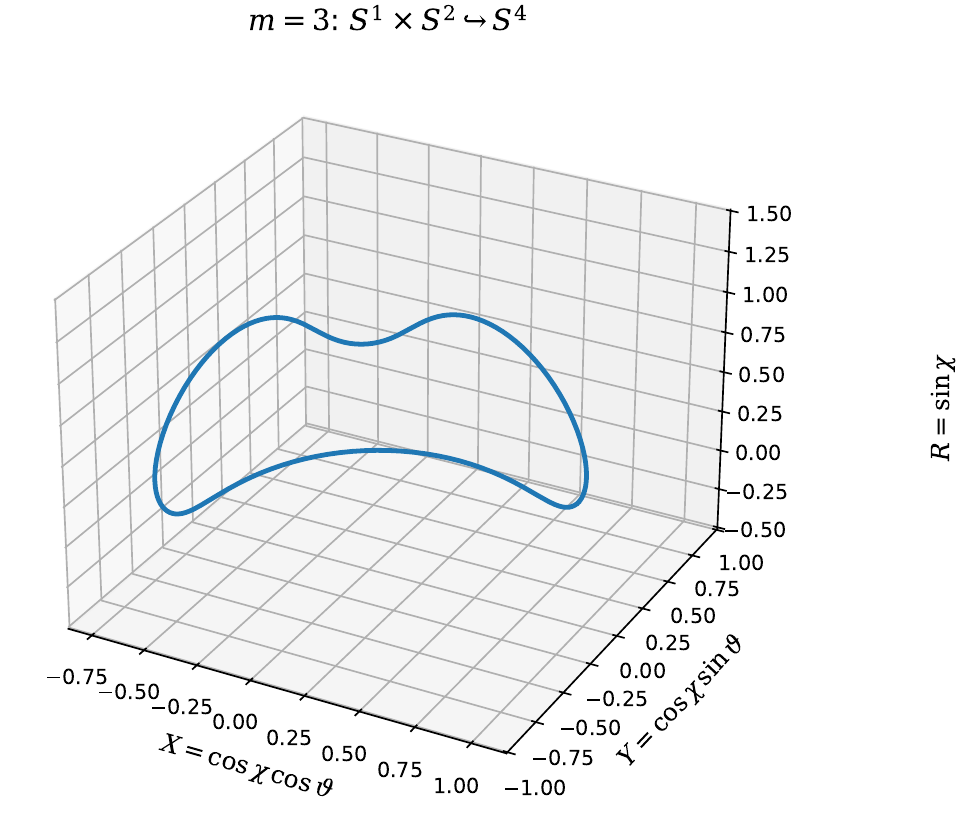}
\caption{Floating-point visualization of the certified $O(3)$-invariant orbit-space profile for $S^1\times\Sph^2\hookrightarrow\Sph^4$.  The three reduced periods give the closed profile.  Existence and embeddedness use the interval bounds in \Cref{thm:m3-certified-root}, not the plotted approximation.}
\label{fig:m3-profile}
\end{figure}

\begin{figure}[t]
\centering
\safeincludegraphics[width=.72\textwidth]{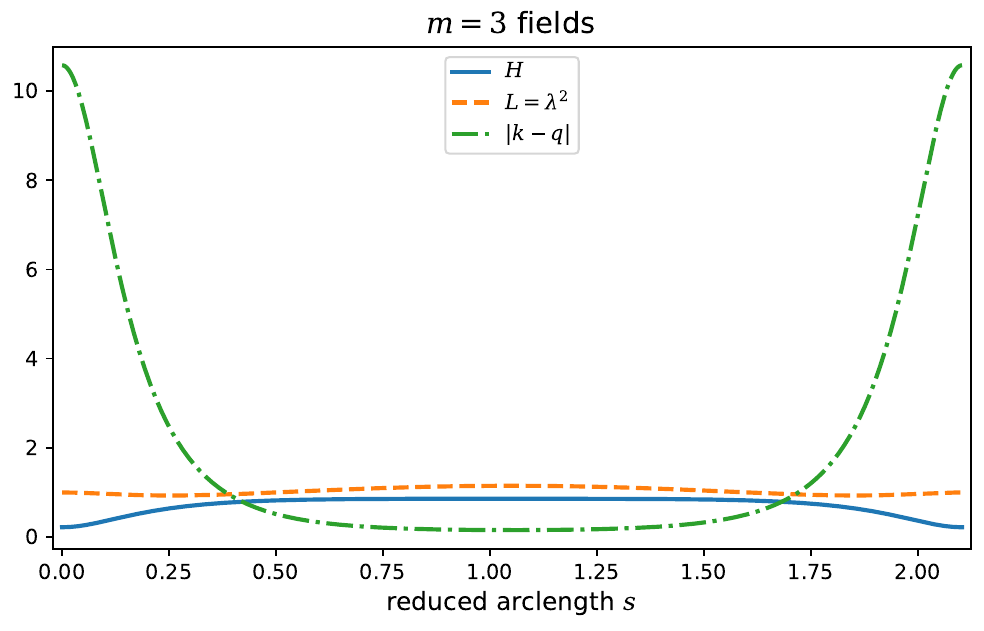}
\caption{Mean curvature, conformal field $L=\lambda^2$, and principal-curvature separation along one reduced period of the certified $S^1\times\Sph^2$ orbit.  Positivity and nonconstancy are certified independently by the interval enclosure in \Cref{thm:m3-certified-root}.}
\label{fig:m3-fields}
\end{figure}

\begin{theorem}[Certified compact $S^1\times\Sph^2$ orbit]
\label{thm:m3-certified-root}
There is a unique zero of the shooting problem in the parameter box centered at
\[
\bar a=
\begin{pmatrix}
0.260851700151849\\
0.661821302770342\\
3.837333122442325\\
1.051412643703885
\end{pmatrix}
\]
with coordinate radius $10^{-7}$.  More sharply, the exact parameters belong to
\begin{align*}
\chi_0&\in
[0.2608517001492214,\,0.2608517001544614],\\
F_0&\in
[0.6618213027552476,\,0.6618213027854251],\\
r_0&\in
[3.8373331223382472,\,3.8373331225466867],\\
T&\in
[1.0514126437018589,\,1.0514126437059290].
\end{align*}
The associated exact half-orbit satisfies
\begin{align*}
0.26083996361726819
&<\chi<
0.83137219763709891,\\
0.93173270742145264
&<L<
1.1508910151145952,\\
0.22054490288608158
&<H<
0.86018692812384312,\\
\sin\alpha&>0.59414896701021536,\\
\vartheta'&>0.65532871175823915,\\
|k-q|&>0.15442503288216392.
\end{align*}
Reflection produces a closed embedded proper biharmonic conformal immersion
\[
S^1\times\Sph^2\hookrightarrow \Sph^4
\]
with positive nonconstant dilation, positive nonconstant mean curvature, and no
umbilical points.
\end{theorem}

\begin{proof}
We use a multiple-shooting version of the interval Taylor--Krawczyk certificate.
Divide the rescaled half-orbit into $16$ equal segments and introduce the eight
state variables at the right endpoint of every segment as additional unknowns.
Together with the four shooting parameters, the resulting system
\[
\mathcal F:\R^{132}\longrightarrow\R^{132}
\]
consists of $16\times8$ continuity equations and the four terminal shooting
equations.

Every parameter and intermediate-state coordinate is enclosed in a box of
radius $10^{-7}$ about the floating-point orbit.  On each shooting segment we
use $16$ equal local steps.  An outward-rounded order-eight interval Taylor
method encloses the state and the nine first variations, corresponding to the
eight initial-state variables and the period parameter $T$.  A Picard-invariant
interval tube is verified on every local step.

Let $C$ be a floating-point inverse of the midpoint Jacobian of the
multiple-shooting system.  For the full interval box $X$, define
\[
K(X)=\bar z-C\mathcal F(\bar z)
 +(I-CD\mathcal F(X))(X-\bar z).
\]
The retained computation gives
\[
K(X)\subset\operatorname{int}X,
\qquad
\max_i
\frac{\operatorname{rad}K_i(X)}{\operatorname{rad}X_i}
<0.001143.
\]
The interval Krawczyk theorem therefore gives a unique exact zero of the full
multiple-shooting system.  Since solutions of the initial-value problem are
unique, its parameter projection is the unique shooting root in the stated
parameter box.  The sharper parameter intervals are the corresponding
coordinate projections of the Krawczyk image.

The orbit-wide inequalities are the outward-rounded unions of the validated
Picard tubes.  They imply $0<\chi<\pi/2$, $L>0$, $H>0$,
$\vartheta'>0$, and $|k-q|>0$.  Thus the profile avoids both singular axes,
is a radial graph in the orbit space, and is nowhere umbilical.  The initial conditions place the orbit in
$\operatorname{Fix}(\mathcal R_0)$, while the terminal equations
$\alpha(T)=\pi/2$, $y(T)=p(T)=0$, and $\vartheta(T)=\pi/3$ place it in
$\operatorname{Fix}(\mathcal R_{\pi/3})$.  By
\Cref{prop:m3-reversibility}, reflection gives a reduced period $2T$ and
angular advance $2\pi/3$.  Repeating this reduced period three times returns
the full state and advances the angular variable by $2\pi$, so the profile
closes.

Finally, $L(0)=1$, whereas the terminal state box gives
\[
L(T)\in
[1.1508787850471776,\,1.1508789850471776],
\]
and
\[
H(T)\in
[0.8599758080429692,\,0.8599760154165270],
\]
which is disjoint from the certified interval for $H(0)=F_0/3$.
Hence both the dilation and the mean curvature are nonconstant.  Since
$3LH>0$, the conformal tension is nonzero, so the biharmonic conformal
immersion is proper.
\end{proof}

\begin{proposition}[Certified genuine coupling]\label{prop:m3-hybridization}
The compact $m=3$ zero mode is not a decoupled normal solution.  Its normal and tangential coefficients are both nontrivial, and at the initial symmetry point the normal off-diagonal contribution is bounded away from zero.
\end{proposition}

\begin{proof}
Write the normal row of the invariant block as
\[
\mathcal D_N(F)+\mathcal C_N(y)=0,
\]
where
\[
\mathcal C_N(y)
=-2(kr+2qCy)-3yH'+kwy.
\]
At the initial point,
\[
\alpha=\frac\pi2,
\qquad L=1,
\qquad y=p=0,
\]
so $C=w=H'=0$ and
\[
q_0=\cot\chi_0,
\qquad
k_0=F_0-2q_0,
\qquad
\mathcal C_N(y)(0)=-2k_0r_0.
\]
Interval evaluation using the certified parameter box in
\Cref{thm:m3-certified-root} gives
\[
52.42315<\mathcal C_N(y)(0)<52.42316.
\]
Therefore the diagonal normal contribution is nonzero and cancels a nonzero off-diagonal contribution.  Moreover $F=3LH>0$ everywhere, while $L(T)\neq L(0)$ implies $y=-L'/2$ is not identically zero.  Thus both components of the mixed conformal tension are nontrivial, and the block kernel is genuinely hybridized.
\end{proof}

\section{\texorpdfstring{Breaking rotational symmetry: $T^2\times\Sph^2$ in $\Sph^5$}{Breaking rotational symmetry: T2 x S2 in S5}}

The preceding rotational example has one orbit curvature of multiplicity two.  To test whether hybridization survives beyond that highly symmetric setting, we next allow two distinct orbit directions.  The resulting $O(2)\times O(3)$ reduction contains three principal curvatures and a more delicate collision geometry.  The certified solution shows that the mixed mechanism persists after symmetry is lowered and that the meridional curvature remains separated from both orbit curvatures.

\label{sec:T2S2-certified}

\subsection{Geometry and exact \texorpdfstring{$O(2)\times O(3)$}{O(2) x O(3)} reduction}

Consider the invariant hypersurface
\begin{equation}\label{eq:T2S2-param}
\Phi(s,\theta,\omega)
=
\bigl(a(s)e^{i\theta},\,b(s)\omega,\,z(s)\bigr)
\in\R^2\oplus\R^3\oplus\R,
\qquad \omega\in\Sph^2,
\end{equation}
where
\[
\gamma=(a,b,z):I\longrightarrow\Sph^2,
\qquad
|\gamma'|=1,
\qquad
a,b>0.
\]
Write
\[
t=\gamma',
\qquad
n=\gamma\times t.
\]
The moving frame satisfies
\begin{equation}\label{eq:T2S2-frame}
\gamma'=t,
\qquad
t'=-\gamma+kn,
\qquad
n'=-kt.
\end{equation}
The induced metric and the three principal curvatures are
\begin{equation}\label{eq:T2S2-metric-curvatures}
g=\dd s^2+a^2\dd\theta^2+b^2g_{\Sph^2},
\qquad
q_1=-\frac{n_a}{a},
\qquad
q_2=-\frac{n_b}{b},
\qquad
H=\frac{k+q_1+2q_2}{4},
\end{equation}
where $q_1$ has multiplicity one and $q_2$ has multiplicity two.

Set
\[
L=\lambda^2,
\qquad
F=4LH,
\qquad
y=-L',
\qquad
p=F',
\qquad
r=y',
\]
and define
\[
A_1=\frac{a'}a,
\qquad
A_2=\frac{b'}b,
\qquad
C=A_1+2A_2,
\qquad
k=\frac FL-q_1-2q_2.
\]
Then the invariant mixed field is
\[
\mathcal T_4=F\xi+y e_s.
\]

\begin{proposition}[Exact lower-symmetry reduction]\label{prop:T2S2-reduction}
The critical four-dimensional block equation
\[
\mathbb B_{4,\rho}\mathcal T_4=0
\]
is equivalent to
\begin{equation}\label{eq:T2S2-system}
\boxed{
\begin{aligned}
\gamma'&=t,\\
t'&=-\gamma+kn,\\
L'&=-y,\\
F'&=p,\\
y'&=r,\\
p'&=-Cp-\frac{2yp}{L}
 +\left(k^2+q_1^2+2q_2^2-4\right)F-2kr\\
&\hspace{1.55cm}
 -2y\left(q_1A_1+2q_2A_2\right)
 -\frac{Fy^2}{L^2}-\frac{ky^2}{L},\\
r'&=-Cr-\frac{yr}{L}
 +\left(A_1^2+2A_2^2+k^2-3\right)y
 +2kp+\frac{Fp}{L}\\
&\hspace{1.55cm}
 +\frac{F^2y}{L^2}+\frac{Fky}{L},
\end{aligned}}
\end{equation}
with $n=\gamma\times t$ and the algebraic definitions above.
\end{proposition}

\begin{proof}
For an invariant scalar $f$ and an invariant tangent field $Y=y e_s$,
\[
\Delta f=f''+Cf',
\qquad
\mathcal C_A(Y)=ky'+(q_1A_1+2q_2A_2)y,
\]
and
\[
\Delta_1(y e_s)
=\bigl(y''+Cy'-(A_1^2+2A_2^2)y\bigr)e_s.
\]
Substitution into the $m=4$ block \eqref{eq:block-m}, followed by
\[
F=4LH,
\qquad
-2L\rho'=-L'=y,
\]
gives the equations for $p'$ and $r'$ in \eqref{eq:T2S2-system}; the remaining equations are the profile-frame and mixed-field definitions.
\end{proof}

\subsection{Reversibility, shooting, and a curvature-collision obstruction}

Set
\[
S=\operatorname{diag}(1,1,-1).
\]
The reversing involution is
\begin{equation}\label{eq:T2S2-reverser}
(\gamma,t,L,F,y,p,r)
\longmapsto
(S\gamma,-St,L,F,-y,-p,r).
\end{equation}
Normalize $L(0)=1$ and use the four shooting parameters
\[
\mathfrak a=(a_0,F_0,r_0,T).
\]
The initial state is
\begin{equation}\label{eq:T2S2-initial}
\gamma(0)=\left(a_0,\sqrt{1-a_0^2},0\right),
\qquad
t(0)=(0,0,1),
\end{equation}
\[
L(0)=1,
\qquad
F(0)=F_0,
\qquad
y(0)=p(0)=0,
\qquad
r(0)=r_0.
\]
The reversible return map is
\begin{equation}\label{eq:T2S2-shooting}
\mathcal F(\mathfrak a)
=
\begin{pmatrix}
z(T)\\
-b(T)a'(T)+a(T)b'(T)\\
y(T)\\
p(T)
\end{pmatrix}.
\end{equation}
A zero whose half-orbit remains in $a,b,L>0$ reflects to a compact hypersurface of topology
\[
S^1_{\rm profile}\times S^1\times\Sph^2
\cong T^2\times\Sph^2.
\]

\begin{proposition}[Unavoidable orbit-curvature collision]\label{prop:T2S2-collision}
Every profile of the form \eqref{eq:T2S2-param} satisfies
\begin{equation}\label{eq:T2S2-collision-identity}
\boxed{
q_1-q_2=-\frac{z'}{ab}.
}
\end{equation}
Consequently, every nonconstant compact profile in the region $a,b>0$ has at least two leaves on which $q_1=q_2$.  In particular, no compact member of this symmetry class can have three pairwise distinct principal curvatures everywhere.
\end{proposition}

\begin{proof}
Since
\[
n_a=bz'-zb',
\qquad
n_b=za'-az',
\]
we have
\[
ab(q_1-q_2)
=-(a^2+b^2)z'+z(aa'+bb').
\]
The identities $a^2+b^2+z^2=1$ and $aa'+bb'+zz'=0$ reduce the right-hand side to $-z'$.  A nonconstant periodic function $z$ has a maximum and a minimum, so $z'=0$ at least twice.
\end{proof}

\subsection{Validated compact proper orbit}

\begin{figure}[t]
\centering
\safeincludegraphics[width=.66\textwidth]{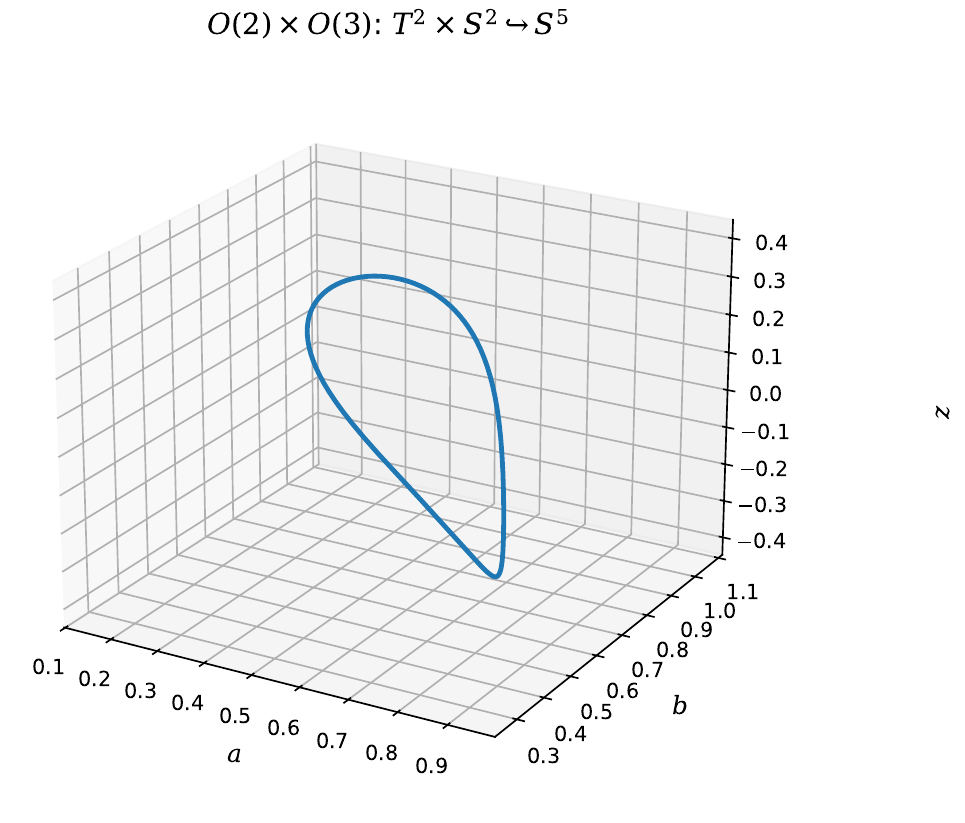}
\caption{Floating-point visualization of the certified $O(2)\times O(3)$ orbit-space profile.  The existence and embeddedness statements use the interval bounds in \Cref{thm:T2S2-certified}, not the plotted approximation.}
\label{fig:T2S2-profile}
\end{figure}

\begin{figure}[t]
\centering
\safeincludegraphics[width=.84\textwidth]{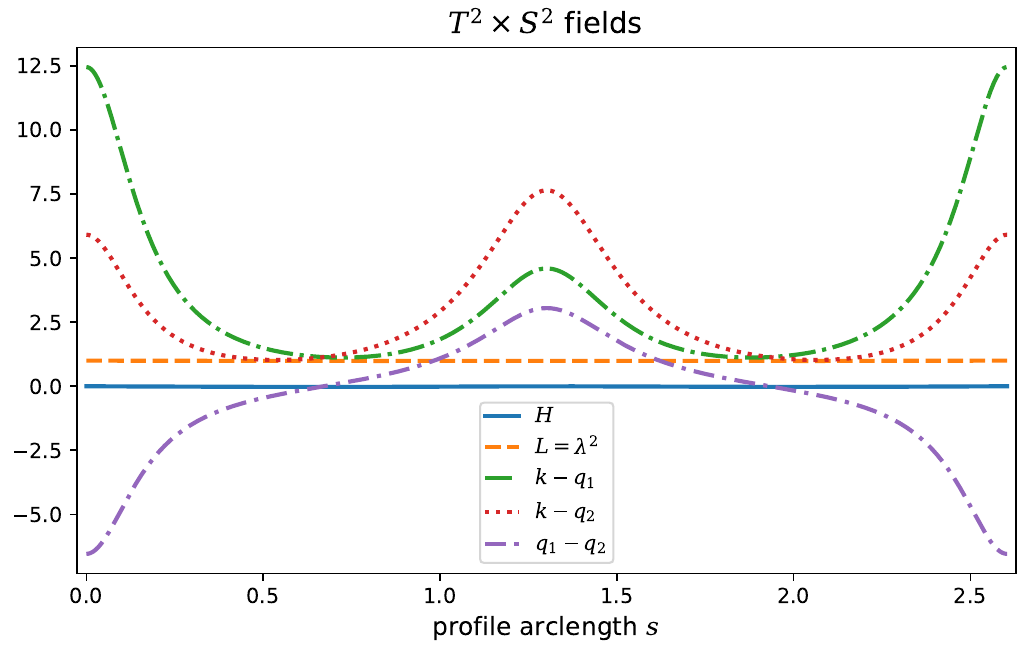}
\caption{Mean curvature, conformal field, and principal-curvature separations along the certified $T^2\times\Sph^2$ orbit.  The crossings of $q_1-q_2$ are forced by \eqref{eq:T2S2-collision-identity}; the meridional curvature remains uniformly separated from both orbit curvatures.}
\label{fig:T2S2-fields}
\end{figure}

\begin{theorem}[Certified compact lower-symmetry hybridized example]
\label{thm:T2S2-certified}
There exists an embedded $O(2)\times O(3)$-invariant hypersurface
\[
\Phi:T^2\times\Sph^2\hookrightarrow\Sph^5
\]
and a smooth positive nonconstant function $\lambda$ such that
\[
\Phi:(T^2\times\Sph^2,\lambda^{-2}\Phi^*h)\longrightarrow\Sph^5
\]
is a proper biharmonic conformal immersion.  Its mean curvature is negative and nonconstant in the chosen orientation, and the hypersurface is nowhere umbilical.  The meridional principal curvature is everywhere distinct from both orbit principal curvatures.  The three principal curvatures are pairwise distinct on an open dense set; equality of the two orbit curvatures occurs on finitely many orbit leaves and is unavoidable for every nonconstant compact profile in this symmetry class.

More precisely, the shooting map \eqref{eq:T2S2-shooting} has a unique exact zero in the parameter intervals
\begin{align*}
0.15469721651450202&<a_0<0.15469721651800192,\\
-0.004031620896272175&<F_0<-0.004031620893771376,\\
0.32369338593613495&<r_0<0.32369338600306274,\\
1.3017720428297248&<T<1.3017720428348816.
\end{align*}
Along the exact half-orbit,
\[
a>0.1546924584,
\qquad
b>0.3491960971,
\]
\[
0.9866737242<L=\lambda^2<1.0003870972,
\]
\[
-0.0265055840<H<-0.0010056668,
\]
and
\[
k-q_1>1.0444939099,
\qquad
k-q_2>0.9521204655.
\]
\end{theorem}

\begin{proof}
Use sixteen multiple-shooting arcs on the normalized half-period and regard the four shooting parameters together with the sixteen endpoint states as a vector in $\R^{180}$.  An order-eight outward-rounded interval Taylor method, with four validated substeps on each arc, encloses the multiple-shooting map and its full Jacobian on boxes of radius $10^{-8}$.  The center residuals are independently enclosed by an order-twelve state-only interval Taylor calculation with sixteen substeps per arc; their maximum absolute endpoint is less than $2.299\times10^{-14}$.

Let $C$ be the retained numerical approximate inverse of the midpoint Jacobian.  The resulting Krawczyk operator satisfies
\[
K(X)\subset\operatorname{int}X,
\qquad
\max_i\frac{\operatorname{rad}K_i(X)}{\operatorname{rad}X_i}
<3.804\times10^{-3}.
\]
Hence the multiple-shooting equations possess a unique exact zero in $X$.

A second outward-rounded integration, initialized on the sharper Krawczyk image, gives the orbit-wide bounds stated above.  On all intermediate validated subarcs,
\[
z>0.0200822016,
\qquad
a b'-b a'<-0.1000970641.
\]
On the first subarc,
\[
z'>0.9921721321,
\qquad
(ab'-ba')'<-5.9104366656,
\]
while on the final subarc,
\[
z'<-0.9894516226,
\qquad
(ab'-ba')'>4.8088622985.
\]
Therefore the upper half-profile is a strict graph over $\operatorname{atan2}(b,a)$; reflection across $z=0$ produces a simple closed profile in the region $a,b>0$.  This proves embeddedness and the topology $T^2\times\Sph^2$.

The bounds $L>0$ and $H<0$ give a positive conformal factor and properness.  The initial enclosure has $r_0>0$, so $L''(0)=-r_0<0$ in the unscaled arclength variable and $\lambda$ is nonconstant.  Moreover,
\[
H(0)\in[-0.0010079052241,-0.0010079052234],
\]
whereas
\[
H(T/2)\in[-0.0264561770003,-0.0264561769943],
\]
so $H$ is nonconstant.  The positive lower bounds for $k-q_1$ and $k-q_2$ prove nowhere-umbilicity.

Finally, \Cref{prop:T2S2-collision} shows that the orbit curvatures must coincide somewhere.  The exact solution is real analytic and $z$ is nonconstant, so the zeros of $z'$ are isolated and hence finite on the compact profile.  Thus all three principal curvatures are distinct on an open dense set.  By \Cref{prop:T2S2-reduction}, the validated orbit solves the full critical four-dimensional conformal biharmonic system.
\end{proof}

\begin{remark}[Sharpness of the three-curvature conclusion]\label{rem:T2S2-sharpness}
The collision leaves do not represent umbilics: the certified lower bounds for $k-q_1$ and $k-q_2$ remain positive there.  Thus \Cref{thm:T2S2-certified} attains the strongest possible curvature-separation statement for compact profiles in the $O(2)\times O(3)$ symmetry class.
\end{remark}

\section{From isolated examples to an infinite product-orbit tower}

A finite list of computer-assisted examples would not by itself show that the mixed mechanism is robust.  This section introduces a two-multiplicity product-orbit ansatz that supports a uniform continuation in the multiplicities.  Equal multiplicities lead to a reversible quarter-return problem and, after compactifying $n$ by $\varepsilon=1/n$, to a single validated tail argument covering all sufficiently large dimensions.  Unequal multiplicities provide a separate test away from factor-exchange symmetry.

\label{sec:product-orbit-tower}

The lower-symmetry construction of \Cref{sec:T2S2-certified} belongs to a dimension-uniform product-orbit system.  This section derives that system once, certifies the first low-multiplicity members and an unequal-multiplicity example, and then proves an infinite equal-multiplicity tower.  Besides adding new topologies, these results show that the mixed mechanism is not tied to the critical dimension $m=4$, to a rotational one-sphere profile, to bounded dimension, or to factor-exchange symmetry.

\paragraph{Visualization convention.}
The orbit and field plots in this section are floating-point reconstructions at the centers of the certified root boxes.  They are included to display the geometry and the coupled channels, but theorem-level existence, closure, embeddedness, positivity, and curvature separation come exclusively from the outward-rounded Taylor--Krawczyk and interval sign-sweep bounds.

\subsection{The universal \texorpdfstring{$O(\mu+1)\times O(\nu+1)$}{O(mu+1) x O(nu+1)} reduction}

Let $\mu,\nu\geq1$ be integers and consider
\begin{equation}\label{eq:product-orbit-param}
\Phi(s,\omega,\eta)
=
\bigl(a(s)\omega,b(s)\eta,z(s)\bigr)
\in\R^{\mu+1}\oplus\R^{\nu+1}\oplus\R,
\end{equation}
where
\[
\omega\in\Sph^\mu,
\qquad
\eta\in\Sph^\nu,
\qquad
\gamma=(a,b,z):I\longrightarrow\Sph^2,
\]
with $|\gamma'|=1$ and $a,b>0$.  Put
\[
t=\gamma',
\qquad
n=\gamma\times t,
\qquad
\gamma'=t,
\qquad
t'=-\gamma+kn.
\]
The induced metric is
\begin{equation}\label{eq:product-orbit-metric}
g=\dd s^2+a^2g_{\Sph^\mu}+b^2g_{\Sph^\nu}.
\end{equation}
The principal curvatures are
\begin{equation}\label{eq:product-orbit-curvatures}
q_1=-\frac{n_a}{a}
\quad\text{with multiplicity }\mu,
\qquad
q_2=-\frac{n_b}{b}
\quad\text{with multiplicity }\nu,
\end{equation}
and the meridional curvature is $k$.  Thus, with
\[
m=\mu+\nu+1,
\]
one has
\begin{equation}\label{eq:product-orbit-H}
H=\frac{k+\mu q_1+\nu q_2}{m}.
\end{equation}

Set
\begin{equation}\label{eq:product-orbit-fields}
L=\lambda^2,
\qquad
F=mLH,
\qquad
y=-\frac{m-2}{2}L',
\qquad
P=F',
\qquad
R=y'.
\end{equation}
Then the invariant conformal tension is
\begin{equation}\label{eq:product-orbit-tension}
\mathcal T_m=F\xi+y e_s.
\end{equation}
Define
\begin{equation}\label{eq:product-orbit-coefficients}
A_1=\frac{a'}a,
\qquad
A_2=\frac{b'}b,
\qquad
C=\mu A_1+\nu A_2,
\end{equation}
\[
S=\mu q_1A_1+\nu q_2A_2,
\qquad
D=\mu A_1^2+\nu A_2^2,
\qquad
k=\frac FL-\mu q_1-\nu q_2.
\]

\begin{proposition}[Exact product-orbit reduction]\label{prop:product-orbit-reduction}
For the ansatz \eqref{eq:product-orbit-param}, the full block equation
\[
\mathbb B_{m,\rho}\mathcal T_m=0
\]
is equivalent to
\begin{equation}\label{eq:product-orbit-system}
\boxed{
\begin{aligned}
\gamma'&=t,\\
t'&=-\gamma+kn,\\
L'&=-\frac{2y}{m-2},\\
F'&=P,\\
y'&=R,\\
P'&=-CP-\frac{2yP}{L}
 +\left(k^2+\mu q_1^2+\nu q_2^2-m\right)F
 -2kR-2Sy\\
&\hspace{1.3cm}
 -\frac{2Fy^2}{(m-2)L^2}-\frac{ky^2}{L},\\
R'&=-CR-\frac{yR}{L}
 +\left(D+k^2-(m-1)\right)y
 +2kP+\frac{FP}{L}\\
&\hspace{1.3cm}
 +\frac{2F^2y}{(m-2)L^2}+\frac{Fky}{L}.
\end{aligned}}
\end{equation}
In particular, the previously certified $T^2\times\Sph^2$ system is the specialization $(\mu,\nu)=(1,2)$.
\end{proposition}

\begin{proof}
For an invariant scalar $f$ and an invariant tangent field $Y=y e_s$,
\[
\Delta f=f''+Cf',
\qquad
\mathcal C_A(Y)=ky'+Sy,
\]
and
\[
\Delta_1(y e_s)=\bigl(y''+Cy'-Dy\bigr)e_s.
\]
Substituting these identities and \eqref{eq:product-orbit-tension} into the normal and tangential rows of \eqref{eq:block-m}, and then using \eqref{eq:product-orbit-fields}--\eqref{eq:product-orbit-coefficients}, gives the $P'$ and $R'$ equations.  The remaining equations are the profile-frame and mixed-field definitions.
\end{proof}

\begin{remark}[Equivalence to the Ou--Cherif equations]\label{rem:product-orbit-equivalence}
Equation \eqref{eq:product-orbit-system} is not an alternative field equation and does not generate solutions outside the Ou--Cherif theory.  It is the exact restriction of that theory to the geometrically admissible mixed field $F\xi+y e_s$.  The Dirac/Clifford organization is useful because it identifies this coupled field before symmetry reduction and turns the search for solutions into a constrained kernel and return-map problem.  Thus the examples below are solutions of Ou's equations; the operator formulation explains why the normal and tangential channels must be solved together and supplies the architecture used to locate and certify them.
\end{remark}

Every product-orbit profile satisfies the identity
\begin{equation}\label{eq:product-orbit-collision}
\boxed{ab(q_1-q_2)=-z'.}
\end{equation}
It is independent of the multiplicities.  Hence the orbit curvatures coincide at extrema of $z$, while a uniform positive lower bound for $k-q_1$ and $k-q_2$ excludes umbilics.

\subsection{Equal multiplicities and quarter-return closure}

When $\mu=\nu$, the system has two reversers: equatorial reflection $z\mapsto-z$ and exchange of the two sphere factors $a\leftrightarrow b$.  Starting from
\begin{equation}\label{eq:product-quarter-initial}
\gamma(0)=\left(a_0,\sqrt{1-a_0^2},0\right),
\qquad
t(0)=(0,0,1),
\end{equation}
\[
L(0)=1,
\qquad
F(0)=F_0,
\qquad
y(0)=P(0)=0,
\qquad
R(0)=R_0,
\]
we use the quarter-return map
\begin{equation}\label{eq:product-quarter-map}
Q(a_0,F_0,R_0,\tau)
=
\begin{pmatrix}
a(\tau)-b(\tau)\\
t_z(\tau)\\
y(\tau)\\
P(\tau)
\end{pmatrix}.
\end{equation}
A zero joins the fixed sets of the two reversers.  Four reflected quarter-arcs produce a smooth closed profile.

\begin{theorem}[Certified $T^3\hookrightarrow\Sph^4$ product orbit]
\label{thm:T3-certified}
For $(\mu,\nu)=(1,1)$, the map \eqref{eq:product-quarter-map} has a unique exact zero in
\begin{align*}
0.16451667294421632&<a_0<0.16451667334421630,\\
-0.003763955169303500&<F_0<-0.003763954769303500,\\
0.26274235380056649&<R_0<0.26274235420056652,\\
0.73748280235429675&<\tau<0.73748280275429678.
\end{align*}
The four reflected quarter-arcs form a simple closed profile in $a,b>0$, and the lift is an embedded hypersurface
\[
T^3\hookrightarrow\Sph^4.
\]
Along the exact orbit,
\[
0.9760435828<L\leq1,
\qquad
-0.0308773833<H<-0.0012323215,
\]
and
\[
k-q_1>0.8795996061,
\qquad
k-q_2>0.8204594540.
\]
The mean curvature and dilation are nonconstant, the hypersurface is nowhere umbilical, and both components of $\mathcal T_3=F\xi+y e_s$ are nontrivial.  The orbit curvatures coincide only on the two factor-exchange symmetry leaves and are distinct elsewhere.
\end{theorem}

\begin{proof}
An outward-rounded order-fourteen interval Taylor integration propagates the state and its four parameter derivatives on the root box.  The resulting four-dimensional Krawczyk operator satisfies
\[
K(X)\subset\operatorname{int}X,
\]
so the zero is exact and locally unique.  A separate $256$-step interval sweep proves $H<0$, the two curvature-gap inequalities, $t_z>0$ in the open fundamental quarter, and $a<b$ until the factor-exchange endpoint.  The first quarter is therefore injective, its exchanged image lies in the disjoint chamber $a>b$, and equatorial reflection separates the upper and lower halves.  Positivity of $a$, $b$, and $L$ prevents orbit collapse and proves embeddedness.  The nonzero intervals for $L''(0)$, $H''(0)$, and $-2k(0)R_0$ certify nonconstancy and genuine hybridization.
\end{proof}

\begin{figure}[tbp]
\centering
\begin{minipage}[c]{0.43\textwidth}
  \centering
  \safeincludegraphics[width=\linewidth]{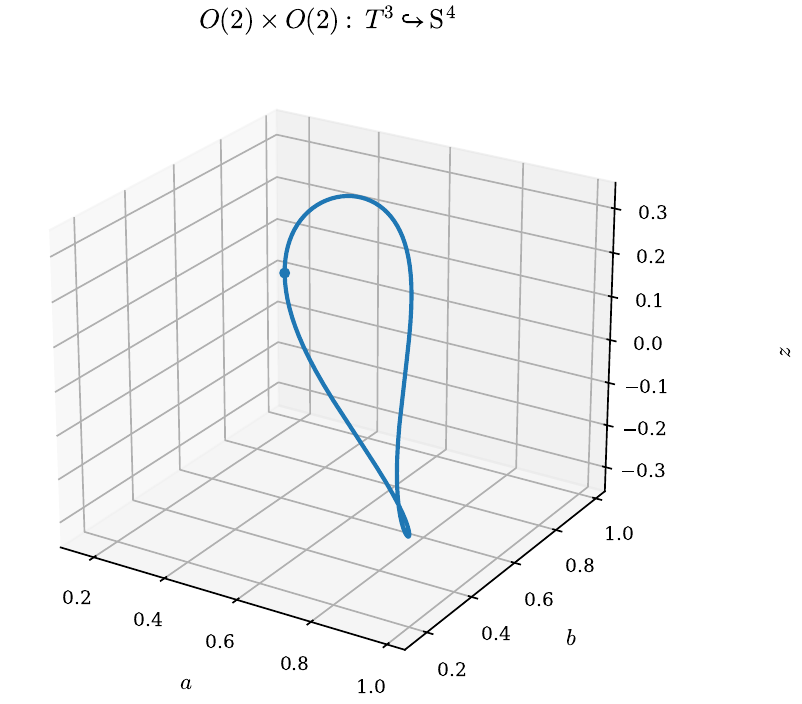}
\end{minipage}\hfill
\begin{minipage}[c]{0.54\textwidth}
  \centering
  \safeincludegraphics[width=\linewidth]{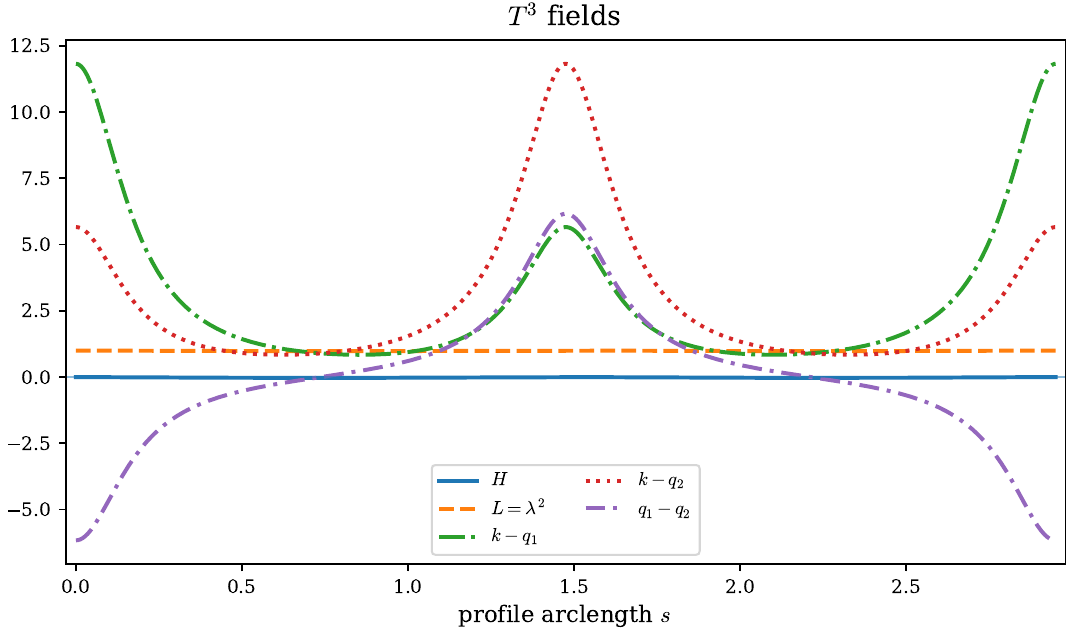}
\end{minipage}
\caption{Floating-point reconstruction of the certified $O(2)\times O(2)$ product orbit in \Cref{thm:T3-certified}.  Left: four reflected quarter-arcs close the simple orbit-space profile for $T^3\hookrightarrow\Sph^4$.  Right: the dilation $L=\lambda^2$, mean curvature $H$, and the three principal-curvature differences.  The crossings of $q_1-q_2$ occur at the factor-exchange leaves, whereas $k-q_1$ and $k-q_2$ remain positive.  The interval proof, rather than the plotted approximation, establishes all global inequalities.}
\label{fig:t3-product-visualization}
\end{figure}

\begin{theorem}[Certified $S^1\times\Sph^2\times\Sph^2\hookrightarrow\Sph^6$ product orbit]
\label{thm:S1S2S2-certified}
For $(\mu,\nu)=(2,2)$, the map \eqref{eq:product-quarter-map} has a unique exact zero in
\begin{align*}
0.31716996583188489&<a_0<0.31716996623188493,\\
-0.0060149805327804999&<F_0<-0.0060149801327804998,\\
0.36300638042873401&<R_0<0.36300638082873404,\\
0.58545607007243672&<\tau<0.58545607047243675.
\end{align*}
The four reflected quarter-arcs lift to an embedded proper BCI
\[
S^1\times\Sph^2\times\Sph^2\hookrightarrow\Sph^6.
\]
Along the exact orbit,
\[
L>0.9907632994,
\qquad
-0.0225645776<H<-0.0011925436,
\]
and
\[
k-q_1>1.3215621430,
\qquad
k-q_2>1.2652121629.
\]
The mean curvature and dilation are nonconstant, the mixed field is genuinely hybridized, and the hypersurface is nowhere umbilical.  The orbit curvatures coincide only on the two factor-exchange symmetry leaves.
\end{theorem}

\begin{proof}
The proof repeats the quarter-return Taylor--Krawczyk construction of \Cref{thm:T3-certified} with $m=5$.  The strict inclusion $K(X)\subset\operatorname{int}X$ proves a unique exact root, while a separate interval sweep proves positivity, $H<0$, the curvature-gap inequalities, strict vertical monotonicity on the open quarter, and strict chamber separation $a<b$ before the swap endpoint.  The four reflected arcs therefore form a simple closed regular profile.  At the initial leaf the intervals for $L''(0)$, $H''(0)$, and $-2k(0)R_0$ exclude zero.
\end{proof}

\begin{figure}[tbp]
\centering
\begin{minipage}[c]{0.43\textwidth}
  \centering
  \safeincludegraphics[width=\linewidth]{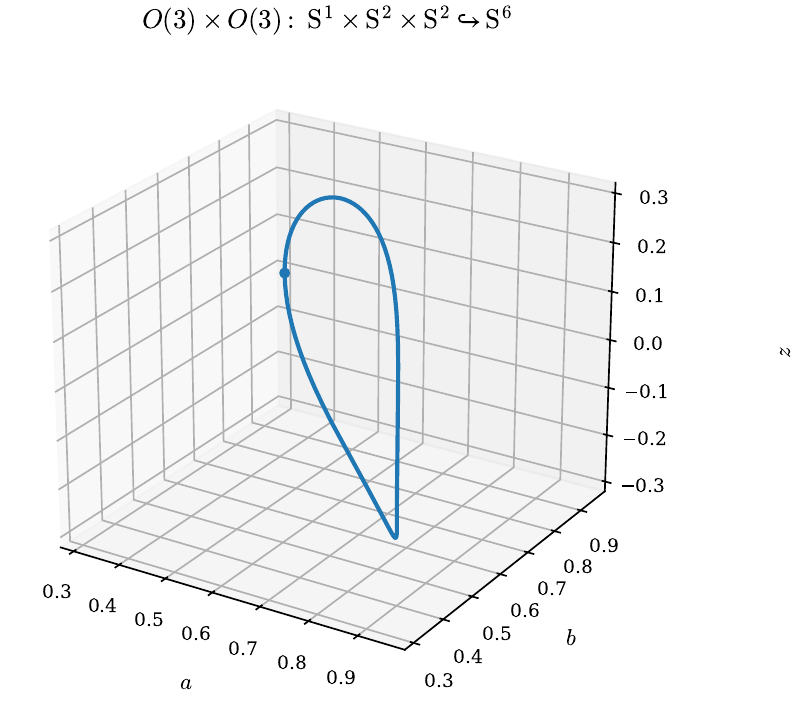}
\end{minipage}\hfill
\begin{minipage}[c]{0.54\textwidth}
  \centering
  \safeincludegraphics[width=\linewidth]{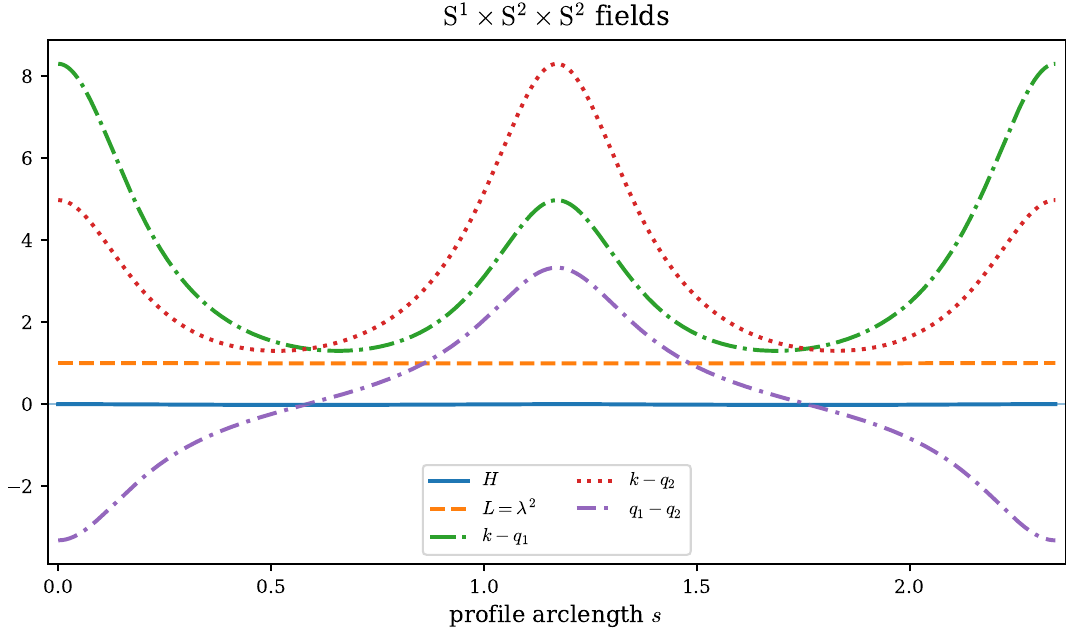}
\end{minipage}
\caption{Floating-point reconstruction of the certified equal-multiplicity $(2,2)$ product orbit in \Cref{thm:S1S2S2-certified}.  Left: the four reflected quarter-arcs for $S^1\times\Sph^2\times\Sph^2\hookrightarrow\Sph^6$.  Right: $L$, $H$, and the principal-curvature differences.  The orbit curvatures agree only at the factor-exchange leaves, while the meridional curvature remains separated from both.}
\label{fig:s1s2s2-product-visualization}
\end{figure}

\begin{theorem}[Certified $S^1\times\Sph^3\times\Sph^3\hookrightarrow\Sph^8$ product orbit]
\label{thm:S1S3S3-certified}
For $(\mu,\nu)=(3,3)$, the quarter-return map
\[
Q(a_0,F_0,R_0,\tau)
=
\begin{pmatrix}
a(\tau)-b(\tau)\\
t_z(\tau)\\
y(\tau)\\
P(\tau)
\end{pmatrix}
\]
has a unique exact zero in
\begin{align*}
0.39089509875840311
&<a_0<
0.39089509876195017,\\
-0.0051079148613970820
&<F_0<
-0.0051079148591732185,\\
0.40027046980545677
&<R_0<
0.40027046999447524,\\
0.49649930434609213
&<\tau<
0.49649930434868372.
\end{align*}
The four reflected quarter-arcs form a simple closed regular profile and lift to an embedded proper biharmonic conformal immersion
\[
\boxed{
S^1\times\Sph^3\times\Sph^3\hookrightarrow\Sph^8.
}
\]
Along the exact closed orbit,
\[
L>0.9952694267,
\qquad
-0.0152298165<H<-0.0007233455,
\]
and
\[
k-q_1>1.5778503057,
\qquad
k-q_2>1.5778503057.
\]
The conformal factor and mean curvature are nonconstant, both the normal and tangential coefficients of the conformal-tension field are nontrivial, and the hypersurface is nowhere umbilical.  The two orbit curvatures coincide only on the two factor-exchange symmetry leaves; away from those leaves the three principal curvatures are pairwise distinct.
\end{theorem}

\begin{proof}
For $(\mu,\nu)=(3,3)$ one has $m=7$.  By \Cref{prop:product-orbit-reduction}, the full Ou--Cherif conformal-hypersurface equation restricted to the product-orbit ansatz is exactly equivalent to \eqref{eq:product-orbit-system}, with
\[
C=3(A_1+A_2),
\qquad
S=3(q_1A_1+q_2A_2),
\qquad
D=3(A_1^2+A_2^2),
\]
\[
k=\frac FL-3q_1-3q_2,
\qquad
L'=-\frac25y,
\]
and the corresponding $m=7$ coefficients in the $P'$ and $R'$ equations.  Thus no projected or auxiliary field equation is introduced.

Use the initial data
\[
\gamma(0)=\bigl(a_0,\sqrt{1-a_0^2},0\bigr),
\qquad
t(0)=(0,0,1),
\]
\[
L(0)=1,
\qquad
F(0)=F_0,
\qquad
y(0)=P(0)=0,
\qquad
R(0)=R_0.
\]
After the rescaling $s=\tau u$, $0\leq u\leq1$, an outward-rounded order-fourteen interval Taylor method propagates the twelve-state system and its four shooting sensitivities through $128$ equal steps.  On the parameter box of radius $2\times10^{-10}$ about
\[
(0.39089509876017653,
 -0.005107914860285176,
 0.40027046989997056,
 0.4964993043473880),
\]
the validated midpoint Jacobian has determinant
\[
\det DQ_{\rm mid}=-0.13510441052001329.
\]
The resulting Krawczyk image is strictly contained in the interior of the parameter box in all four coordinates; its coordinate intervals are exactly those displayed in the theorem.  The interval Newton--Krawczyk theorem therefore gives a unique exact zero of $Q$.  Since the initial point lies in the fixed set of equatorial reflection and the endpoint lies in the fixed set of factor exchange, the four reversible quarter-arcs join smoothly to a closed profile of period $4\tau$.

A separate $256$-step order-fourteen interval Taylor sweep, initialized with the root-containing Krawczyk box, gives the global bounds
\[
a\geq0.3908498693,
\qquad
b\geq0.6850531186,
\qquad
L\geq0.9952694267
\]
on the fundamental quarter.  It also proves
\[
-0.0152298165\leq H\leq-0.0007233455,
\]
\[
k-q_1\geq1.6293907818,
\qquad
k-q_2\geq1.5778503057.
\]
Factor exchange swaps the two curvature-gap bounds, so the smaller lower bound applies to both gaps on the full reflected orbit.

For embeddedness, the same sweep proves
\[
t_z>0
\quad\text{and}\quad
a<b
\qquad (0\leq u<1).
\]
More precisely, on all but the last Taylor tube,
\[
t_z\geq0.0020327061,
\qquad
a-b\leq-0.0017828601,
\]
while on the last tube
\[
\frac{d t_z}{du}\leq-0.7823946674,
\qquad
\frac{d(a-b)}{du}\geq0.7010054751.
\]
Together with the exact endpoint conditions $t_z(1)=0$ and $a(1)=b(1)$, these inequalities prove strict vertical monotonicity and strict chamber separation throughout the open quarter.  Hence the first quarter is injective, its factor-exchanged image lies in the disjoint chamber $a>b$, and equatorial reflection separates the upper and lower halves.  Since $a$, $b$, and $L$ stay positive, no orbit collapses and the lift is embedded with topology $S^1\times\Sph^3\times\Sph^3$.

The fixed negative sign of $H$ makes the immersion nonminimal and hence proper.  At the initial reversing leaf, interval evaluation gives
\[
L''(0)\in
[-0.16010818803998825,-0.16010818787998818],
\]
\[
H''(0)\in
[-0.69353873152061229,-0.69353872704032093],
\]
so both the dilation and mean curvature are nonconstant.  Moreover,
\[
F=7LH<0
\]
everywhere, while $y'(0)=R_0>0$, so both the normal and tangential coefficients of
\(
\mathcal T_7=F\xi+y e_s
\)
are nontrivial.  The initial off-diagonal coupling satisfies
\[
-2k(0)R_0
\in
[-4.6310453087767911,-4.6310452957662465],
\]
which excludes decoupling of the two channels.

Finally, the strict inequalities $k-q_1>0$ and $k-q_2>0$ exclude umbilics everywhere.  The exact moving-frame identity
\[
ab(q_1-q_2)=-z'=-t_z
\]
shows that $q_1\neq q_2$ on the open quarter and that equality occurs only at the factor-exchange endpoint and its reflected partner.  This proves all assertions.
\end{proof}

\begin{figure}[tbp]
\centering
\begin{minipage}[c]{0.43\textwidth}
  \centering
  \safeincludegraphics[width=\linewidth]{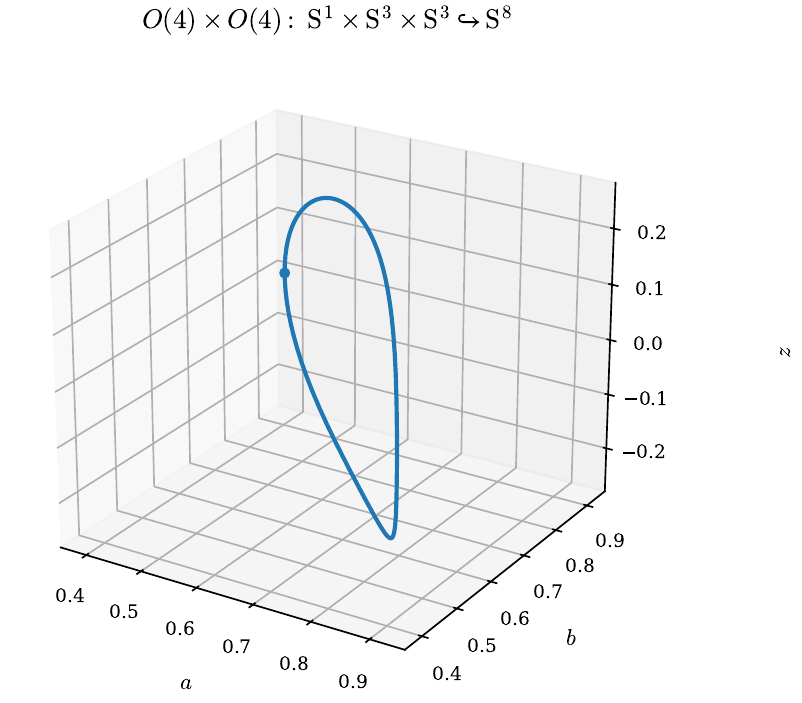}
\end{minipage}\hfill
\begin{minipage}[c]{0.54\textwidth}
  \centering
  \safeincludegraphics[width=\linewidth]{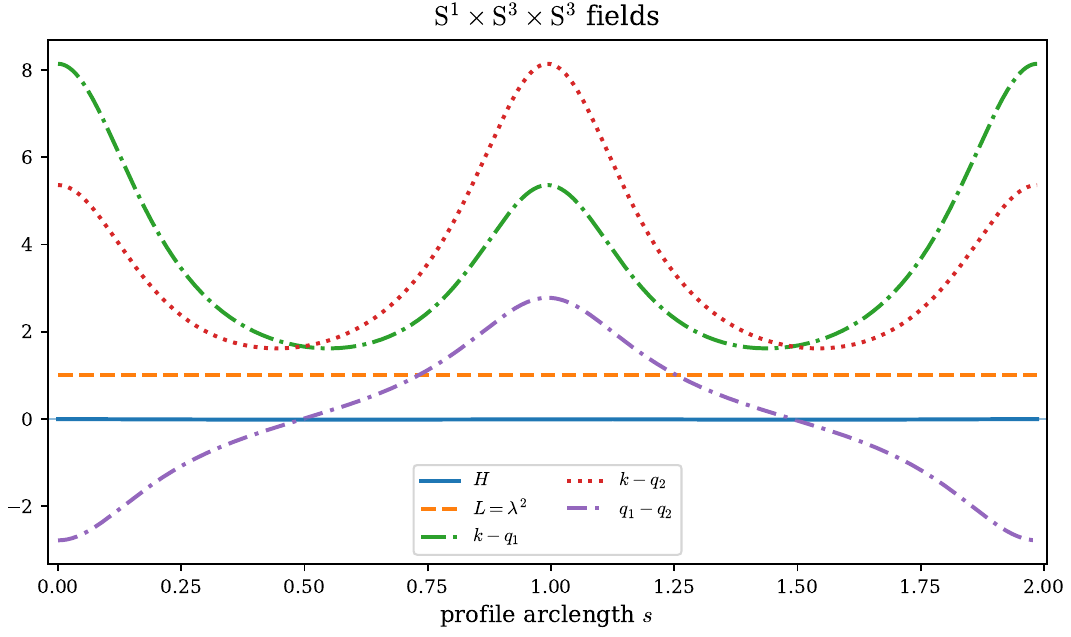}
\end{minipage}
\caption{Floating-point reconstruction of the certified equal-multiplicity $(3,3)$ product orbit in \Cref{thm:S1S3S3-certified}.  Left: the four reflected quarter-arcs for $S^1\times\Sph^3\times\Sph^3\hookrightarrow\Sph^8$.  Right: the positive dilation $L$, fixed-sign nonconstant mean curvature $H$, and the three principal-curvature differences.  The interval proof, rather than the plotted approximation, establishes closure, embeddedness, positivity, and the global curvature gaps.}
\label{fig:s1s3s3-product-visualization}
\end{figure}

\subsection{Large-multiplicity compactification}
\label{subsec:equal-tower-large-n}

The equal-multiplicity branch approaches the Clifford point
\(
a=b=2^{-1/2}
\)
as $n\to\infty$.  The apparent singularity of the original product-orbit coefficients is removed by a parabolic rescaling.  Put
\[
\varepsilon=\frac1n,
\qquad
\delta=\sqrt{\varepsilon},
\qquad
x=\frac{s}{\delta},
\]
and use the oriented orthonormal basis
\[
e_0=\frac1{\sqrt2}(1,1,0),
\qquad
e_1=\frac1{\sqrt2}(-1,1,0),
\qquad
e_2=(0,0,1).
\]
Write
\begin{equation}\label{eq:equal-tower-scaled-profile}
\gamma
=
Ue_0+\delta\chi e_1+\delta\zeta e_2,
\qquad
U=\sqrt{1-\varepsilon(\chi^2+\zeta^2)},
\end{equation}
and set
\[
u=\dot\chi,
\qquad
v=\dot\zeta,
\qquad
L=1+\varepsilon^2\ell,
\qquad
F=\delta f,
\qquad
y=\delta g,
\qquad
P=p,
\qquad
R=r,
\]
where a dot denotes $d/dx$.  Define
\[
\Delta_0=1-\varepsilon(2\chi^2+\zeta^2),
\qquad
\mathcal A=\chi u+\zeta v,
\qquad
\mathcal B=\chi v-\zeta u,
\]
\[
M=Uv+\varepsilon\frac{\zeta\mathcal A}{U},
\qquad
J=\frac{\mathcal A}{U},
\]
\begin{align*}
Q_2&=
\frac{2\{(M^2+\varepsilon\mathcal B^2)(U^2+\varepsilon\chi^2)
+4\varepsilon M\mathcal B U\chi\}}{\Delta_0^2},\\
D_2&=
\frac{2\{(u^2+\varepsilon J^2)(U^2+\varepsilon\chi^2)
+4\varepsilon uJU\chi\}}{\Delta_0^2},\\
S_2&=
\frac{2(Mu+\varepsilon\mathcal B J)(U^2+\varepsilon\chi^2)
+4\varepsilon(MJ+\mathcal B u)U\chi}{\Delta_0^2},\\
c&=-\frac{2(2\chi u+\zeta v)}{\Delta_0},\\
K&=\varepsilon\frac fL+
\frac{2\{U(2\chi v-\zeta u)+\varepsilon\chi\zeta\mathcal A/U\}}
{\Delta_0},\\
N_1&=-Uv-\varepsilon\frac{\zeta\mathcal A}{U},
\qquad
N_2=Uu+\varepsilon\frac{\chi\mathcal A}{U}.
\end{align*}
Direct substitution into \eqref{eq:product-orbit-system} gives the exact regular system
\begin{equation}\label{eq:equal-tower-regular-system}
\begin{aligned}
\dot\chi&=u,&
\dot\zeta&=v,\\
\dot u&=-\varepsilon\chi+KN_1,&
\dot v&=-\varepsilon\zeta+KN_2,\\
\dot\ell&=-\frac{2g}{2-\varepsilon},&
\dot f&=p,\\
\dot g&=r,\\
\dot p&=-cp-\frac{2\varepsilon gp}{L}
 +(K^2+Q_2-2-\varepsilon)f-2Kr-2S_2g\\
&\quad-\frac{2\varepsilon^3fg^2}{(2-\varepsilon)L^2}
 -\frac{\varepsilon Kg^2}{L},\\
\dot r&=-cr-\frac{\varepsilon gr}{L}
 +(D_2+K^2-2)g+2Kp+\frac{\varepsilon fp}{L}\\
&\quad+\frac{2\varepsilon^3f^2g}{(2-\varepsilon)L^2}
 +\frac{\varepsilon fKg}{L}.
\end{aligned}
\end{equation}
Thus the large-$n$ singularity is removable.

At $\varepsilon=0$, the profile subsystem is
\begin{equation}\label{eq:equal-tower-limit-geometry}
\dot\chi=u,
\qquad
\dot\zeta=v,
\qquad
\dot u=-\kappa v,
\qquad
\dot v=\kappa u,
\qquad
\kappa=4\chi v-2\zeta u.
\end{equation}
The leading mixed equations possess the exact solution
\begin{equation}\label{eq:equal-tower-leading-mixed}
\boxed{
f=R\zeta u,
\qquad
g=R\zeta v,
\qquad
p=\frac{d}{dx}(R\zeta u),
\qquad
r=\frac{d}{dx}(R\zeta v).
}
\end{equation}
In particular, the leading geometry closes before the mixed amplitude is selected; the first correction in $\varepsilon$ removes this rank loss.

Let
\[
\chi(0)=\alpha+\varepsilon A,
\qquad
\zeta(0)=u(0)=0,
\qquad
v(0)=1,
\]
\[
\ell(0)=0,
\qquad
f(0)=\varepsilon\phi,
\qquad
g(0)=p(0)=0,
\qquad
r(0)=R,
\]
and let the quarter time be $T+\varepsilon B$.  Dividing the endpoint conditions
\(
(\chi,v,g,p)=0
\)
by $\varepsilon$ produces an analytic return map at $\varepsilon=0$.

\begin{proposition}[Certified compactified root]\label{prop:equal-tower-limit-root}
The limiting profile and first-correction return problem have a unique exact root in
\begin{align*}
0.68280035920548254&<\alpha_*<0.68280035960548258,\\
0.92684538005145478&<T_*<0.92684538045145481,\\
-0.098133992017761712&<A_*<-0.098133988017761714,\\
-0.20558228720832272&<B_*<-0.20558228320832272,\\
-0.063722384880642399&<\phi_*<-0.063722380880642401,\\
0.50593638674595620&<R_*<0.50593639074595631.
\end{align*}
The midpoint determinant of the six-dimensional desingularized return Jacobian is
\[
0.18664477749690586.
\]
Moreover, the limiting quarter profile satisfies
\[
\chi>0,\qquad
\zeta>0,\qquad
u<0,\qquad
v>0,
\qquad
\kappa>0.9682395521
\]
on its open interior.
\end{proposition}

\begin{proof}
An outward-rounded interval Taylor integration of the limiting profile and first-correction variational equations, followed by a six-dimensional Krawczyk inclusion, gives the unique root.  A separate interval sign sweep proves the displayed profile and curvature inequalities and the required terminal transversality.
\end{proof}

\begin{figure}[tbp]
\centering
\begin{minipage}[c]{0.47\textwidth}
\centering
\safeincludegraphics[width=\linewidth]{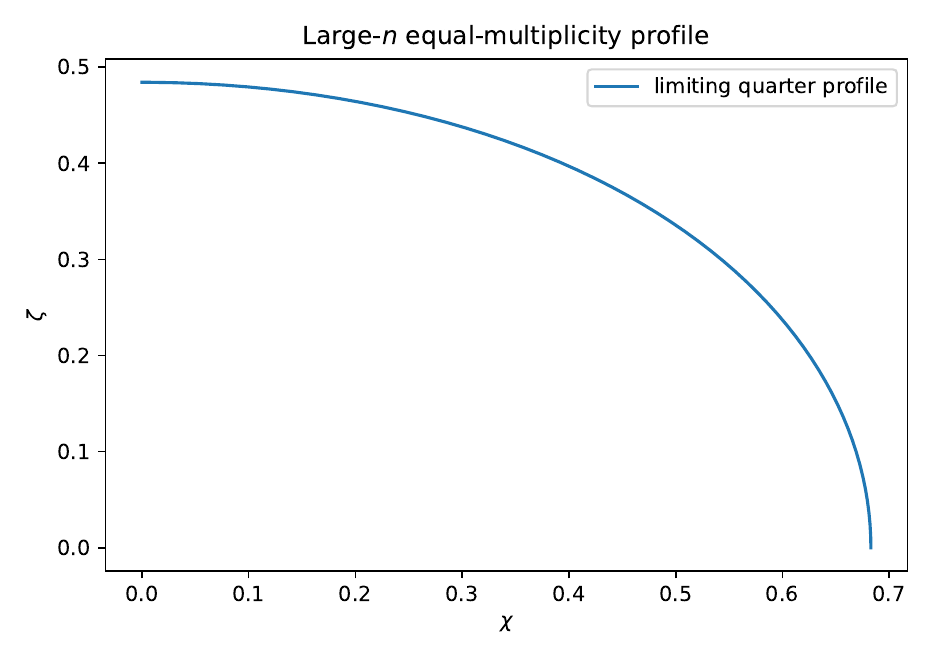}
\end{minipage}\hfill
\begin{minipage}[c]{0.49\textwidth}
\centering
\safeincludegraphics[width=\linewidth]{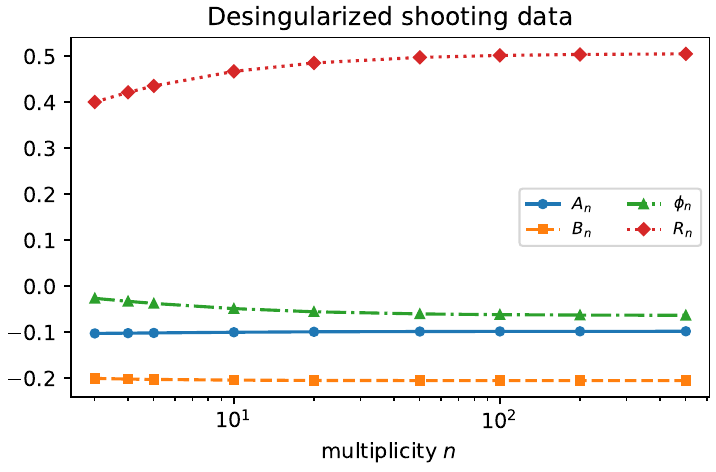}
\end{minipage}
\caption{Large-multiplicity compactification of the equal-multiplicity product-orbit branch.  Left: the certified limiting quarter profile for \eqref{eq:equal-tower-limit-geometry}.  Right: floating-point convergence of the desingularized shooting data toward the certified limiting constants.  The convergence plot is diagnostic; the limiting root and the all-$n$ existence theorem are established by the retained Krawczyk and radii-polynomial certificates.}
\label{fig:equal-tower-limit}
\end{figure}

\subsection{Finite certificates and the all-integer tower}
\label{subsec:equal-tower-all-n}

For $1\leq n\leq20$, fixed-multiplicity outward-rounded Taylor--Krawczyk calculations certify the quarter-return roots and separate interval sweeps certify the global geometry.  The cases $n=1,2,3$ are displayed above; the same desingularized verifier certifies $n=4,\ldots,20$.  Representative margins are:
\begin{table}[ht]
\centering
\scriptsize
\begin{tabular}{c@{\quad}r@{\quad}r@{\quad}r@{\quad}r@{\quad}r}
\toprule
$n$ & $\det D\widehat Q$ & $\min L$ & $\sup f$ &
$\min(K-\sqrt\varepsilon q_1)$ &
$\min(K-\sqrt\varepsilon q_2)$\\
\midrule
6  & $-0.151496$ & $0.998621$ & $-0.006755$ & $0.964187$ & $0.944497$\\
10 & $-0.147002$ & $0.999469$ & $-0.004801$ & $0.965746$ & $0.952556$\\
15 & $-0.144910$ & $0.999756$ & $-0.003471$ & $0.966478$ & $0.956894$\\
20 & $-0.143905$ & $0.999860$ & $-0.002635$ & $0.966599$ & $0.957915$\\
\bottomrule
\end{tabular}
\caption{Representative fixed-$n$ certificate margins in the desingularized variables.  The full ancillary table records every integer $5\leq n\leq20$.}
\label{tab:equal-tower-finite-margins}
\end{table}

To cover the infinite tail, normalize the quarter interval to $u\in[0,1]$ and write
\[
Y(u,\varepsilon)=Y_0(u;R(\varepsilon))+\varepsilon Z(u,\varepsilon).
\]
The divided return problem is analytic on
\[
(u,\varepsilon)\in[0,1]\times[0,1/20].
\]
A tensor Chebyshev approximation of degree $40$ in $u$ and degree $6$ in $\varepsilon$, together with degree $10$ for the shooting branch, gives the validated radii constants
\[
Y_{\rm rad}<1.360\times10^{-6},
\qquad
Z_0<7.818\times10^{-3},
\qquad
H_Q<1.0451\times10^4,
\qquad
C<24.
\]
Thus the radii polynomial
\begin{equation}\label{eq:equal-tower-radii-poly}
p_{\rm rad}(r)
=
Y_{\rm rad}+(Z_0-1)r+\frac{CH_Q}{2}r^2
\end{equation}
satisfies the directed-rounding enclosure
\[
p_{\rm rad}(3\times10^{-6})
\subset
[-4.883621161982128,-4.883621161982057]\times10^{-7}<0.
\]

\begin{theorem}[Equal-multiplicity tower]\label{thm:equal-multiplicity-tower}
For every integer $n\geq1$, there exists a smooth positive nonconstant conformal factor for which an embedded $O(n+1)\times O(n+1)$-invariant hypersurface
\[
S^1\times\Sph^n\times\Sph^n\hookrightarrow\Sph^{2n+2}
\]
is a proper biharmonic conformal immersion.  The mean curvature has fixed sign and is nonconstant, the hypersurface is nowhere umbilical, and both the normal and tangential coefficients of its conformal-tension field are nontrivial.
\end{theorem}

\begin{proof}
The fixed-multiplicity certificates prove the assertion for $1\leq n\leq20$.  For $\varepsilon\in[0,1/20]$, negativity of \eqref{eq:equal-tower-radii-poly} gives a unique exact reversible quarter orbit in the radius-$3\times10^{-6}$ tube around the tensor Chebyshev approximation.

After inflation by the full radii bound, a separate uniform sweep proves
\[
L>0.9998579967,
\qquad
K-\sqrt\varepsilon q_1>0.9583764135,
\qquad
K-\sqrt\varepsilon q_2>0.9530930755,
\]
and keeps both orbit radii positive.  The shooting parameters satisfy
\[
R>0.4852660836,
\qquad
\phi<-0.0557271908.
\]
Interior and endpoint-aware bounds give
\[
u<0,\qquad
v>0,\qquad
p<0
\]
on the open quarter.  Consequently the profile stays in one factor chamber, is simple, and closes by factor exchange and equatorial reflection.  Since $f(0)=\varepsilon\phi<0$ and $f'=Tp<0$, the mean curvature has fixed negative sign and is nonconstant.  Since $r(0)=R>0$, the tangential mixed coefficient is nontrivial.  The two positive curvature-gap bounds prove nowhere-umbilicity.  Taking $\varepsilon=1/n$ proves the result for every $n\geq20$, and the finite certificates complete the proof.
\end{proof}

\begin{remark}[Asymptotic shooting data]\label{rem:equal-tower-asymptotics}
The tower satisfies
\[
\chi_n(0)=\alpha_*+\frac{A_*}{n}+O(n^{-2}),
\qquad
\sqrt n\,\tau_n=T_*+\frac{B_*}{n}+O(n^{-2}),
\]
\[
n^{3/2}F_n(0)=\phi_*+O(n^{-1}),
\qquad
R_n(0)=R_*+O(n^{-1}),
\]
with the constants of \Cref{prop:equal-tower-limit-root}.
\end{remark}

\begin{remark}[Certificate implementation]\label{rem:equal-tower-implementation}
The limiting Krawczyk calculation and the final scalar radii-polynomial inequality use directed-rounding interval arithmetic.  The tensor defect and range layer is certified independently by exact rational Chebyshev restriction together with a formal binary64 roundoff audit for every matrix reduction, coefficient convolution, truncation, and Taylor-model operation.  The raw audited bounds are retained and automatically compared with the rounded constants entering the radii polynomial.
\end{remark}

\subsection{Unequal multiplicities and half-return closure}

If $\mu\neq\nu$, factor exchange is no longer a symmetry, but equatorial reversibility remains.  We retain the initial data \eqref{eq:product-quarter-initial} and use the half-return map
\begin{equation}\label{eq:product-half-map}
G(a_0,F_0,R_0,T)
=
\begin{pmatrix}
z(T)\\
-b(T)t_a(T)+a(T)t_b(T)\\
y(T)\\
P(T)
\end{pmatrix}.
\end{equation}
At a zero, the frame constraints imply $t_a(T)=t_b(T)=0$, so the endpoint lies in the equatorial reversing set.  Equatorial reflection then closes the profile with period $2T$.

\begin{theorem}[Certified unequal-multiplicity product orbit]
\label{thm:S1S2S3-certified}
For $(\mu,\nu)=(2,3)$, the half-return map \eqref{eq:product-half-map} has a unique exact zero in
\begin{align*}
0.28706160111521378&<a_0<0.28706160117544527,\\
-0.0048371228555237709&<F_0<-0.0048371227712868609,\\
0.38605117882890916&<R_0<0.38605118013824291,\\
1.0716396164052124&<T<1.0716396164723341.
\end{align*}
The reflected orbit-space profile is simple and lifts to an embedded proper BCI
\[
S^1\times\Sph^2\times\Sph^3\hookrightarrow\Sph^7.
\]
Along the exact half orbit,
\[
a>0.2870581039,
\qquad
b>0.4283690387,
\qquad
L>0.9935798354,
\]
\[
-0.0182477513<H<-0.0008056646,
\]
and
\[
k-q_1>1.4714515325,
\qquad
k-q_2>1.4313874002.
\]
The mean curvature and dilation are nonconstant, both normal and tangential mixed coefficients are nontrivial, and the hypersurface is nowhere umbilical.
\end{theorem}

\begin{proof}
Continuation from the certified $(2,2)$ solution reaches $(2,3)$ without a numerical fold.  For the endpoint, an outward-rounded order-twelve interval Taylor integration with $128$ steps encloses both the center residual and the full $60$-variable state-and-sensitivity system on a uniform shooting box of radius $7\times10^{-10}$.  The resulting Krawczyk map satisfies
\[
K(X)\subset\operatorname{int}X,
\]
with midpoint determinant
\[
\det\operatorname{mid}DG(X)>4.38.
\]
An endpoint-aware $2048$-tube interval sign cover proves the stated curvature and positivity bounds.  To prove embeddedness, put
\[
N_\phi=a t_b-b t_a,
\qquad
\phi=\operatorname{atan2}(b,a).
\]
The sign cover consists of a forward derivative region with $z'>0$ and $N_\phi'<0$, a direct middle region with $z>0$ and $N_\phi<0$, and a backward derivative region with $z'<0$ and $N_\phi'>0$.  Since the exact return equations give $z=N_\phi=0$ at the reversing endpoints, it follows that $N_\phi<0$ and $z>0$ throughout the open half orbit, hence $\phi'<0$.  The two equatorial endpoints are distinct.  Thus the upper profile is a strict angular graph, its equatorial reflection lies in $z<0$, and the closed profile is simple.  The intervals for $L''(0)$, $H''(0)$, and $-2k(0)R_0$ exclude zero and certify nonconstancy and genuine hybridization.
\end{proof}

\begin{figure}[tbp]
\centering
\begin{minipage}[c]{0.43\textwidth}
  \centering
  \safeincludegraphics[width=\linewidth]{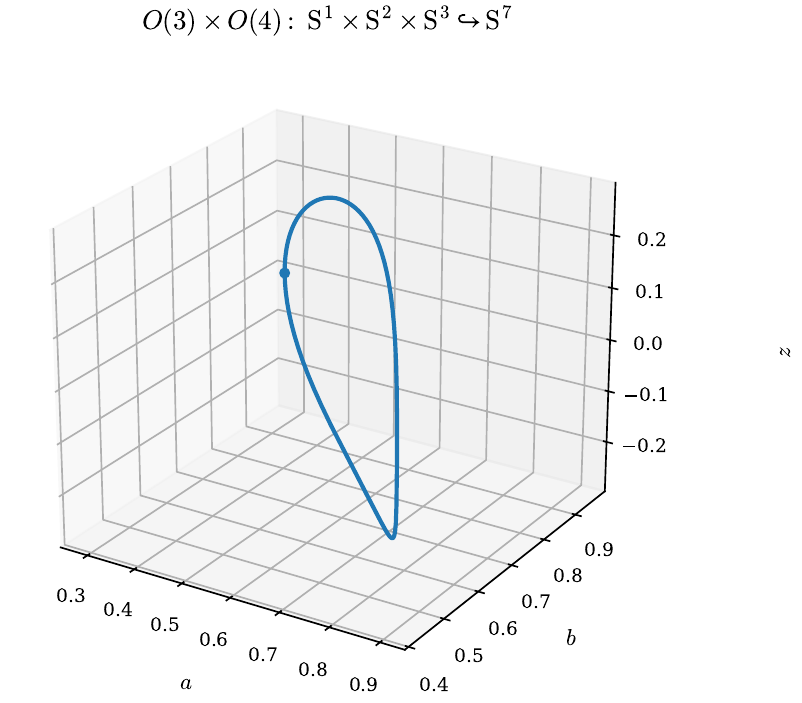}
\end{minipage}\hfill
\begin{minipage}[c]{0.54\textwidth}
  \centering
  \safeincludegraphics[width=\linewidth]{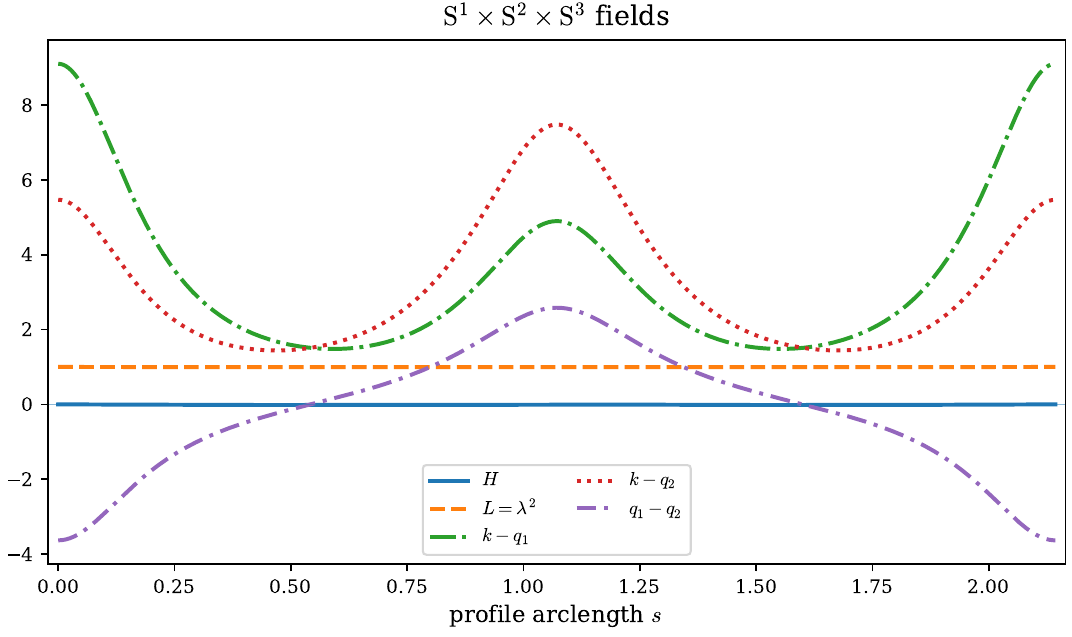}
\end{minipage}
\caption{Floating-point reconstruction of the certified unequal-multiplicity $(2,3)$ product orbit in \Cref{thm:S1S2S3-certified}.  Left: equatorial reflection closes the two half-orbits for $S^1\times\Sph^2\times\Sph^3\hookrightarrow\Sph^7$.  Right: $L$, $H$, and the principal-curvature differences.  The plot displays persistence of the compact mixed mechanism after factor-exchange symmetry is lost; the rigorous bounds are supplied by the half-return certificate and the $2048$-step interval sweep.}
\label{fig:s1s2s3-product-visualization}
\end{figure}

\begin{corollary}[A certified product-orbit hierarchy]\label{cor:product-orbit-hierarchy}
The full Ou--Cherif system admits compact embedded hybridized product-orbit solutions for every diagonal multiplicity pair
\[
(\mu,\nu)=(n,n),
\qquad n\geq1,
\]
and for the unequal pairs
\[
(\mu,\nu)=(1,2),\ (2,3).
\]
Thus equal multiplicity produces an infinite dimension-indexed tower, while neither equal multiplicity nor factor-exchange symmetry is necessary for compact non-CMC normal--tangential hybridization.
\end{corollary}

\section{Rotation number and a local ladder of closed hypersurfaces}

The $O(4)$-rotational problem supplies a complementary global parameter: the angular rotation number.  Instead of proving only one isolated shooting root, we validate the variational equation along the certified orbit and show that the rotation number changes monotonically on a real-analytic local branch.  Rational values on this branch yield infinitely many exactly closed immersions, while the certified $1/3$ member is the embedded representative singled out by the orbit-space geometry.

\subsection{Geometry and exact reduction}

Consider
\begin{equation}\label{eq:O4-param}
\Phi(s,\omega)
=
\left(
\cos\chi(s)\cos\vartheta(s),
\cos\chi(s)\sin\vartheta(s),
\sin\chi(s)\,\omega
\right),
\qquad \omega\in\Sph^3.
\end{equation}
Choose profile arclength and write
\[
\chi'=\cos\alpha,
\qquad
\vartheta'=\frac{\sin\alpha}{\cos\chi}.
\]
The induced metric is
\[
g=\dd s^2+\sin^2\chi\,g_{\Sph^3}.
\]
The principal curvatures are
\[
k=\alpha'-\sin\alpha\tan\chi
\]
in the meridional direction and
\[
q=\sin\alpha\cot\chi
\]
with multiplicity three.  Hence
\[
H=\frac{k+3q}{4}.
\]
Set
\[
P=H',
\qquad
w=\rho',
\qquad
Q=\Delta\rho,
\qquad
C=\cos\alpha\cot\chi.
\]
Then \eqref{eq:N4}--\eqref{eq:T4-tangent} are exactly equivalent to
\begin{equation}\label{eq:O4-system}
\boxed{
\begin{aligned}
\chi'&=\cos\alpha,\\
\alpha'&=k+\sin\alpha\tan\chi,\\
H'&=P,\\
w'&=Q-3Cw,\\
P'&=-3CP+H(k^2+3q^2)-4H
 +(k-2H)Q+3Cw(q-k)+kw^2,\\
Q'&=-4(k+2H)P
 -2(3+4Hk-k^2)w-4Qw+6Cw^2,
\end{aligned}}
\end{equation}
where
\[
q=\sin\alpha\cot\chi,
\qquad
k=4H-3q.
\]
The remaining fields satisfy
\[
\rho'=w,
\qquad
\vartheta'=\frac{\sin\alpha}{\cos\chi}.
\]

The system is reversible under
\begin{equation}\label{eq:reverser}
\mathcal R(\chi,\alpha,H,P,w,Q)
=(\chi,\pi-\alpha,H,-P,-w,Q).
\end{equation}
Its fixed set is
\[
\alpha=\frac\pi2,
\qquad P=0,
\qquad w=0.
\]

\subsection{A local hybridized germ}

Take
\[
\chi(0)=\frac\pi4,
\quad
\alpha(0)=\frac\pi2,
\quad
H(0)=\frac7{10},
\quad
P(0)=\frac15,
\quad
w(0)=\frac3{10},
\quad
Q(0)=\frac1{10}.
\]
Then
\[
q(0)=1,
\qquad
k(0)=-\frac15,
\]
so the hypersurface is nonumbilical.  It is non-CMC and has nonconstant dilation because $P(0)\neq0$ and $w(0)\neq0$.  Direct substitution into \eqref{eq:N4} and \eqref{eq:T4-tangent} gives equal and opposite diagonal and coupling contributions.  Thus neither channel vanishes separately, while the full block does.  Analytic ODE existence produces an exact local proper biharmonic conformal hypersurface.

\subsection{The closed orbit}

For a parameter vector
\[
a=(\chi_0,H_0,Q_0,T),
\]
start on the reverser fixed set:
\[
(\chi,\alpha,H,P,w,Q,\vartheta)(0)
=(\chi_0,\pi/2,H_0,0,0,Q_0,0).
\]
Define the shooting map
\begin{equation}\label{eq:shooting-map}
F(a)=
\begin{pmatrix}
\alpha(T)-\pi/2\\
P(T)\\
w(T)\\
2\vartheta(T)-2\pi/3
\end{pmatrix}.
\end{equation}
A zero of $F$ gives a reversible periodic reduced orbit with angular advance $2\pi/3$.  Three reduced periods close the profile.

\begin{figure}[t]
\centering
\safeincludegraphics[width=.66\textwidth]{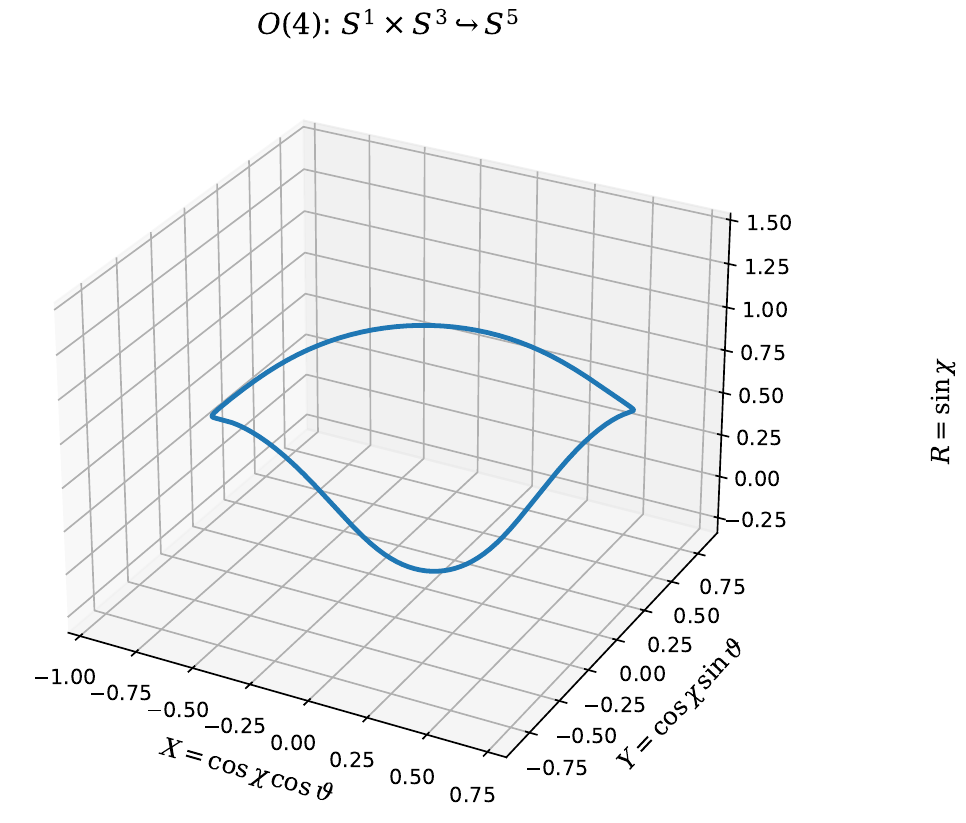}
\caption{Floating-point visualization of the certified three-lobed profile in the rotational orbit space.  The existence and embeddedness statements use the interval bounds in \Cref{sec:certificate}, not the plotted approximation.}
\label{fig:profile}
\end{figure}

\begin{figure}[t]
\centering
\safeincludegraphics[width=.84\textwidth]{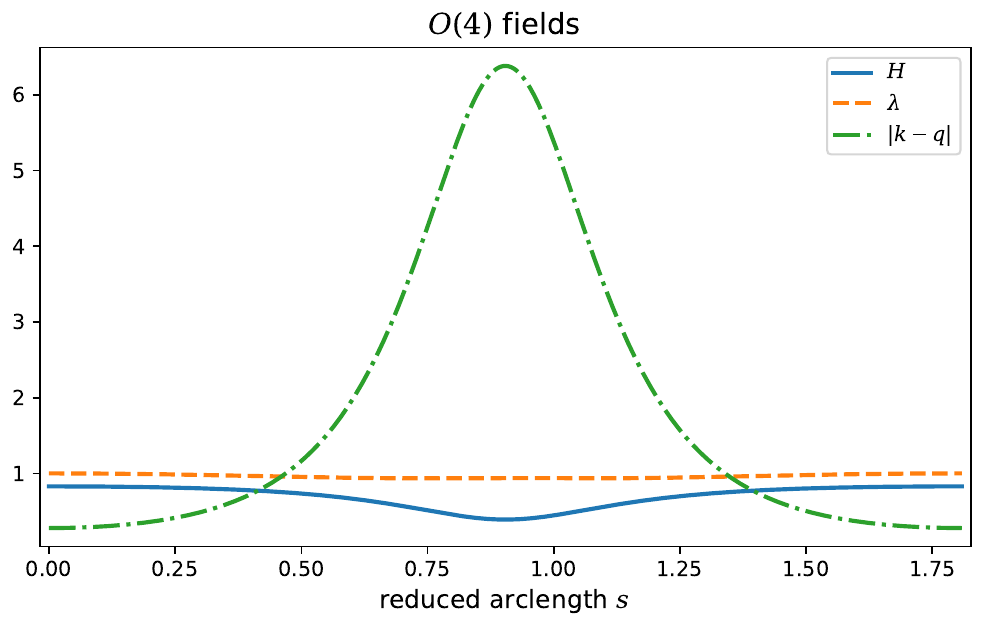}
\caption{Geometric and conformal fields over one reduced period of the closed orbit.  The mean curvature and conformal factor remain positive, while the principal curvatures are distinct.}
\label{fig:fields}
\end{figure}

\section{Computer-assisted existence and validation}\label{sec:certificate}

The preceding sections derived exact geometric ODEs and stated interval-certified solutions.  We now collect the common validation architecture in one place.  The distinction between numerical exploration and proof is essential: floating-point integration locates a candidate, but only an outward-rounded enclosure can establish that an exact orbit exists nearby and satisfies the global inequalities required for geometric reconstruction.

The workflow has four stages.  First, interval Taylor integration encloses the orbit and its variational equations.  Second, a Krawczyk operator proves existence and local uniqueness of the shooting root.  Third, interval sweeps over the validated orbit certify positivity, embeddedness, and curvature separation.  Fourth, for the all-integer tower, a tensor Chebyshev representation and a radii-polynomial estimate replace infinitely many separate shooting computations by one uniform tail certificate.

We use a finite-dimensional shooting formulation, rigorous Taylor propagation of the flow and its variational equation, and an interval Krawczyk inclusion.  The underlying existence-and-uniqueness test is the interval Newton--Krawczyk principle \cite{MooreKrawczyk}; the rigorous high-order propagation is in the spirit of the $C^r$-Lohner algorithms and their CAPD implementation \cite{WilczakZgliczynski,CAPDDynSys}.

Let
\begin{equation}\label{eq:center-box}
\bar a=
\begin{pmatrix}
0.8391986606188933\\
0.8285225801714209\\
-0.48622995457118967\\
0.9041656858600123
\end{pmatrix},
\qquad
X=\bar a+[-10^{-6},10^{-6}]^4.
\end{equation}

\begin{theorem}[Certified shooting root]\label{thm:certified-root}
The shooting map \eqref{eq:shooting-map} has a unique zero $a_*\in X$.
\end{theorem}

\begin{proof}
The half-orbit was rescaled to $\tau\in[0,1]$ and divided into $1024$ equal binary subintervals.  Outward-rounded Boost.Interval arithmetic was used throughout.  An order-eight interval Taylor method enclosed the state, while an order-six interval Taylor method enclosed the state and variational equations.  A Picard-invariant interval tube was verified on every step.

Let $C$ be a floating-point approximate inverse of $DF(\bar a)$.  The Krawczyk operator is
\[
K(X)=\bar a-CF(\bar a)+(I-CDF(X))(X-\bar a).
\]
The retained interval computation gives
\begin{align*}
K_1-\bar a_1&\subset[-7.8908,7.8908]\times10^{-8},\\
K_2-\bar a_2&\subset[-1.1677,1.1677]\times10^{-7},\\
K_3-\bar a_3&\subset[-1.5947,1.5947]\times10^{-7},\\
K_4-\bar a_4&\subset[-7.4793,7.4793]\times10^{-8}.
\end{align*}
The maximum coordinate radius is less than $0.160$ times the radius of $X$.  Therefore
\[
K(X)\subset\operatorname{int}X.
\]
The Krawczyk theorem yields a unique exact zero $a_*\in X$.
\end{proof}

\begin{proposition}[Exact closure and geometry]\label{prop:closure-geometry}
The exact orbit corresponding to $a_*$ has reduced period $2T_*$ and angular advance $2\pi/3$.  After three reduced periods it closes to an embedded hypersurface of topology $S^1\times S^3$.  Along the orbit,
\[
0.46656<\chi<0.83920,
\qquad
H>0.38823,
\]
\[
\sin\alpha>0.77874,
\qquad
\vartheta'>0.94026,
\qquad
|k-q|>0.27707.
\]
Moreover $H$ and $\lambda=e^\rho$ are nonconstant.
\end{proposition}

\begin{proof}
Both endpoints of the exact half-orbit belong to $\operatorname{Fix}(\mathcal R)$.  Reversibility gives
\[
x(2T_*)=x(0).
\]
Since $w$ is odd under the reflection,
\[
\int_0^{2T_*}w\dd s=0,
\]
so $\rho$ and $\lambda$ are periodic.  The fourth shooting condition gives angular advance $2\pi/3$ per reduced period and hence $2\pi$ after three periods.

The stated inequalities are the outward-rounded orbit-wide enclosures from the same validated integration.  In particular the profile avoids both rotational axes and $\vartheta$ is strictly increasing, so the closed profile is a radial graph and is embedded.  The bound $|k-q|>0$ proves nowhere-umbilicity.  Finally,
\[
P'(0)\in[-0.487546,-0.487482],
\qquad
w'(0)=Q_0\in[-0.486231,-0.486228],
\]
so $H$ and $\lambda$ are nonconstant.
\end{proof}

\begin{proof}[Proof of \Cref{thm:compact-main}]
Part \textup{(i)} is \Cref{thm:m3-certified-root}: the validated $O(3)$-rotational half-orbit closes by reversibility to an embedded compact hypersurface $S^1\times\Sph^2\hookrightarrow\Sph^4$.  The mixed-field system \eqref{eq:m3-mixed-system} is the invariant reduction of the $m=3$ normal--tangential block, so the associated conformal immersion is biharmonic.  The certified bounds $L>0$ and $H>0$ imply that the conformal factor is positive and the tension is nonzero, while the endpoint separation proves that both $L$ and $H$ are nonconstant.

Part \textup{(ii)} is \Cref{thm:T2S2-certified}.  The $180$-dimensional Krawczyk inclusion gives a unique exact reversible return, and the orbit-wide interval inequalities prove that its reflection is an embedded $T^2\times\Sph^2\hookrightarrow\Sph^5$.  The lower-symmetry system \eqref{eq:T2S2-system} is equivalent to the critical $m=4$ block, while $L>0$ and $H<0$ prove positivity of the conformal factor and properness.  The strict bounds for $k-q_1$ and $k-q_2$ prove nowhere-umbilicity.

Part \textup{(iii)} follows from \Cref{thm:certified-root,prop:closure-geometry}.  The $O(4)$-rotational profile closes exactly and gives an embedded compact hypersurface $S^1\times\Sph^3\hookrightarrow\Sph^5$.  The rotational ODE is algebraically equivalent to \eqref{eq:N4}--\eqref{eq:T4-tangent}; therefore the conformal immersion is biharmonic.  The certified positive nonconstant mean curvature gives properness, and the validated variational enclosure yields the rational-rotation statement.

Part \textup{(iv)} is \Cref{thm:equal-multiplicity-tower}.  The cases $1\leq n\leq20$ follow from fixed-multiplicity Taylor--Krawczyk certificates and interval sign sweeps, including \Cref{thm:T3-certified,thm:S1S2S2-certified,thm:S1S3S3-certified}.  The compactified Chebyshev/radii-polynomial certificate proves the same properties uniformly for $\varepsilon=1/n\in[0,1/20]$.  Proposition \ref{prop:product-orbit-reduction} identifies every member with the full Ou--Cherif block.

Part \textup{(v)} is \Cref{thm:S1S2S3-certified}.  The mixed-precision half-return Krawczyk inclusion proves a unique exact equatorial return for multiplicities $(2,3)$, and the $2048$-step interval sweep proves angular monotonicity, separation of the reflected halves, positivity, fixed-sign mean curvature, and the two curvature-gap inequalities.  Thus the lift is an embedded $S^1\times\Sph^2\times\Sph^3\hookrightarrow\Sph^7$ proper BCI.

In every case positivity of the conformal factor and compactness imply completeness of the induced and conformally changed metrics.  The certified nonzero normal and tangential coefficients show that every compact kernel lies genuinely in the coupled normal--tangential block.
\end{proof}

\subsection{Direct verification against the Ou--Cherif equations}
\label{subsec:direct-ou-verification}

\begin{proposition}[Independent equation-level verification]
\label{prop:compact-direct-ou-check}
For every compact geometry in \Cref{thm:compact-main}, the reconstructed pair $(\Phi,\lambda)$ satisfies the original conformal-hypersurface biharmonic equation
\begin{equation}\label{eq:compact-direct-ou}
\left(J_g^\Phi+(m-2)\nabla^\Phi_{\grad\rho}\right)
\left[\lambda^2\bigl(mH\xi-(m-2)\grad\rho\bigr)\right]=0,
\qquad \rho=\log\lambda,
\end{equation}
identically.  Thus the computer-assisted part of the proof certifies exact solutions of the original Ou--Cherif system, rather than zeros of an auxiliary or projected equation.
\end{proposition}

\begin{proof}
For the $S^1\times\Sph^2\hookrightarrow\Sph^4$ example, substitute the $O(3)$-invariant metric, shape operator, and mixed field
\[
\mathcal T_3=F\xi+y e_s,
\qquad F=3\lambda^2H,
\qquad y=-\lambda^2\rho',
\]
into \eqref{eq:drifted-Jacobi}.  The normal and meridional components are exactly the $p'$ and $r'$ equations in \eqref{eq:m3-mixed-system}; the remaining equations are the profile and field definitions.

For every equal-multiplicity member $(n,n)$ of \Cref{thm:equal-multiplicity-tower}, and for the unequal product-orbit examples $(1,2)$ and $(2,3)$, the invariant scalar Laplacian, coupling contraction, and tangential rough Laplacian are
\[
\Delta f=f''+Cf',
\qquad
\mathcal C_A(y e_s)=ky'+Sy,
\qquad
\Delta_1(y e_s)=\bigl(y''+Cy'-Dy\bigr)e_s.
\]
Substitution of $\mathcal T_m=F\xi+y e_s$ into the two rows of \eqref{eq:block-m} gives exactly \eqref{eq:product-orbit-system}; its $(1,2)$ specialization is \eqref{eq:T2S2-system}.  This is an algebraic identity in the profile variables and is independent of the interval computation.

For the $O(4)$-rotational $S^1\times\Sph^3\hookrightarrow\Sph^5$ example, direct substitution of
\[
g=\dd s^2+\sin^2\chi\,g_{\Sph^3},
\qquad A=\operatorname{diag}(k,q,q,q),
\qquad V=w e_s
\]
into \eqref{eq:N4}--\eqref{eq:T4-tangent} gives precisely \eqref{eq:O4-system}.  In particular,
\[
\tr(A\circ\nabla^2\rho)=kw'+3qCw,
\qquad
\Ric^M(e_s,e_s)=3+4Hk-k^2,
\]
and these identities reproduce the displayed equations for $P'$ and $Q'$.

In every case the interval computation is used only to prove the existence, closure, positivity, embeddedness, and geometric inequalities of an exact orbit of an algebraically equivalent system.  Equation \eqref{eq:compact-direct-ou} then follows identically.
\end{proof}

\begin{remark}[Scope of the verified portfolio]
Only examples established by exact substitution together with an analytic or interval existence argument are counted as constructions in this paper.  Exploratory spectral near-resonances, fixed-block transverse scans, and candidate branches without a validated nonlinear continuation are not used in any theorem or novelty claim.
\end{remark}

\subsection{The validated local rational-rotation ladder}\label{subsec:rational-ladder}

Remove the fixed angular condition from \eqref{eq:shooting-map} and define the reversible return map
\begin{equation}\label{eq:G-return-map}
G(a)
=
\begin{pmatrix}
\alpha(T;a)-\pi/2\\
P(T;a)\\
w(T;a)
\end{pmatrix},
\qquad
a=(\chi_0,H_0,Q_0,T).
\end{equation}
We use the normalized half-period rotation
\begin{equation}\label{eq:rotation-function}
\Rot(a):=\frac{\vartheta(T;a)}{\pi}.
\end{equation}
Thus the fourth component of the original shooting map is
\begin{equation}\label{eq:F4-rotation}
F_4(a)=2\pi\left(\Rot(a)-\frac13\right).
\end{equation}

The interval integration of the variational equations used in
\Cref{thm:certified-root} also encloses the full shooting Jacobian on $X$.
One retained enclosure is
\begingroup
\small
\setlength{\arraycolsep}{3pt}
\begin{align}
J_-&:=
\begin{pmatrix}
-6.551430&-0.842629&-1.281838&-3.903394\\
-42.576635&-15.195150&-6.640095&13.392696\\
2.754489&2.340551&1.676040&-0.627588\\
-0.917059&-2.297612&0.259174&2.238456
\end{pmatrix},\label{eq:J-minus}\\[1mm]
J_+&:=
\begin{pmatrix}
-6.515398&-0.819335&-1.276193&-3.885840\\
-40.870294&-14.316300&-6.360659&13.921683\\
2.964113&2.448592&1.710037&-0.562627\\
-0.914219&-2.295740&0.259662&2.240385
\end{pmatrix}.
\label{eq:J-plus}
\end{align}
\endgroup
Thus, componentwise,
\begin{equation}\label{eq:DF-enclosure}
DF(X)\subset[J_-,J_+].
\end{equation}
The first three rows are $DG(X)$.  Evaluating the determinant by interval
multilinearity gives
\begin{align}
81.864&<\det D_{(\chi_0,H_0,Q_0)}G<102.884,
\label{eq:G-minor-bound}\\
-638.340&<\det DF<-527.810.
\label{eq:F-det-bound}
\end{align}

\begin{theorem}[Local rational-rotation ladder]\label{thm:local-rational-ladder}
Let $a_*=(\chi_{0,*},H_{0,*},Q_{0,*},T_*)$ be the exact root from
\Cref{thm:certified-root}.  There exist open intervals
\[
T_*\in J\subset\R,
\qquad
\frac13\in I\subset\R,
\]
and a real-analytic map
\[
T\longmapsto
 a(T)=\bigl(\chi_0(T),H_0(T),Q_0(T),T\bigr)
\]
with the following properties.
\begin{enumerate}[label=\textup{(\roman*)},leftmargin=2.2em]
\item The reversible return conditions hold identically:
\[
G(a(T))=0
\qquad (T\in J).
\]
\item The rotation map
\[
\Rot:J\longrightarrow I,
\qquad
\Rot(T):=\Rot(a(T)),
\]
is a real-analytic diffeomorphism.  After shrinking $J$ if necessary,
\begin{equation}\label{eq:rotation-derivative-bound}
-1.242<\Rot'(T)<-0.816
\qquad (T\in J).
\end{equation}
\item For every reduced rational number $p/q\in I$ there is a unique
$T_{p/q}\in J$ such that
\[
\Rot(T_{p/q})=\frac pq.
\]
The associated reversible orbit closes after $q$ reduced periods with total
angular winding $p$, and determines a smooth $O(4)$-invariant proper
biharmonic conformal immersion
\[
\Phi_{p/q}:S^1\times S^3\longrightarrow S^5.
\]
\item Every member of this local rational ladder has positive nonconstant
mean curvature, a positive nonconstant conformal factor, and two distinct
principal curvatures everywhere.
\end{enumerate}
Consequently there are infinitely many exact closed, non-CMC,
nowhere-umbilical proper biharmonic conformal immersions
$S^1\times S^3\to S^5$ whose rotation numbers accumulate at $1/3$.
The member with rotation number $1/3$ is the dimension-four embedded hypersurface of
\Cref{thm:compact-main}.
\end{theorem}

\begin{proof}
The reduced system \eqref{eq:O4-system} is real analytic on
$0<\chi<\pi/2$.  The orbit-wide bounds of
\Cref{prop:closure-geometry} therefore give an open neighborhood of the
certified solution on which the flow, $G$, and $\Rot$ depend real analytically
on the shooting parameters.

Write $u=(\chi_0,H_0,Q_0)$.  By \eqref{eq:G-minor-bound}, $D_uG(a_*)$ is
invertible.  The analytic implicit-function theorem produces an interval
$J\ni T_*$ and a unique real-analytic map $u=u(T)$ such that
\[
G(u(T),T)=0.
\]
We shrink $J$ so that $(u(T),T)\in X$.

Along this implicit branch, the block determinant or Schur-complement
identity gives
\[
\det DF
=
\det D_uG\,
\frac{\dd}{\dd T}F_4(u(T),T).
\]
Using \eqref{eq:F4-rotation},
\begin{equation}\label{eq:schur-rotation}
2\pi\Rot'(T)
=
\frac{\det DF}{\det D_uG}.
\end{equation}
The determinant bounds \eqref{eq:G-minor-bound}--\eqref{eq:F-det-bound}
then imply \eqref{eq:rotation-derivative-bound}.  Hence $\Rot$ is strictly
decreasing and, after a further shrinking of $J$, is an analytic
diffeomorphism onto an open interval $I$ containing $1/3$.

Let $p/q\in I$ be reduced.  There is a unique $T_{p/q}\in J$ with
$\Rot(T_{p/q})=p/q$.  The condition $G=0$ places both endpoints of the
half-orbit in the fixed set of the reverser \eqref{eq:reverser}.  Reversibility
extends the solution to a periodic reduced orbit of period $2T_{p/q}$.  Its
angular advance during one reduced period is
\[
2\vartheta(T_{p/q})=2\pi\frac pq.
\]
After $q$ reduced periods the reduced state returns to its initial value and
the angular coordinate advances by $2\pi p$, so the profile closes.
Moreover $w$ is odd under the reversing involution, and therefore
\[
\int_0^{2T_{p/q}}w(s)\dd s=0.
\]
Thus $\rho$ and $\lambda=e^\rho$ are periodic.

At the certified orbit the interval proof gives the strict inequalities
\[
0<\chi<\frac\pi2,
\qquad H>0,
\qquad \sin\alpha>0,
\qquad \vartheta'>0,
\qquad |k-q|>0,
\]
and also $P'(0)<0$ and $w'(0)=Q_0<0$.  Continuous dependence of the flow on
parameters preserves these strict inequalities after $J$ is shrunk.  Hence
every nearby rational member is a smooth immersion with positive nonconstant
mean curvature, positive nonconstant conformal factor, and no umbilic points.
Since \eqref{eq:O4-system} is equivalent to the critical conformal
biharmonic equations, every member is proper biharmonic conformal.
Finally, every open interval contains infinitely many reduced rationals.
\end{proof}

\begin{lemma}[Orbit-annulus embedding and winding obstruction]
\label{lem:orbit-annulus-embedding}
Consider a closed member of the local $O(4)$-rotational branch satisfying
$0<\chi<\pi/2$ and $\vartheta'>0$.  Let
\[
\gamma(s)=\bigl(\chi(s),e^{i\vartheta(s)}\bigr)
\]
be its closed profile in the rotational orbit annulus
\[
\mathcal A=(0,\pi/2)\times S^1.
\]
Then the rotational hypersurface \eqref{eq:O4-param} is embedded if and only
if $\gamma$ is a simple closed curve.  If the reduced rotation number is the
reduced fraction $p/q>0$, then the degree of the projection
$\gamma\to S^1$ is $p$.  In particular, embeddedness forces $p=1$.
\end{lemma}

\begin{proof}
Because $0<\chi<\pi/2$, both $\sin\chi$ and $\cos\chi$ are nonzero.  Hence
the first two coordinates of \eqref{eq:O4-param} determine the orbit-space
point $(\chi,e^{i\vartheta})$.  If two profile parameters determine the same
orbit-space point, choosing the same $\omega\in S^3$ gives the same point of
the hypersurface, so an embedded hypersurface must have a simple profile.
Conversely, if the profile is simple and
$\Phi(s,\omega)=\Phi(t,\eta)$, the first two coordinates give
$\gamma(s)=\gamma(t)$ and therefore $s=t$ on the profile circle; the last four
coordinates then give $\omega=\eta$.  Thus $\Phi$ is injective.  Since the
domain is compact and the map is already an immersion, it is an embedding.

For rotation $p/q$, the closed profile consists of $q$ reduced periods and
its total angular advance is $2\pi p$.  Therefore the projection to the
$S^1$ factor has degree $p$.  A simple essential closed curve in an annulus
represents a primitive element of
$\pi_1(\mathcal A)\cong\mathbb Z$, so its degree is $\pm1$.  The inequality
$\vartheta'>0$ fixes the sign, and hence $p=1$.
\end{proof}

\begin{corollary}[Local uniqueness of the embedded ladder member]
\label{cor:embedded-member-isolated}
The interval $I$ in \Cref{thm:local-rational-ladder} may be chosen so that
\[
I\subset\left(\frac14,\frac12\right).
\]
Within the resulting local rational-rotation ladder, the rotation-$1/3$
hypersurface is the unique embedded member.
\end{corollary}

\begin{proof}
Let $p/q\in I$ be reduced and suppose that the corresponding ladder member is
embedded.  The strict orbit inequalities in
\Cref{thm:local-rational-ladder} give $0<\chi<\pi/2$ and $\vartheta'>0$, so
\Cref{lem:orbit-annulus-embedding} applies and forces $p=1$.  Thus the
rotation number must be $1/q$.  After shrinking $I$ inside
$(1/4,1/2)$, the only reciprocal integer contained in $I$ is $1/3$.
The certified $1/3$ member is embedded by
\Cref{prop:closure-geometry}; hence it is the unique embedded member of the
local ladder.
\end{proof}

\begin{remark}[Local theorem and global continuation]
\label{rem:local-global-ladder}
\Cref{thm:local-rational-ladder} is a rigorous consequence of the existing
interval certificate and requires no separate Krawczyk computation for each
nearby rational rotation.  Floating-point pseudo-arclength continuation
suggests a stronger global branch joining the proper small-sphere limit at
$\Rot=1/4$ to the Clifford equilibrium at
\[
\Rot=\frac{1}{\sqrt{3+\sqrt{13}}}.
\]
That global connection, the numerical umbilic threshold, and rational members
far from $1/3$ are not used in the theorem and require validated continuation
before being promoted to global statements.
\end{remark}

\section{Geometric consequences of the constructions}\label{sec:synthesis}

The examples can now be read as answers to the construction problem posed in the introduction rather than as unrelated shooting computations; \Cref{tab:construction-synthesis} summarizes the distinct roles of the five families.

\begin{table}[tbp]
\centering
\caption{Geometric role of the certified compact constructions.}\label{tab:construction-synthesis}
\begin{tabular}{p{.25\textwidth}p{.28\textwidth}p{.37\textwidth}}
\toprule
Geometry & Coupling mechanism & Rigorous conclusion \\
\midrule
$S^1\times\Sph^2\subset\Sph^4$
& rotational normal--tangential cancellation
& embedded mean-convex non-CMC proper BCI \\
\addlinespace
$T^2\times\Sph^2\subset\Sph^5$
& lower-symmetry three-curvature coupling
& embedded fixed-sign example with certified curvature separation \\
\addlinespace
$S^1\times\Sph^3\subset\Sph^5$
& reversible rotational return and varying rotation number
& certified $1/3$ orbit and an infinite local rational ladder \\
\addlinespace
$S^1\times\Sph^n\times\Sph^n\subset\Sph^{2n+2}$
& equal-multiplicity product-orbit hybridization
& embedded fixed-sign solution for every $n\geq1$ \\
\addlinespace
$S^1\times\Sph^2\times\Sph^3\subset\Sph^7$
& unequal-multiplicity product-orbit coupling
& embedded fixed-sign non-CMC proper BCI \\
\bottomrule
\end{tabular}
\end{table}

The operator language is essential to the interpretation.  These are not solutions of a modified ``Clifford equation.''  The block is exactly the normal--tangential splitting of the original conformal biharmonic equation, and every reduced orbit is substituted back into that system.  The Clifford realization identifies which representation components are present and how their couplings act; it does not weaken the geometric problem.

The surface no-go theorems and the compact examples form two halves of one statement.  In the scalar channel, coercivity can remove the kernel, the constraint $U=\lambda^2H$ can make a zero crossing geometrically sterile, and first-order compatibility can freeze the conformal factor.  In the mixed channel, the normal and tangential components are independent, and the off-diagonal terms allow neither diagonal equation to vanish separately.  The explicit initial-point estimates in the certified examples verify that this cancellation is nontrivial.

The equal-multiplicity theorem shows that the mechanism is stable under increasing dimension rather than tied to a special low-dimensional coincidence.  The lower-symmetry and unequal-multiplicity examples show that it also survives when the orbit curvatures split.  Finally, the rotation-number result shows that one certified closed orbit belongs to a local analytic family, giving infinitely many exact rational closures rather than a lone isolated solution.

Taken together, the two analyses locate the obstruction precisely.  The scalar--chiral classification determines the reach of the smallest factorization class and rules out its non-CMC branch.  The normal--tangential analysis identifies the coefficient bundle omitted by that scalar ansatz and shows that its off-diagonal coupling supports compact non-CMC geometry.  In this sense the rigidity result is sharp with respect to coefficient bundle: the first enlargement already forced by conformal tension is sufficient.

\section{Conclusion and outlook}

The starting question was whether the failure of the scalar Dirac channel reflects a genuine geometric nonexistence phenomenon or only a limitation of the coefficient bundle.  The answer is now precise.  For hypersurfaces of dimension greater than two, the conformal tension field itself contains independent normal and tangential data.  The bundle $\mathbb R\xi\oplus TM$ is therefore the minimal natural carrier of the problem, and the conformal biharmonic equation becomes a strongly elliptic normal--tangential Jacobi block with an exact two-grade Clifford realization.

This reformulation reveals a mechanism hidden in the component equations.  The diagonal channels reproduce the scalar obstructions seen on surfaces, while the off-diagonal terms permit a hybridized zero mode in which both parts of the tension are nonzero.  Exact cohomogeneity-one reductions turn this mechanism into reversible boundary-value problems, and validated numerics prove that the required closed orbits exist.  The resulting hypersurfaces are embedded, proper, non-CMC, and nowhere umbilical, with an infinite family
\[
S^1\times\Sph^n\times\Sph^n
\hookrightarrow\Sph^{2n+2},
\qquad n\geq1,
\]
as the clearest indication that the construction is structurally stable.

Several problems remain.  The mixed block has not yet been factored into a complete first-order matrix Dirac system; doing so would require a matrix Riccati theory capable of retaining the geometric line constraint.  The local rational-rotation ladder should be extended by validated global continuation to determine its endpoints and possible curvature-collision thresholds.  It is also natural to ask which other cohomogeneity-one actions, ambient geometries, or higher-codimension coefficient bundles support the same normal--tangential cancellation.  Finally, the analytic and Morse-theoretic properties of the certified solutions---including their index, stability, and behavior under geometric flows---remain open.

The main conceptual conclusion is nevertheless complete.  The first half of the paper closes the scalar channel; the second shows that conformal geometry itself supplies the minimal channel that replaces it, and that this mixed channel is large enough to create compact nontrivial geometry while remaining exactly faithful to the original biharmonic map equation.

\section*{Statements and Declarations}
\noindent\textbf{Competing interests.}
The author declares no competing interests.

\smallskip
\noindent\textbf{Data and code availability.}
The computer-assisted certificates, source code, retained outputs, and plotting data supporting the validated numerical results are supplied with the article as Online Resource~1.  The analytic parts of the paper use no external datasets.

\smallskip
\noindent\textbf{Use of generative AI.}
During preparation of the manuscript, the author used OpenAI ChatGPT as an auxiliary tool for language editing, code drafting, and exploratory consistency checks.  The author independently reviewed and verified the mathematical arguments, computations, citations, and conclusions and takes responsibility for the final manuscript.

\appendix

\section{Exact Bernstein certificates for the local torsion}\label{app:bernstein-certificate}

We record the finite polynomial certificates used in \Cref{thm:finite-type-spherical}.  The use of Bernstein coefficients to turn positivity on intervals or simplices into finite algebraic certificates is standard; see, for example, \cite{BoudaoudCarusoRoy}.  Put
\begin{equation}\label{eq:xi-upsilon-simplex}
\xi=\frac{8X_0}{9},
\qquad
\upsilon=\frac{4Y_0}{9}.
\end{equation}
Then the closed discriminant region is the standard simplex
\[
\Sigma=\bigl\{(\xi,\upsilon):\ \xi\ge0,\ \upsilon\ge0,
\ \xi+\upsilon\le1\bigr\}.
\]
After the substitution
\[
X_0=\frac98\xi,
\qquad
Y_0=\frac94\upsilon,
\]
the six polynomials requiring certification are
\begin{equation}\label{eq:six-certificate-polynomials}
N,
\quad
N^2-32D(16X_0U^2+Y_0V^2),
\quad
P_C,
\quad
-F_2,
\quad
F_3,
\quad
-F_4,
\end{equation}
where all symbols are displayed explicitly in
\eqref{eq:U-local}--\eqref{eq:N-local},
\eqref{eq:PC-local}, and
\eqref{eq:F2-local}--\eqref{eq:F4-local}.

For a polynomial $P$ of total degree $n$, its triangular Bernstein expansion is
\begin{equation}\label{eq:triangular-Bernstein}
P(\xi,\upsilon)
=\sum_{i+j+k=n}b_{ijk}
\frac{n!}{i!j!k!}
\xi^i\upsilon^j(1-\xi-\upsilon)^k.
\end{equation}
If every $b_{ijk}$ is positive, then $P>0$ on $\Sigma$.  All coefficients below are exact rational numbers.

The amplitude polynomial
\[
\mathcal P_{\mathrm{amp}}
=N^2-32D(16X_0U^2+Y_0V^2)
\]
requires one localized midpoint refinement.  Label the vertices of $\Sigma$ by
\[
A=(0,0),
\qquad
B=(1,0),
\qquad
C=(0,1),
\]
and let $M_{AB},M_{AC},M_{BC}$ be their midpoints.  The four children are
\begin{align*}
0&=(A,M_{AB},M_{AC}),
&1&=(M_{AB},B,M_{BC}),\\
2&=(M_{AC},M_{BC},C),
&3&=(M_{AB},M_{BC},M_{AC}).
\end{align*}
Only child $0$ needs subdivision.  Applying the same labeling inside it yields the seven terminal cells
\[
1,
\quad2,
\quad3,
\quad00,
\quad01,
\quad02,
\quad03.
\]
The least coefficient on each terminal cell is recorded in
\Cref{tab:bernstein-amplitude}.

\begin{table}[ht]
\centering
\begin{tabular}{@{}cr@{}}
\toprule
Cell & least Bernstein coefficient\\
\midrule
$1$  & $397305729/32$\\
$2$  & $800442$\\
$3$  & $276933249/160$\\
$00$ & $8698768443/1280$\\
$01$ & $431300457/64$\\
$02$ & $153533961/64$\\
$03$ & $97875338067/20480$\\
\bottomrule
\end{tabular}
\caption{Exact terminal-cell certificate for $\mathcal P_{\mathrm{amp}}$.}
\label{tab:bernstein-amplitude}
\end{table}

The remaining five polynomials are positive on the original simplex without subdivision.
\begin{table}[ht]
\centering
\begin{tabular}{@{}lrrr@{}}
\toprule
Polynomial & degree & coefficient count & least coefficient\\
\midrule
$N$      & $3$ & $10$ & $108$\\
$P_C$    & $3$ & $10$ & $5832$\\
$-F_2$   & $2$ & $6$  & $810$\\
$F_3$    & $3$ & $10$ & $5832$\\
$-F_4$   & $4$ & $15$ & $104976$\\
\bottomrule
\end{tabular}
\caption{Exact root-simplex Bernstein certificates.}
\label{tab:bernstein-root}
\end{table}

The ancillary script \path{verify_bernstein_certificates.py} performs the following operations over $\mathbb Q$:
\begin{enumerate}[label=\textup{(\roman*)},leftmargin=2.2em]
\item defines the six polynomials in \eqref{eq:six-certificate-polynomials};
\item carries out the normalization \eqref{eq:xi-upsilon-simplex};
\item converts the power basis to the degree-preserving triangular Bernstein basis \eqref{eq:triangular-Bernstein};
\item applies the stated affine midpoint subdivisions to $\mathcal P_{\mathrm{amp}}$;
\item asserts strict positivity and prints every coefficient.
\end{enumerate}
The complete machine-readable output is supplied as
\path{bernstein_certificate_full.txt}.  Thus the sign portion of the proof can be checked independently without reconstructing the differential-geometric reduction.

As a direct consistency check, at the center $x=y=0$ the torsion reduces to
\begin{equation}\label{eq:center-torsion-check}
\mathcal C
=-\frac{\sqrt2}{6912}
\left(32256v^2+81\rho^2+25632\rho+11008\right)<0,\qquad \rho>0.
\end{equation}

\section{Adapted-tetrad derivation of the weighted hypersurface equations}\label{app:tetrad}

We record the moving-frame calculation in arbitrary hypersurface dimension.  Let
\[
\phi:M^m\longrightarrow (N^{m+1},h)
\]
be an oriented isometric hypersurface immersion, with unit normal $\xi$, shape operator $A$, and normalized mean curvature
\[
H=\frac1m\tr A.
\]
Choose an adapted orthonormal frame
\[
(e_1,\ldots,e_m,e_{m+1}=\xi)
\]
with dual coframe $(\theta^1,\ldots,\theta^{m+1})$.  Along $M$,
\begin{equation}\label{eq:Cartan-hypersurface}
\theta^{m+1}=0,
\qquad
\omega^{m+1}{}_a=h_{ab}\theta^b,
\end{equation}
where $h_{ab}=g(Ae_a,e_b)$.  The tension field is
\begin{equation}\label{eq:tension-hypersurface}
\tau(\phi)=mH\xi.
\end{equation}

At a fixed point $p$, choose the tangent frame geodesic.  Then
\begin{equation}\label{eq:first-derivative-Hxi}
\nabla^\phi_{e_a}(H\xi)=e_a(H)\xi-HA(e_a).
\end{equation}
Tracing one further derivative and splitting into normal and tangential components gives
\begin{align}
\left[\tr_g\nabla^\phi\nabla^\phi(H\xi)\right]^\perp
&=(\Delta H-H|A|^2)\xi,\label{eq:rough-normal}\\
\left[\tr_g\nabla^\phi\nabla^\phi(H\xi)\right]^\top
&=-2A(\grad H)-H\diver A.\label{eq:rough-tangent}
\end{align}
The contracted Codazzi equation is
\begin{equation}\label{eq:contracted-Codazzi}
\diver A=m\grad H-(\Ric^N(\xi))^\top.
\end{equation}
The curvature trace splits as
\begin{equation}\label{eq:curvature-trace}
\sum_{a=1}^mR^N(e_a,\xi)e_a
=-(\Ric^N(\xi))^\top-\Ric^N(\xi,\xi)\xi
\end{equation}
for the curvature convention used by Ou.  Substitution into the bitension field yields the familiar system
\begin{align}
\Delta H-H|A|^2+H\Ric^N(\xi,\xi)&=0,\label{eq:Ou-hyp-normal}\\
2A(\grad H)+mH\grad H-2H(\Ric^N(\xi))^\top&=0.\label{eq:Ou-hyp-tangent}
\end{align}
Equivalently, the second equation can be written
\[
2A(\grad H)+\frac m2\grad(H^2)-2H(\Ric^N(\xi))^\top=0.
\]

Let $f>0$.  The $f$-bitension field is
\begin{equation}\label{eq:f-bitension}
\tau_{2,f}(\phi)
=f\tau_2(\phi)+(\Delta f)\tau(\phi)
+2\nabla^\phi_{\grad f}\tau(\phi).
\end{equation}
Using \eqref{eq:first-derivative-Hxi}, its normal component combines by the scalar product rule into
\begin{equation}\label{eq:weighted-normal-app}
\Delta(fH)-(fH)\bigl(|A|^2-\Ric^N(\xi,\xi)\bigr)=0.
\end{equation}
Its tangential component combines into
\begin{equation}\label{eq:weighted-tangent-app}
A(\grad(fH))+(fH)
\left[\frac m2\grad H-(\Ric^N(\xi))^\top\right]=0.
\end{equation}
Equations \eqref{eq:weighted-normal-app}--\eqref{eq:weighted-tangent-app} are Ou's improved $f$-biharmonic hypersurface equations \cite{OuFII}.  For $m=2$, $f=\lambda^2$, and a three-dimensional space form, they reduce exactly to
\[
\Delta U-U(|A|^2-2c)=0,
\qquad
A(\grad U)+U\grad H=0,
\qquad U=\lambda^2H,
\]
which are the surface equations used in \eqref{eq:Ou-normal}--\eqref{eq:Ou-tangent}.

\begin{remark}
The frame calculation shows why the two equations are structurally unavoidable: the normal tetrad component contains the scalar Jacobi potential, whereas the tangent tetrad components encode the divergence of the second fundamental form through contracted Codazzi.  The spinorial construction recombines these tetrad components as the scalar and vector pieces of one Clifford-module equation.
\end{remark}

\section{Reproducibility of the interval certificates}
\label{sec:reproducibility}

The files described in this appendix are supplied as Online Resource~1, a reproducibility archive organized by theorem rather than by the chronology of the computations.  From the archive root,
\begin{verbatim}
./verify_all.sh
\end{verbatim}
checks the SHA-256 inventory, verifies every retained success marker, and checks the Python and C++ sources.  The more expensive command
\begin{verbatim}
./verify_all.sh --recompute
\end{verbatim}
rebuilds and reruns the interval certificates.  Every runner resolves its own directory, so no command depends on the caller's working directory or on files hidden inside a nested archive.  The compilation flags include
\begin{verbatim}
-frounding-math -fno-fast-math
\end{verbatim}
for the Boost.Interval calculations.

\subsection{The critical \texorpdfstring{$m=4$}{m=4} single-shooting certificate}
The $O(4)$-rotational verifier is located in
\path{certificates/s1_s3_sphere5/}.  It uses the parameter vector
\(
(
\chi_0,H_0,Q_0,T
)
\)
and the shooting map \eqref{eq:shooting-map}.  It can be reproduced with
\begin{verbatim}
cd certificates/s1_s3_sphere5
./run_certificate.sh
\end{verbatim}
A successful run prints \texttt{KRAWCZYK\_INCLUDED=YES} and reproduces the orbit-wide inequalities used in \Cref{thm:certified-root,prop:closure-geometry}.  The same variational calculation supplies the Jacobian bounds used in \Cref{thm:local-rational-ladder}.

\subsection{The \texorpdfstring{$m=3$}{m=3} multiple-shooting certificate}
The $O(3)$-invariant $S^1\times\Sph^2\hookrightarrow\Sph^4$ verifier is located in
\path{certificates/s1_s2_sphere4/}.  It propagates sixteen multiple-shooting segments and the first-variation equations in the mixed variables of \eqref{eq:m3-mixed-system}.  Run
\begin{verbatim}
cd certificates/s1_s2_sphere4
./run_certificate.sh
\end{verbatim}
The retained calculation satisfies
\[
\max_i\frac{\operatorname{rad}K_i(X)}{\operatorname{rad}X_i}
=0.0011427481183545751<1.
\]

\subsection{The lower-symmetry \texorpdfstring{$m=4$}{m=4} multiple-shooting certificate}
The $T^2\times\Sph^2\hookrightarrow\Sph^5$ package is located in
\path{certificates/t2_s2_sphere5/}.  Its 180-dimensional multiple-shooting Krawczyk calculation can be regenerated by
\begin{verbatim}
cd certificates/t2_s2_sphere5
./run_certificate.sh
\end{verbatim}
The script regenerates the nodes, center residual, inverse preconditioner, interval Jacobian, and final Krawczyk output.  It exits unsuccessfully unless the retained success condition is reproduced.

\subsection{Unified fixed-multiplicity certificates for \texorpdfstring{$1\leq n\leq20$}{1 <= n <= 20}}
The product-orbit cases
\[
T^3\hookrightarrow\Sph^4,
\qquad
S^1\times\Sph^2\times\Sph^2\hookrightarrow\Sph^6,
\qquad
S^1\times\Sph^3\times\Sph^3\hookrightarrow\Sph^8.
\]
Two conditioning-compatible verifiers reproduce these cases and every equal-multiplicity integer through $n=20$.  They are stored in the directory
\begin{center}
\path{certificates/equal_tower_finite/}.
\end{center}
The original-variable verifier covers $n=1,\ldots,5$.  The desingularized verifier in $(A,B,\phi,R)$ covers $n=6,\ldots,20$.  The command
\begin{verbatim}
cd certificates/equal_tower_finite
./run_equal_tower_n1_n20.sh
\end{verbatim}
reads the retained centers from
\path{equal_tower_roots_n1_n20.csv} and
\path{scaled_roots_n5_n20.csv}.  Every retained output contains
\begin{verbatim}
KRAWCZYK SUCCESS: true
CERTIFICATE SUCCESS: true
\end{verbatim}
For $n\leq5$ the original-variable output additionally prints
\verb|SIGN SWEEP SUCCESS: true|.  For $n\geq6$, the scaled verifier incorporates the geometric sign inequalities into the final certificate Boolean and prints their individual margins under \verb|Geometry signs (scaled)|.  The retained $n=4$ certificate is the corrected successful run with root-box radius $2.1\times10^{-8}$; the earlier failed exploratory output is not part of the proof archive.  Likewise, failed unscaled exploratory outputs at larger $n$ have been replaced by the successful desingularized certificates.

\subsection{Tensor Chebyshev/radii-polynomial tail package}
The compactified tail calculation is contained in
\path{certificates/equal_tower_tail/}.  Its location-independent runner is
\begin{verbatim}
cd certificates/equal_tower_tail
./run_equal_tower_tail_certificate.sh
\end{verbatim}
The limiting-root Krawczyk calculation and the final scalar radii inequality are evaluated with directed-rounding Boost intervals.  The local tensor range layer is independently certified by the formal binary64 audit described in the preceding remark: exact rational restriction matrices are combined with $\gamma_k$ bounds for every floating-point reduction \cite{Higham}, and the resulting raw enclosures are checked against the rounded constants entering the radii polynomial.  The package retains the state-defect, variational, derivative, transition, sign-sweep, and final audit outputs together with a SHA-256 manifest.

\subsection{The unequal \texorpdfstring{$(2,3)$}{(2,3)} half-return certificate}
The proof of
\(
S^1\times\Sph^2\times\Sph^3\hookrightarrow\Sph^7
\)
is located in \path{certificates/unequal_2_3/}.  It is reproduced by
\begin{verbatim}
cd certificates/unequal_2_3
./run_certificate.sh
\end{verbatim}
The verifier uses the exact product-orbit system \eqref{eq:product-orbit-system}, a four-dimensional Krawczyk inclusion for the half-return map \eqref{eq:product-half-map}, and an endpoint-aware $2048$-tube sign cover.  Near the initial reversing leaf it proves $z'>0$ and $N_\phi'<0$; in the middle it proves $z>0$ and $N_\phi<0$; near the terminal reversing leaf it proves $z'<0$ and $N_\phi'>0$.  Together with the exact endpoint equations this proves that the upper half-profile is a strict angular graph.

\subsection{Floating-point visualization data}
Online Resource~1 also contains the scripts and sampled data used for the figures.  The script \path{generate_product_orbit_figures.py} integrates the reduced product-orbit system at the centers of the certified root boxes and writes
\begin{itemize}
\item \path{t3_floating_orbit_v1.csv};
\item \path{s1s2s2_floating_orbit_v1.csv};
\item \path{s1s3s3_floating_orbit_v1.csv};
\item \path{s1s2s3_floating_orbit_v1.csv}.
\end{itemize}
The file \path{equal_tower_scaled_branch.csv} supplies the floating-point branch data for the convergence panel in \Cref{fig:equal-tower-limit}; \path{generate_equal_tower_convergence.py} regenerates that panel.  The script \path{generate_bci_figures.py} regenerates the rotational and lower-symmetry visualization panels.  These files are for visualization only and are not substituted for interval enclosures in any proof.

\paragraph{Online Resource 1.}
Reproducibility archive containing the exact symbolic checks, interval-certificate source code and retained outputs, fixed- and large-multiplicity verification data, visualization scripts, and sampled plotting data used in the manuscript.

\begin{figure}[tbp]
\centering
\begin{minipage}[c]{0.485\textwidth}
  \centering
  \safeincludegraphics[width=\linewidth]{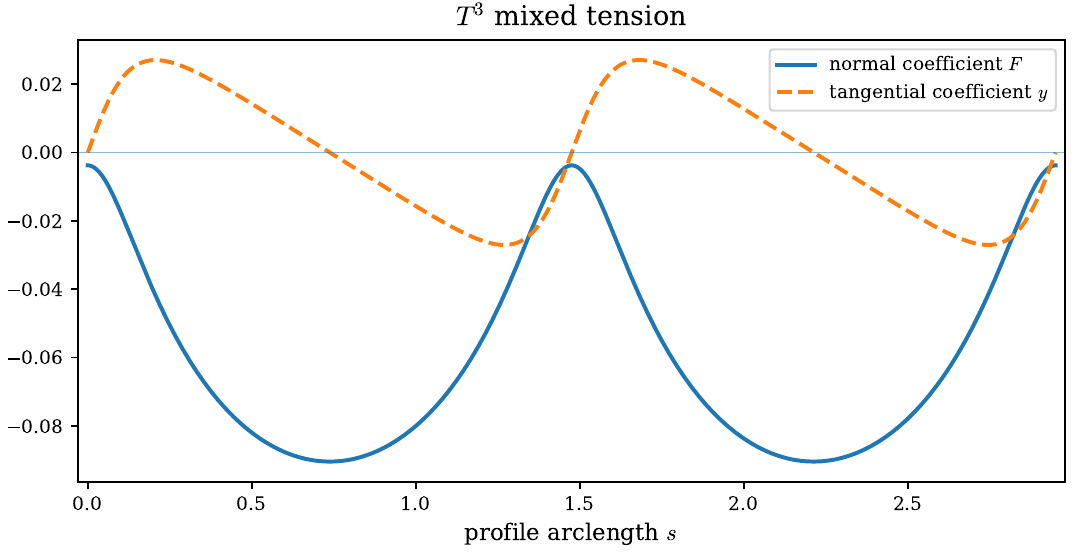}
\end{minipage}\hfill
\begin{minipage}[c]{0.485\textwidth}
  \centering
  \safeincludegraphics[width=\linewidth]{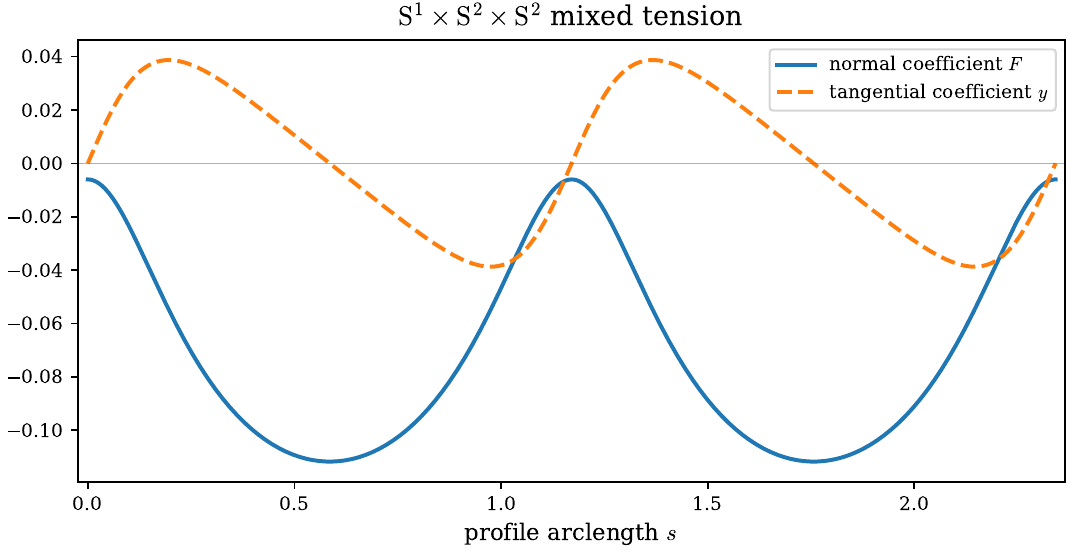}
\end{minipage}

\medskip
\begin{minipage}[c]{0.485\textwidth}
  \centering
  \safeincludegraphics[width=\linewidth]{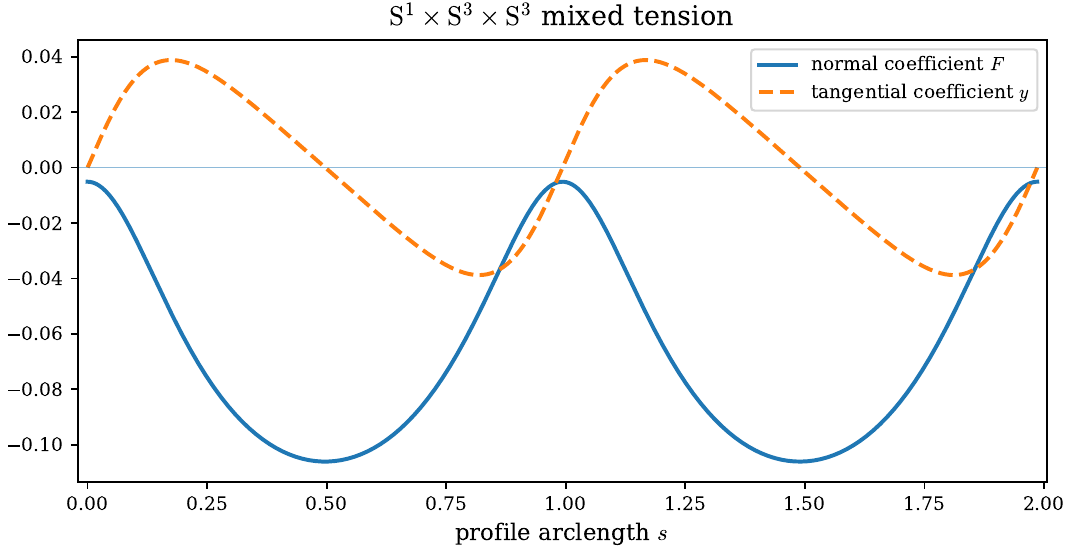}
\end{minipage}\hfill
\begin{minipage}[c]{0.485\textwidth}
  \centering
  \safeincludegraphics[width=\linewidth]{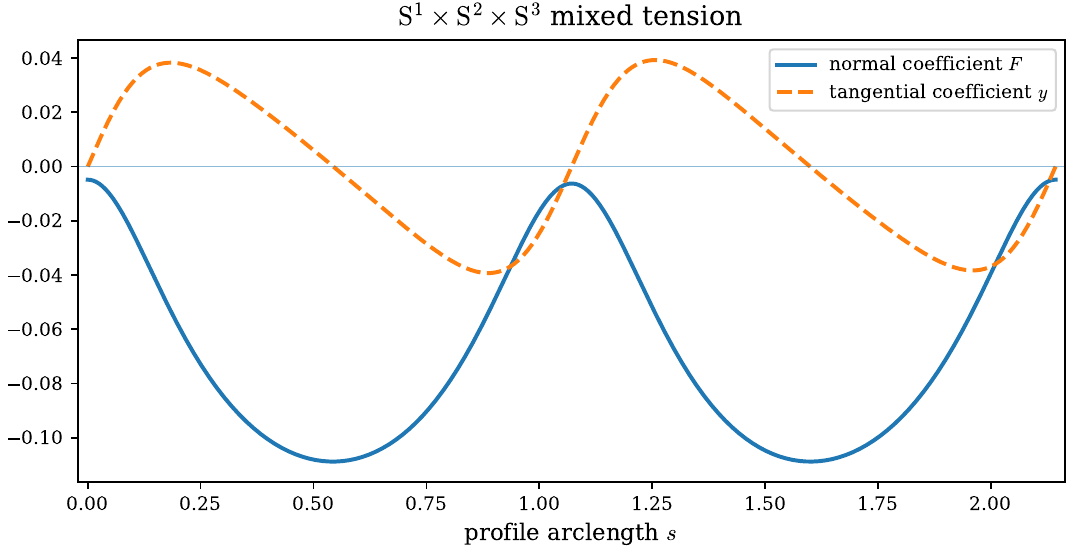}
\end{minipage}
\caption{Floating-point traces of the normal and tangential coefficients $(F,y)$ of the mixed conformal-tension field.  Top left: $(\mu,\nu)=(1,1)$.  Top right: $(2,2)$.  Bottom left: $(3,3)$.  Bottom right: $(2,3)$.  In every case both channels are nontrivial, making the normal--tangential hybridization behind the certified compact examples visible.}
\label{fig:product-hybridization-traces}
\end{figure}

\section{A compact derivation of the \texorpdfstring{$m=3$}{m=3} rotational system}

For the warped metric
\[
g=\dd s^2+\sin^2\chi\,g_{\Sph^2},
\]
the invariant mixed field is
\[
\mathcal T_3=F\xi+y e_s,
\qquad
F=3LH,
\qquad
 y=-L\rho'.
\]
For an invariant scalar $f$ and invariant tangent field $Y=y e_s$, one has
\[
\Delta f=f''+2Cf',
\qquad
\mathcal C_A(Y)=kr+2qCy,
\]
and the tangential rough Laplacian reduces to
\[
\Delta_1(y e_s)=\bigl(y''+2Cy'-2C^2y\bigr)e_s.
\]
Substitution into the $m=3$ block \eqref{eq:block-m}, followed by the identities
\[
L'=-2y,
\qquad
w=-\frac yL,
\qquad
H=\frac{F}{3L},
\qquad
H'=\frac{pL+2Fy}{3L^2},
\]
gives exactly the two equations for $p'$ and $r'$ in \eqref{eq:m3-mixed-system}.  This verifies directly that the validated rotational ODE is equivalent to the full conformal biharmonic hypersurface equation, rather than to a reduced scalar condition.

\section{A compact derivation of the \texorpdfstring{$O(2)\times O(3)$}{O(2) x O(3)} system}

For
\[
g=\dd s^2+a^2\dd\theta^2+b^2g_{\Sph^2},
\qquad
A_1=\frac{a'}a,
\qquad
A_2=\frac{b'}b,
\qquad
C=A_1+2A_2,
\]
the invariant scalar Laplacian is
\[
\Delta f=f''+Cf'.
\]
If $Y=y e_s$, then
\[
\mathcal C_A(Y)=ky'+(q_1A_1+2q_2A_2)y
\]
and
\[
\Delta_1Y
=\bigl(y''+Cy'-(A_1^2+2A_2^2)y\bigr)e_s.
\]
The Gauss equation gives
\[
\Ric^M(e_s,e_s)
=3+4Hk-k^2.
\]
Substitute the mixed field
\[
\mathcal T_4=F\xi+y e_s,
\qquad
F=4LH,
\qquad
y=-L'
\]
into \eqref{eq:block-four}.  Using
\[
k=\frac FL-q_1-2q_2,
\qquad
F'=p,
\qquad
y'=r,
\]
the normal and tangential components reduce exactly to the $p'$ and $r'$ equations in \eqref{eq:T2S2-system}.  This verifies directly that the interval-validated lower-symmetry ODE is the full critical conformal biharmonic system rather than a scalar or partially constrained reduction.

\section{A compact derivation of the rotational Ricci term}

For the warped metric
\[
g=\dd s^2+\sin^2\chi\,g_{\Sph^3},
\]
the Gauss equation for a hypersurface in the unit sphere gives
\[
\Ric^M(e_s,e_s)
=3+4Hk-k^2.
\]
Since $V=w e_s$, the tangential equation \eqref{eq:T4-tangent} reduces to
\[
4(k+2H)P+2(3+4Hk-k^2)w+Q'+2Qw+2ww'=0.
\]
Using $w'=Q-3Cw$ gives the $Q'$ equation in \eqref{eq:O4-system}.  The normal equation follows from
\[
\Delta H=P'+3CP,
\]
\[
\tr(A\circ\nabla^2\rho)=kw'+3qCw,
\]
and
\[
\langle AV,V\rangle=kw^2.
\]
This verifies directly that the rotational system is equivalent to the original critical Ou--Cherif equations.

\end{document}